\documentclass[final, 3p, times]{elsarticle}%

\usepackage{graphicx}
\usepackage{float}
\usepackage[section]{placeins}
\usepackage{caption}


\usepackage{xcolor}

\usepackage{amsmath}
\usepackage{amssymb}
\usepackage{amsthm}
\usepackage{mathtools}

\usepackage{enumitem}

\usepackage[
  colorlinks=true,
  linkcolor=blue,
  citecolor=blue,
  urlcolor=blue
]{hyperref}
\usepackage{cleveref}

\makeatletter
\@ifundefined{tmptocnumberline}{%
  \newdimen\appnamewidth
  \def\tmptocnumberline#1{%
    \setbox0=\hbox{\appendixname}%
    \appnamewidth=\wd0
    \addtolength\appnamewidth{2.5pc}%
    \hb@xt@\appnamewidth{#1\hfill}%
  }%
}{}
\makeatother

\newcommand{\T}{\mathbb T}
\newcommand{\eps}{\varepsilon}
\newcommand{\Reps}{\mathcal R_\varepsilon}
\newcommand{\Aepsp}{\mathcal A_\varepsilon^+}
\newcommand{\Aeps}{\Aepsp}
\newcommand{\Aepsm}{\mathcal A_\varepsilon^-}

\newcommand{\norm}[1]{\left\lVert #1\right\rVert}

\newcommand{\dt}{\tau}

\newtheorem{theorem}{Theorem}[section]
\newtheorem{lemma}[theorem]{Lemma}
\newtheorem{proposition}[theorem]{Proposition}
\newtheorem{corollary}[theorem]{Corollary}
\newtheorem{remark}[theorem]{Remark}
\newtheorem{example}[theorem]{Example}

\biboptions{compress}

\journal{}

\begin{document}
\raggedbottom
\begin{frontmatter}
\title{Pointwise Monotonicity of the Allen--Cahn Flow and Dynamical Limitations of Energy-Stable Schemes}

\tnotetext[label1]{The research of Dongling Wang is supported in part by the National Natural Science Foundation of China under grants 12271463,
and the Natural Science Foundation of Hunan Province under Grant 2026JJ50363.
The work of Pansheng Li is supported by the Postgraduate Scientific Research Innovation Project of Hunan Province. \\ Declarations of interest: none.}

\author[XTU]{Pansheng Li}
\ead{lpsmath@smail.xtu.edu.cn; ORCID 0009-0004-5289-9190}
\author[XTU]{Dongling Wang\corref{mycorrespondingauthor}}
\ead{wdymath@xtu.edu.cn; ORCID 0000-0001-8509-2837}
\cortext[mycorrespondingauthor]{Corresponding author. }

\address[XTU]{Hunan Research Center of the Basic Discipline Fundamental Algorithmic Theory and Novel Computational Methods, National Center for Applied Mathematics in Hunan, School of Mathematics and Computational Science, Xiangtan University, Xiangtan 411105, Hunan, China}

\begin{abstract}

Energy stability is fundamental in the numerical approximation of phase-field models, but it does not by itself guarantee faithful reproduction of local dynamics. Recent ODE-level studies have shown that energy dissipation alone does not ensure dynamical fidelity for spatially homogeneous phase-field models (Xu and Xu, 2023; Li and Wang, 2026). However, these ODE arguments do not extend directly to spatially inhomogeneous Allen--Cahn solutions: diffusion enters the full residual $\mathcal R_\varepsilon(u):=\Delta u+\varepsilon^{-2}(u-u^3)$ and may change its pointwise sign even when $0<u\leq1$.

Motivated by this distinction, we introduce sign-dependent admissible classes determined by the full Allen--Cahn residual and prove that the exact PDE flow preserves the corresponding pointwise monotone-growth or monotone-decay direction. The fully implicit Euler method inherits this structure in its standard unique-solvability regime. We then examine several widely used energy-stable schemes. First-order convex splitting and sufficiently strong first-order stabilization can preserve the correct direction, but on a reduced effective time scale, causing artificial delay or damping. By contrast, second-order convex splitting, stabilized CN/AB, IEQ, and SAV schemes admit monotone data for which large time steps generate wrong-signed pointwise increments despite dissipation of the original or a modified energy. For all four schemes, these strict reversals persist under sufficiently small smooth spatially nonhomogeneous admissible perturbations, so the counterexamples remain valid with genuinely active diffusion. Numerical experiments confirm this classification. Thus, pointwise monotonicity is a local dynamical criterion complementary to energy stability and maximum-bound preservation.

\end{abstract}


\end{frontmatter}


\section{Introduction}

Phase-field models provide a fundamental framework for describing phase transitions and interfacial dynamics in materials science~\cite{allen1979microscopic,cahn1958free,hawkins2012numerical}. The Allen--Cahn equation is a central prototype in this class of models. In this paper, we consider the following normalized flow:
\begin{equation}
\left\{
\begin{array}{ll}
    u_t=\Delta u-\dfrac{1}{\eps^2}f(u), & (x,t)\in\Omega\times(0,T],\\[1.2ex]
    u(x,0)=u_0(x), & x\in\Omega,
\end{array}
\right.
\label{eq:AC}
\end{equation}
under periodic or homogeneous Neumann boundary conditions, where $\Omega\subset \mathbb{R}^d$ ($d\geqslant1$) is a smooth bounded domain (or a periodic cell), the parameter $\varepsilon>0$ represents the interfacial width, and
\[
    f(u)=F'(u)=u^3-u,\qquad F(u)=\frac14(u^2-1)^2.
\]
The Allen--Cahn equation \eqref{eq:AC} can be viewed as the $L^2$ gradient flow of the energy functional
\begin{equation}
    E(u)=\int_{\Omega}\left(\frac12|\nabla u|^2+\frac1{\eps^2}F(u)\right)\,dx,
    \label{eq:AC-energy}
\end{equation}
and thus satisfies the energy dissipation law
\begin{equation}
    \frac{d}{dt}E(u)=
    \left(
    \frac{\delta E(u)}{\delta u},
    \frac{\partial u}{\partial t}
    \right)=
    -\|\partial_t u\|_{L^2(\Omega)}^2
    \le 0,
    \qquad \forall\, t>0,
    \label{eq:energy-dis}
\end{equation}
where $(\cdot,\cdot)$ denotes the standard $L^2$ inner product.
This Lyapunov structure is the basic stability mechanism behind many numerical methods for Allen--Cahn and related phase-field equations~\cite{du1991numerical,elliott1993global,feng2003numerical,zhang2009numerical,feng2015analysis}.

A large body of work has been devoted to constructing numerical schemes that preserve a discrete counterpart of \eqref{eq:energy-dis}. Classical time discretizations include explicit Euler, implicit Euler, Crank--Nicolson, modified Crank--Nicolson, and implicit Runge--Kutta type methods~\cite{crank1947practical,jeong2016comparison,condette2011spectral,shen2010numerical,hou2017numerical,zhang2021numerical,zhang2021preserving}. To obtain efficient and robust schemes for stiff phase-field dynamics, several partially implicit energy-stable techniques have been developed, including convex splitting methods~\cite{eyre1998unconditionally,tangqiao2020phasefield,graser2013time,guan2014second,guillen2014second}, stabilized semi-implicit methods~\cite{xu2006stability,he2007large,feng2013stabilized,tangyang2016imex}, invariant energy quadratization (IEQ) methods~\cite{yang2016linear,yang2020convergence}, and scalar auxiliary variable (SAV) methods~\cite{shen2018scalar,shen2019new}. These schemes are designed so that either the original energy or a modified energy is nonincreasing. Such a property is indispensable for long-time computation. However, energy dissipation is a global Lyapunov-type constraint and does not necessarily determine the local pointwise direction of the dynamics.

Indeed, the Allen--Cahn equation also has pointwise and order-preserving structures that are not contained in the single scalar inequality \eqref{eq:energy-dis}. The maximum principle, comparison principle, metastability, and sharp-interface asymptotics have been extensively studied in the PDE literature~\cite{evans1992phase,pego1989front,ilmanen1993convergence,alikakos1994convergence,chen1994spectrum,chen1996global,caginalp1998convergence}. These results indicate that a reliable time discretization should not merely dissipate energy; it should also reproduce the correct direction of motion, transient time scale, and pointwise behavior whenever these structures are present in the continuous dynamics.

This issue is already apparent in the spatially homogeneous setting.
Under periodic or homogeneous Neumann boundary conditions, spatially
homogeneous initial data give rise to spatially homogeneous solutions.
More precisely, if $u(x,0)\equiv u_0$, where $u_0$ is a constant, then
the corresponding solution remains independent of the spatial variable,
namely, $u(x,t)=u(t)$. A precise statement and a proof based on the
Duhamel principle are given in Theorem~\ref{thm:homogeneous-reduction} of
Appendix A. Hence, if
$u_0(x)\equiv u_0$, the Allen--Cahn PDE \eqref{eq:AC} reduces exactly to the scalar ODE
\begin{equation}
    u_t+\frac1{\eps^2}f(u)=0,
    \qquad u(0)=u_0.
    \label{eq:AC-ODE}
\end{equation}
Its exact solution is
\[
    u(t)=\frac{u_0}
    {\sqrt{e^{-2t/\eps^2}+u_0^2\bigl(1-e^{-2t/\eps^2}\bigr)}}.
\]
Hence, if $0<u_0<1$, the solution increases monotonically to $+1$, while if $-1<u_0<0$, it decreases monotonically to $-1$. On this invariant class, the PDE and ODE trajectories, equilibria, and monotone directions coincide exactly. This correspondence is important for the later analysis: if a time discretization loses monotonicity for spatially homogeneous data, then the same construction is already a genuine counterexample for the corresponding PDE scheme, rather than a separate ODE phenomenon. Recent ODE-level studies have shown that several energy-stable or partially implicit schemes may exhibit precisely such defects for large time steps, including oscillations, delayed dynamics, and convergence to an incorrect steady state~\cite{xu2019stability,xu2023lack,li2026asymptotic}.

At the PDE level, Hao et al.~\cite{hao2025stability} recently investigated the stability and robustness of several time discretization schemes for the Allen--Cahn equation through bifurcation and perturbation analysis. Their notions are formulated in terms of the uniqueness of the forward and inverse discrete solution maps and are used to assess sensitivity to initial conditions and convergence toward physical solutions. The present work addresses a different but complementary question: whether a scheme preserves the local pointwise time direction selected by the full Allen--Cahn residual. This criterion detects transient dynamical distortions, including artificial delay and wrong-signed pointwise increments, even when the discrete solution remains energy stable, stays within the maximum-bound range, or ultimately approaches the correct equilibrium.

Nevertheless, the homogeneous reduction captures only the constant spatial mode and does not identify the full PDE with the scalar ODE. The Allen--Cahn PDE evolves in an infinite-dimensional function space and also contains nonconstant spatial modes, on which diffusion acts. Once the initial datum is spatially inhomogeneous, the diffusion term no longer vanishes, and the initial pointwise direction is determined by the full residual
\[
    \Reps(u):=\Delta u-\frac1{\eps^2}f(u)=\Delta u+\frac1{\eps^2}(u-u^3).
\]
Thus the sign of the PDE velocity depends not only on the range of the data, but also on its spatial profile. In particular, the condition $0<u_0<1$, which fixes the increasing direction of the homogeneous ODE, does not by itself guarantee a nonnegative initial PDE velocity.

A simple periodic example makes this distinction explicit. Let
\[
u_0(x)=a+b\cos(kx), \quad x\in \Omega=(0,2\pi),
\]
where $k\in\mathbb N$ and $0<a-b<a+b<1$. Then
$0<u_0(x)<1$ for all $x\in\Omega$.

For the scalar ODE \eqref{eq:AC-ODE}, the situation is
straightforward. Indeed, the scalar reaction field
\(
s\mapsto \varepsilon^{-2}s(1-s^2)
\)
is positive for every \(0<s<1\). Hence every scalar state in the range
of \(u_0\) has the same positive time direction. In other words, if the
diffusion is suppressed and the reaction dynamics is considered
pointwise, each value \(u_0\in(0,1)\) is driven upward toward the
stable state \(1\).

For the PDE \eqref{eq:AC}, however, the situation is different because an additional spatial contribution
enters through diffusion. Its initial time direction is not determined by the range
of \(u_0\) alone, but by the full Allen--Cahn residual
\[
\mathcal R_\varepsilon(u_0)(x)=
-bk^2\cos(kx)
+\frac{1}{\varepsilon^2}
\bigl(a+b\cos(kx)\bigr)
\Bigl[1-\bigl(a+b\cos(kx)\bigr)^2\Bigr].
\]
At the crest \(x=0\), this gives
\(
\mathcal R_\varepsilon(u_0)(0)
=
-bk^2
+\frac{1}{\varepsilon^2}
(a+b)\bigl[1-(a+b)^2\bigr].
\)
Therefore, if
\(
k^2>
(a+b)\bigl[1-(a+b)^2\bigr]/
(\varepsilon^2 b),
\)
then
\(
\mathcal R_\varepsilon(u_0)(0)<0.
\)
On the other hand, at the trough \(x=\pi/k\), one has
\(
\mathcal R_\varepsilon(u_0)(\pi/k)
=
bk^2
+\frac{1}{\varepsilon^2}
(a-b)\bigl[1-(a-b)^2\bigr]
>0.
\)
Thus, for sufficiently large spatial frequency \(k\), the same initial
profile satisfies \(0<u_0(x)<1\) everywhere, but its initial PDE
velocity has different signs at different spatial points: it is
negative at a sharp crest and positive at a trough. This phenomenon is
illustrated in Figure~\ref{fig:initial-velocity-sign-contrast-final}.

This example shows that the range condition \(0<u_0<1\) is sufficient
to determine the positive direction of the scalar reaction ODE, but it
is not sufficient to determine the pointwise time direction of the
spatially varying Allen--Cahn PDE. For the PDE, the diffusion term may
overcome the positive reaction contribution near a sufficiently sharp
crest, leading to
\(
u_t(x,0)=\mathcal R_\varepsilon(u_0)(x)<0
\)
even though \(0<u_0(x)<1\). Consequently, pointwise monotone growth
requires a sign condition on the full residual \(\mathcal R_\varepsilon(u_0)\),
rather than merely a range condition on \(u_0\). This observation leads
to the central PDE-level question studied in this paper: beyond energy dissipation, which time
discretizations faithfully preserve the pointwise dynamical direction of the Allen--Cahn flow?

This paper formulates and analyzes this pointwise directional structure at the PDE level. Define the Allen--Cahn residual operator by
\begin{equation}
    \Reps(u):=\Delta u+\frac1{\eps^2}(u-u^3).
    \label{eq:intro-residual-operator}
\end{equation}
Then $u_t=\Reps(u)$ for the exact flow. The sign-dependent admissible classes are defined by
\begin{equation}
    \mathcal A_\varepsilon^\sigma:=
    \left\{
    u_0\in C^2(\overline\Omega):
    0< \sigma u_0\le 1,\quad
    \sigma\Reps(u_0)\ge0
    \right\},
    \qquad \sigma\in\{+1,-1\}.
    \label{eq:Aeps-sigma}
\end{equation}
Compatibility with the imposed periodic or homogeneous Neumann boundary
condition is understood in the definition of these classes.
For $\sigma=+1$, the positive-direction class, denoted by $\Aeps$, consists of data satisfying $0< u_0\le1$ and $\Reps(u_0)\ge0$, so the initial velocity is nonnegative at every spatial point. For $\sigma=-1$, the negative-direction class, denoted by $\Aepsm$, describes the symmetric decreasing case. The exact Allen--Cahn flow is shown to preserve these pointwise time directions:
\begin{equation}
    u_0\in\Aeps \quad\Longrightarrow\quad u_t(x,t)\ge0,
    \qquad x\in\Omega,\quad t\ge0,
    \label{eq:intro-monotone-growth}
\end{equation}
and, by the odd symmetry $u\mapsto -u$,
\begin{equation}
    u_0\in\Aepsm \quad\Longrightarrow\quad u_t(x,t)\le0,
    \qquad x\in\Omega,\quad t\ge0.
    \label{eq:intro-monotone-decay}
\end{equation}

Note that this directional property is distinct from the classical maximum principle. The maximum principle constrains only the range of the continuous solution, whereas \eqref{eq:intro-monotone-growth}--\eqref{eq:intro-monotone-decay} prescribe the sign of its time derivative at each fixed spatial point. Thus, pointwise monotonicity is a local-in-space dynamical feature of the Allen--Cahn flow, stronger than mere bound preservation. This distinction is important at the discrete level: a time-stepping method may keep the solution within the admissible range while still failing to reproduce the correct pointwise time direction.

This PDE-level monotonicity is then used as a benchmark for several widely used time discretizations. The resulting analysis shows that unconditional energy stability and pointwise monotonicity are fundamentally different requirements. The main contributions are as follows.
\begin{itemize}[leftmargin=2em]
    \item A PDE-level pointwise-monotonicity framework is developed through sign-dependent admissible classes determined by the full Allen--Cahn residual. The exact flow is proved to preserve the corresponding monotone-growth and monotone-decay directions, and the fully implicit Euler method inherits this structure in its unique-solvability regime. Together, these results provide continuous and discrete benchmarks complementary to the energy law and the maximum principle.

    \item Using these benchmarks, we give a unified dynamical classification of representative energy-stable schemes. First-order convex splitting and sufficiently strong first-order stabilization may preserve the pointwise direction only on a distorted effective time scale, producing artificial delay or damping. By contrast, second-order convex splitting, stabilized CN/AB, IEQ, and SAV can generate wrong-signed pointwise increments for sufficiently large time steps despite dissipation of the original or a modified energy.

    \item The strict sign reversals are shown to persist beyond the spatially homogeneous invariant class under sufficiently small smooth nonhomogeneous admissible perturbations with nonzero diffusion. One- and two-dimensional numerical experiments, including grid-refinement tests, confirm the resulting classification and demonstrate that the observed reversals are not artifacts of an ODE reduction or a particular spatial resolution.
\end{itemize}

The rest of the paper is organized as follows. Section \ref{sec:monotonicity}  establishes pointwise monotonicity of the exact Allen--Cahn flow for both sign-dependent admissible classes. Section \ref{sec:IE} studies the fully implicit Euler method as a monotonicity-preserving reference scheme. Section \ref{sec:energy-stable} collects standard unconditionally energy-stable schemes, including convex splitting, stabilized semi-implicit, IEQ, and SAV schemes, and analyzes their dynamical limitations. Section \ref{sec:nonhomogeneous-persistence} establishes the persistence of the strict sign-reversal counterexamples under smooth spatially nonhomogeneous admissible perturbations. Section \ref{sec:num} presents numerical illustrations. Concluding remarks are given in Section \ref{sec:con}.


\section{Pointwise Monotonicity of the Exact Allen--Cahn Flow}
\label{sec:monotonicity}
Throughout this section, we consider the Allen--Cahn equation \eqref{eq:AC} under either periodic boundary conditions or homogeneous Neumann boundary conditions,
$\partial_\nu u=0$ on $\partial\Omega$. Initial data are understood to be compatible with the imposed boundary condition.
Using the definition of \(f\) in \eqref{eq:AC}, the equation can be equivalently written as
$u_t=\Delta u+\frac{1}{\eps^2}(u-u^3).$
Define the residual operator $\Reps(u)$ in \eqref{eq:intro-residual-operator},
which is the instantaneous velocity determined by \eqref{eq:AC}: if
\(u(x,0)=u_0(x)\), then
$u_t(x,0)=\Reps(u_0)(x).$

Following the sign-dependent definition \eqref{eq:Aeps-sigma}, we write
\(\Aeps\) and \(\Aepsm\) for the positive- and negative-direction classes, respectively. Thus
\begin{equation}
\begin{aligned}
    \Aeps
    &=\left\{u_0\in C^2(\overline\Omega):
      0<u_0\le1,\quad \Reps(u_0)\ge0\right\},\\
    \Aepsm
    &=\left\{u_0\in C^2(\overline\Omega):
      -1\le u_0<0,\quad \Reps(u_0)\le0\right\}.
\end{aligned}
    \label{eq:admissible-negative-class}
\end{equation}
The condition \(u_0\in\Aeps\) means that the initial velocity is nonnegative at every spatial point, whereas \(u_0\in\Aepsm\) means that the initial velocity is nonpositive. Since the two classes are related by the symmetry \(u\mapsto -u\), it is enough to prove the monotone-growth result for \(\Aeps\); the monotone-decay statement then follows immediately.

\subsection{Pointwise Monotonicity of the Allen--Cahn Flow}
\begin{theorem}[Pointwise monotone growth]\label{thm:exact-monotone-growth}
Let \(u\) be a sufficiently smooth solution of \eqref{eq:AC} with periodic or homogeneous Neumann boundary conditions. If
$u_0\in\Aeps$, then $u_t(x,t)\ge0$ for all $(x,t)\in\Omega\times(0,\infty)$.
Consequently, for every fixed \(x\in\Omega\), the map \(t\mapsto u(x,t)\) is nondecreasing.
\end{theorem}

\begin{proof}
Set
$v(x,t):=u_t(x,t).$
Differentiating the equivalent form of \eqref{eq:AC} with respect to time gives
\begin{equation}
    v_t=\Delta v+\frac{1}{\eps^2}(1-3u^2)v.
    \label{eq:v-equation-raw}
\end{equation}
Define $c(x,t):=\frac{1}{\eps^2}\bigl(1-3u(x,t)^2\bigr).$
Then $v_t=\Delta v+c(x,t)v.$
By the maximum principle for the Allen--Cahn equation and the assumption \(0\le u_0\le1\), we have
$0\le u(x,t)\le1$. Hence $c$ is bounded.

At \(t=0\), using the equation itself,
$v(x,0)=u_t(x,0)=\Reps(u_0)(x)\ge0.$
Fix \(T>0\). Choose \(K>\sup_{\Omega\times[0,T]}c\), and set
$w(x,t):=e^{-Kt}v(x,t).$
Then \(w\) satisfies
\begin{equation}
    w_t=\Delta w+(c-K)w,
    \label{eq:w-equation-continuous}
\end{equation}
With \(w(x,0)\ge0\) and \(c-K<0\) on \(\Omega\times[0,T]\), the parabolic maximum principle for periodic or homogeneous Neumann boundary conditions gives \(w\ge0\). Hence \(v=e^{Kt}w\ge0\) on \(\Omega\times[0,T]\). Since \(T\) is arbitrary, \(u_t=v\ge0\) for all \(t\ge0\).
Finally, for \(0\le t_1<t_2\),
\[
    u(x,t_2)-u(x,t_1)=\int_{t_1}^{t_2}u_t(x,s)\,ds\ge0.
\]
This completes the proof.
\end{proof}

\begin{corollary}[Pointwise monotone decay for negative initial data]
Let \(u\) be a sufficiently smooth solution of \eqref{eq:AC} with periodic or homogeneous Neumann boundary conditions. If
$u_0\in\Aepsm,$
then $u_t(x,t)\le0$ for all $(x,t)\in\Omega\times(0,\infty)$.
Consequently, for every fixed \(x\in\Omega\), the map \(t\mapsto u(x,t)\) is nonincreasing.
\end{corollary}

\begin{proof}
Set \(\widetilde u=-u\). Since the Allen--Cahn nonlinearity is odd, \(\widetilde u\) satisfies the same Allen--Cahn equation with initial datum \(\widetilde u_0=-u_0\). If \(u_0\in\Aepsm\), then
$0\le \widetilde u_0\le1$,
and
$
\Reps(\widetilde u_0)=\Reps(-u_0)=-\Reps(u_0)\ge0.
$
Hence \(\widetilde u_0\in\Aeps\). By the preceding theorem, \(\widetilde u_t\ge0\). Therefore \(u_t=-\widetilde u_t\le0\), which proves the claim.
\end{proof}

\begin{remark}[Difference from the maximum principle]
The maximum principle gives the range constraint, for instance \(0\le u(x,t)\le1\) on the positive branch, whereas pointwise monotone growth gives the time-directional order property \(u_t(x,t)\ge0\). These are different properties. The residual condition \(\Reps(u_0)\ge0\) is essential for monotone growth, and the sign-reversed condition \(\Reps(u_0)\le0\) yields the corresponding monotone decay on the negative branch.
\end{remark}

\begin{corollary}[A simple checkable sufficient condition]\label{cor:checkable-condition}
Let \(u_0(x)\in C^2(\overline\Omega)\) satisfy
\(
    0<m\le u_0(x)\le M<1.
\)
If $\eps^2\norm{\Delta u_0}_{L^\infty(\Omega)}\le m(1-M^2)$,
then \(\Reps(u_0)\ge0\). Consequently, the exact solution satisfies \(u_t\ge0\).
\end{corollary}

\begin{proof}
For every \(x\in\Omega\),
$u_0(x)-u_0(x)^3=u_0(x)(1-u_0(x)^2)\ge m(1-M^2)$,
and
$\Delta u_0(x)\ge-\norm{\Delta u_0}_{L^\infty(\Omega)}.$
Thus
$
\Reps(u_0)(x)\ge -\norm{\Delta u_0}_{L^\infty(\Omega)}+\frac{1}{\eps^2}m(1-M^2)\ge0.
$
\end{proof}

\subsection{Examples of admissible initial data}

\begin{example}[Constant initial data]
Let \(u_0(x)\equiv u_0\), where \(0<u_0<1\). Then
\(
    \Reps(u_0)=\frac{1}{\eps^2}(u_0-u_0^3)
    =\frac{1}{\eps^2}u_0(1-u_0^2)>0.
\)
Hence \(u_0\in\Aeps\).
\end{example}

\begin{example}[A shifted cosine initial value]
Let \(\Omega=(0,2\pi)\) with periodic boundary conditions and take
$u_0(x)=c+a\cos(kx),\, k\in\mathbb N.$
Assume \(0<c-|a|\) and \(c+|a|<1\). Since \(\norm{u_0''}_{L^\infty}=|a|k^2\), a sufficient condition for \(u_0\in\Aeps\) is
that $\eps^2 |a|k^2\le (c-|a|)\bigl(1-(c+|a|)^2\bigr).$
\end{example}

\begin{example}[A small periodic perturbation]
Let
$u_0(x)=c+\eta\phi(x),\, x\in\T^d,
$
where \(\phi\in C^2_{\rm per}(\T^d)\), \(\norm{\phi}_{L^\infty}\le1\), and \(\norm{\Delta\phi}_{L^\infty}\le L\). If \(0<c-\eta\), \(c+\eta<1\), and
$\eps^2\eta L\le (c-\eta)\bigl(1-(c+\eta)^2\bigr),$
then \(u_0\in\Aeps\).
\end{example}

\section{Fully Implicit Euler Scheme and Pointwise Monotonicity}
\label{sec:IE}

The fully implicit Euler scheme for the Allen--Cahn equation \eqref{eq:AC} is given by
\begin{equation}
    \frac{u^{n+1}-u^n}{\dt}
    =\Delta u^{n+1}+\frac{1}{\eps^2}\bigl(u^{n+1}-(u^{n+1})^3\bigr).
    \label{eq:IE}
\end{equation}
For each time level, define the discrete Allen--Cahn residual
$ R^n:=\Delta u^n-\frac{1}{\eps^2}f(u^n)
    =\Delta u^n+\frac{1}{\eps^2}\bigl(u^n-(u^n)^3\bigr).$
The usual convexity argument for the fully implicit Allen--Cahn step is based on the functional
\begin{equation}
    \mathcal E_{\dt}^{n+1}(v;u^n)
    :=
    \int_\Omega\left(\frac12|\nabla v|^2+\frac{1}{\eps^2}F(v)\right)dx
    +\frac{1}{2\dt}\norm{v-u^n}_{L^2(\Omega)}^2,
    \label{eq:IE-functional}
\end{equation}
where \(F(v)=\frac14(v^2-1)^2\). This functional is strictly convex, and the fully implicit step is uniquely solvable, under \(\dt<\eps^2\); see \cite{xu2019stability}. We use this strict condition below so that the zero-order coefficient in the monotonicity proof is positive.

\begin{theorem}[Pointwise monotonicity of fully implicit Euler]\label{thm:FIE-monotone}
Assume that \eqref{eq:IE} is solved with periodic or homogeneous Neumann boundary conditions, and that the time step satisfies
$0<\dt<\eps^2.$ If $R^n\ge0$ in $\Omega,$
then $u^{n+1}(x)\ge u^n(x)$ for all $x\in\Omega.$
Moreover, the residual nonnegativity propagates, i.e., $R^{n+1}\ge0$.
Consequently, if \(R^0\ge0\), then
$u^0(x)\le u^1(x)\le u^2(x)\le\cdots$, for all $x\in\Omega.$
\end{theorem}

\begin{proof}
Let \(w^{n+1}:=u^{n+1}-u^n\). From \eqref{eq:IE},
we have
$
 \frac{w^{n+1}}{\dt}=\Delta u^{n+1}+\frac{1}{\eps^2}\bigl(u^{n+1}-(u^{n+1})^3\bigr).
$
Using \(u^{n+1}=u^n+w^{n+1}\), we obtain
\begin{equation}
    \frac{w^{n+1}}{\dt}-\Delta w^{n+1}
    =R^n+\frac{1}{\eps^2}\bigl[(u^{n+1}-(u^{n+1})^3)-(u^n-(u^n)^3)\bigr].
    \label{eq:IE-w-2}
\end{equation}
By the mean value theorem, for each \(x\) there exists \(\xi^n(x)\) between \(u^n(x)\) and \(u^{n+1}(x)\) such that
\[
    (u^{n+1}-(u^{n+1})^3)-(u^n-(u^n)^3)
    =\bigl(1-3(\xi^n)^2\bigr)w^{n+1}.
\]
Since \(1-3(\xi^n)^2 \le1\), \eqref{eq:IE-w-2} becomes
\begin{equation}
    -\Delta w^{n+1}+a^n(x)w^{n+1}=R^n,
    \label{eq:IE-elliptic-w}
\end{equation}
where $ a^n(x):=\frac{1}{\dt}-\frac{1}{\eps^2}\bigl(1-3(\xi^n(x))^2\bigr)\ge
    \frac{1}{\dt}-\frac{1}{\eps^2}> 0.$

Let \((w^{n+1})^-:=\max\{-w^{n+1},0\}\), \((w^{n+1})^+:=\max\{w^{n+1},0\}\), thus
   $ w^{n+1}=(w^{n+1})^+-(w^{n+1})^-$.
Since \(R^n\ge0\), multiplying \eqref{eq:IE-elliptic-w} by \((w^{n+1})^-\) and integrating, we obtain
\[
    \int_\Omega\bigl(-\Delta w^{n+1}+a^n w^{n+1}\bigr)(w^{n+1})^-\,dx
    =
    \int_\Omega R^n(w^{n+1})^-\,dx\ge0.
\]
On the other hand, periodic or homogeneous Neumann boundary conditions imply
\[
    \int_\Omega (-\Delta w^{n+1})(w^{n+1})^-\,dx
    =-\left\|\nabla (w^{n+1})^-\right\|_{L^2(\Omega)}^2\le0,
\]
and
\[
    \int_\Omega a^n w^{n+1}(w^{n+1})^-\,dx
    =-\int_\Omega a^n\bigl((w^{n+1})^-\bigr)^2\,dx\le0.
\]
Thus both sides must be zero. Hence \((w^{n+1})^-=0\) and \(w^{n+1}\ge0\), i.e., \(u^{n+1}\ge u^n\).

Finally, the scheme itself gives
$
    R^{n+1}=\Delta u^{n+1}+\frac{1}{\eps^2}\bigl(u^{n+1}-(u^{n+1})^3\bigr)=\frac{u^{n+1}-u^n}{\dt}\ge0.
$
This proves the propagation of residual nonnegativity and completes the induction.
\end{proof}

\begin{remark}[Compatibility with the continuous admissible class]
If \(u_0\in\Aeps\) and \(0<\dt<\eps^2\), then
\[
    R^0=\Delta u^0+\frac{1}{\eps^2}\bigl(u^0-(u^0)^3\bigr)=\Reps(u_0)\ge0.
\]
Thus the fully implicit Euler monotonicity theorem applies directly to admissible initial data.
\end{remark}

\section{Dynamical Limitations of Unconditionally Energy-Stable Schemes}
\label{sec:energy-stable}

This section collects four representative mechanisms for constructing unconditionally energy-stable or modified-energy-stable Allen--Cahn schemes: convex splitting, stabilization, invariant energy quadratization, and scalar auxiliary variables. The common theme is that global energy dissipation does not, by itself, guarantee the preservation of the local pointwise monotone-growth direction established in the preceding sections.

\subsection{Convex splitting schemes}
\label{sec:CSS}

This subsection studies two convex splitting discretizations. We first analyze the standard first-order convex splitting scheme CSS1. Although it is unconditionally energy stable with respect to the original Allen--Cahn energy and preserves the monotone direction, this property is inherited through its equivalence to a fully implicit Euler step with a reduced effective time step. As a consequence, the scheme may reproduce the correct direction but on a distorted time scale, leading to artificial delay in the transient dynamics. We then consider the standard second-order convex splitting scheme of Crank--Nicolson/Adams--Bashforth type. This higher-order convex splitting scheme is stable in a modified-energy sense rather than in the original-energy sense. We show that this modified-energy stability still does not imply pointwise monotonicity for large time steps.

\subsubsection{Scheme and unconditional energy stability}
The first-order convex splitting scheme for the Allen--Cahn equation \eqref{eq:AC} is given by
\begin{equation}
    \frac{u^{n+1}-u^n}{\tau}
    =\Delta u^{n+1}
    +\frac{1}{\eps^2}\bigl(u^n-(u^{n+1})^3\bigr).
    \label{eq:CSS1}
\end{equation}
It is unconditionally energy stable with respect to the original Allen--Cahn
energy for every $\tau>0$; see \cite{eyre1998unconditionally,tangqiao2020phasefield,xu2019stability}.
In what follows, this standard energy property is used only as background. Our concern is whether
this energy-stable scheme also reproduces the pointwise monotone dynamics of the Allen--Cahn flow.

\subsubsection{Effective implicit-Euler structure}

The following elementary identity is the structural reason behind both the monotonicity and
the delay of CSS1.

\begin{lemma}[Effective fully implicit form of CSS1 \cite{tangqiao2020phasefield,xu2019stability}]
\label{lem:CSS1-effective-step}
The CSS1 scheme \eqref{eq:CSS1} is equivalent to
\begin{equation}
    \frac{u^{n+1}-u^n}{\tau_e} =
    \Reps(u^{n+1})=
    \Delta u^{n+1}
    +\frac{1}{\eps^2}\bigl(u^{n+1}-(u^{n+1})^3\bigr),
    \label{eq:CSS1-effective-IE}
\end{equation}
where $\tau_e:=\frac{\eps^2\tau}{\eps^2+\tau}$.
Consequently, $0<\tau_e<\eps^2$ for every $\tau>0.$
\end{lemma}

\begin{theorem}[Pointwise monotonicity of CSS1]
\label{thm:CSS1-monotonicity}

Assume that \eqref{eq:CSS1} is solved with periodic or homogeneous Neumann boundary conditions. Then, for every nominal time step $\tau>0$, if $R^{n}=\Reps(u^{n})\ge0$ in $\Omega,$
then the CSS1 solution satisfies $u^{n+1}(x)\ge u^n(x)$ and $R^{n+1}=\Reps(u^{n+1})\ge0$ for all $x\in\Omega$.
Consequently, if $R^0\ge0$, then $u^0(x)\le u^1(x)\le u^2(x)\le\cdots$.
\end{theorem}

\begin{proof}
By Lemma~\ref{lem:CSS1-effective-step}, one CSS1 step with nominal step $\tau$ is exactly
one fully implicit Euler step with step $\tau_e$. Since $0<\tau_e<\eps^2$,
the uniqueness and monotonicity result for the fully implicit Euler method from Theorem~\ref{thm:FIE-monotone} applies and yields $u^{n+1}\ge u^n$.
Moreover, \eqref{eq:CSS1-effective-IE} yields
\[
    R^{n+1}=\Reps(u^{n+1})=\frac{u^{n+1}-u^n}{\tau_e}\ge0.
\]
The induction follows immediately.
\end{proof}

\subsubsection{Large-step delay}

Although CSS1 is unconditionally energy stable and preserves the monotone direction through Lemma~\ref{lem:CSS1-effective-step}, it does not evolve on the nominal physical time scale. Dividing \eqref{eq:CSS1-effective-IE} by $\tau$ gives
\begin{equation}
    \frac{u^{n+1}-u^n}{\tau} =
    \alpha_{\rm css}(\tau)\Reps(u^{n+1}),
    \qquad
    \alpha_{\rm css}(\tau):=\frac{\tau_e}{\tau}
    =\frac{\eps^2}{\eps^2+\tau}.
    \label{eq:CSS1-slope-compression}
\end{equation}
Thus CSS1 compresses the physical-time velocity by the factor $\alpha_{\rm css}(\tau)$. In particular, $\alpha_{\rm css}(\tau)\to0$ as $\tau\to\infty$. Equivalently, after $n$ steps the nominal time is $t_n=n\tau$, whereas the effective implicit-Euler time is
\begin{equation}
    t_n^{\rm eff}=n\tau_e
    =\frac{\eps^2}{\eps^2+\tau}t_n .
    \label{eq:CSS1-effective-time}
\end{equation}
This identity gives a precise interpretation of the large-step behavior: CSS1 can be monotone in the step index while remaining delayed in physical time.


\begin{theorem}[Large-step limit of CSS1]
\label{thm:CSS1-large-step-limit}
Assume that the first-order convex splitting scheme \eqref{eq:CSS1} is solved with periodic or homogeneous Neumann boundary conditions and that
$u^0(x)\equiv u_0$ with $0<u_0<1$. Let
\(u(t)\) denote the exact homogeneous Allen--Cahn solution of
\eqref{eq:AC-ODE} with \(u(0)=u_0\), and let \(t_1(\tau)=\tau\) be the
nominal physical time after one step. Then
\begin{equation}
    \lim_{\tau\to\infty}u^1
    =(u_0)^{1/3}<1
    =\lim_{\tau\to\infty}u\bigl(t_1(\tau)\bigr).
    \label{eq:CSS1-large-step-limit}
\end{equation}
Consequently, one nominally large CSS1 step does not attain the correct
physical-time equilibrium.
\end{theorem}

\begin{proof}
For homogeneous data, \eqref{eq:CSS1} becomes
\begin{equation}
    u_0
    =\left(1-\frac{\tau_e}{\eps^2}\right)u^1
    +\frac{\tau_e}{\eps^2}(u^1)^3.
    \label{eq:CSS1-scalar-rho}
\end{equation}
For every $\tau>0$, the right-hand side of \eqref{eq:CSS1-scalar-rho},
viewed as a function of $u^1$, is strictly increasing on $[0,1]$ and
takes the values $0$ and $1$ at the endpoints. Hence $u^1\in(0,1)$ is
uniquely determined. Since
$\tau_e\to\eps^2$ as $\tau\to\infty$, every limit point $U$ of
$u^1$ satisfies $u_0=U^3$. Thus $U=(u_0)^{1/3}$, which proves the first
limit in \eqref{eq:CSS1-large-step-limit}.

The exact homogeneous solution satisfies
\[
    u'(t)=\frac{1}{\eps^2}u(t)\bigl(1-u(t)^2\bigr),
    \qquad u(0)=u_0\in(0,1),
\]
so \(u(t)\to1\) as \(t\to\infty\). Since
\(t_1(\tau)=\tau\to\infty\), the second limit in
\eqref{eq:CSS1-large-step-limit} follows.
\end{proof}

\begin{remark}[Large-step delay versus the discrete long-time limit]
The large-step limit in Theorem~\ref{thm:CSS1-large-step-limit} should not be
confused with the conventional discrete long-time limit. In both cases, the
nominal physical time
$t_n=n\tau$
tends to infinity, but it does so in different ways. If \(n\) is fixed and
\(\tau\to\infty\), then
$t_n^{\rm eff}=n\tau_e\longrightarrow n\eps^2,$
so that a finite number of nominally large CSS1 steps represents only a finite
amount of effective implicit-Euler evolution. In contrast, if \(\tau\) is
fixed and \(n\to\infty\), then
$t_n^{\rm eff}=n\tau_e\longrightarrow\infty.$
Thus, the large-step limit does not contradict the usual discrete long-time
behavior. Rather, it shows that CSS1 may preserve the correct pointwise
direction while substantially lagging behind the dynamics associated with the
nominal physical time. This mismatch is the artificial large-step delay.
\end{remark}

\subsubsection{Second-order convex splitting of Crank--Nicolson/Adams--Bashforth type}

The preceding discussion shows that CSS1 is exceptional: its monotonicity follows from an exact reformulation as a fully implicit Euler step with a smaller effective step. This mechanism is not available for the usual higher-order convex splitting schemes. Following the standard construction reviewed in \cite{tangqiao2020phasefield}, we next consider a second-order Crank--Nicolson/Adams--Bashforth convex splitting method: the convex quartic part is approximated by a modified Crank--Nicolson discrete gradient, whereas the concave quadratic part is extrapolated explicitly.
Define the discrete derivative of the convex part by
\begin{equation}
    D F_c(a,b) := \frac{F_c(a)-F_c(b)}{a-b} =\frac{(a+b)(a^2+b^2)}{4},
    \qquad a\ne b,
    \label{eq:DFc-discrete-gradient}
\end{equation}
with the continuous extension \(D F_c(a,a)=a^3\). With this discrete derivative, the second-order convex splitting scheme for the Allen--Cahn equation \eqref{eq:AC} is given by
\begin{equation}
\begin{aligned}
    \frac{u^{n+1}-u^n}{\tau}
    ={}&\Delta \left(\frac{u^{n+1}+u^n}{2}\right)
    +\frac1{\eps^2}
    \left[
        \frac32u^n-\frac12u^{n-1}-D F_c(u^{n+1},u^n)
    \right],
    \qquad n\ge1.
\end{aligned}
    \label{eq:CSmodCN}
\end{equation}

Here one starting value \(u^1\) is required. In the theoretical counterexample below, \(u^1\) is chosen from the exact homogeneous Allen--Cahn flow so that the starting pair is already in the correct monotone direction. In the numerical tests, the same role is played by a highly resolved monotone fully implicit Euler starter. Thus any subsequent negative increment is caused by the two-step convex splitting update itself, not by an inconsistent starting procedure.

The scheme \eqref{eq:CSmodCN} is a standard unconditionally modified-energy-stable
convex splitting discretization; see \cite{tangqiao2020phasefield,guan2014second,guillen2014second,xu2023lack}.
Its energy law involves the usual quadratic correction in the previous time increment rather than
the original Allen--Cahn energy alone. We use this standard fact only to emphasize that modified-energy
dissipation does not provide a pointwise monotonicity guarantee. The following theorem shows that even an exact monotone start can be followed by an increment with the wrong sign when the time step is sufficiently large.


\begin{theorem}[Large-step loss of pointwise monotonicity]
\label{thm:CSmodCN-nonmonotone}
Consider the second-order convex splitting scheme \eqref{eq:CSmodCN}
under periodic or homogeneous Neumann boundary conditions. There exist
admissible monotone data such that, for all sufficiently large time steps
\(\tau\), the numerical solution generated by the scheme fails to
preserve pointwise monotonicity. Consequently, the scheme is not
unconditionally pointwise monotone.
\end{theorem}

\begin{proof}
For simplicity, we exhibit this loss of monotonicity using a constant
initial datum. Fix \(u_0\in(0,1)\), set \(u^0(x)\equiv u_0\), and choose
the exact increasing starter \(u^1(x)\equiv u(\tau)\), where \(u(t)\)
solves \eqref{eq:AC-ODE} with \(u(0)=u_0\). Spatially homogeneous states
are preserved by \eqref{eq:CSmodCN}, so the spatial variable may be
omitted and the scheme reduces to
\begin{equation}
    D F_c(u^{n+1},u^n)
    +\frac{\eps^2}{\tau}(u^{n+1}-u^n)
    =\frac32u^n-\frac12u^{n-1}.
    \label{eq:csmodcn-scalar-limit-form}
\end{equation}
For fixed \(u^n\), the left-hand side is strictly increasing in
\(u^{n+1}\), since its derivative is
\[
    \frac{2(u^{n+1})^2+(u^{n+1}+u^n)^2}{4}
    +\frac{\eps^2}{\tau}>0.
\]
Together with its cubic growth, this implies that each update is unique
and depends continuously on the preceding values.

Since \(u^1=u(\tau)\to1\) as \(\tau\to\infty\), passing to the limit in
the \(n=1\) update gives
$
u^2\longrightarrow U,
D F_c(U,1)=\frac{3-u_0}{2},
$
where \(U\) is the unique solution of the displayed scalar equation.
Because
$
DF_c(1,1)=1<\frac{3-u_0}{2}
$
and \(D F_c(\,\cdot\,,1)\) is strictly increasing, we have \(U>1\).

For the \(n=2\) update, evaluate the left-hand side of
\eqref{eq:csmodcn-scalar-limit-form} at the comparison value
\(u^3=u^2\). It then equals
\[
    D F_c(u^2,u^2)=(u^2)^3.
\]
By the strict monotonicity established above, \(u^3<u^2\) whenever
\(
    (u^2)^3>\frac32u^2-\frac12u^1.
\)
As \(\tau\to\infty\), the difference between the two sides tends to
\[
    U^3-\frac32U+\frac12
    =\frac12(U-1)(2U^2+2U-1)>0.
\]
Hence the inequality, and therefore \(u^3(x)<u^2(x)\) for every
\(x\in\Omega\), holds for all sufficiently large \(\tau\). This proves
the claim.
\end{proof}

\begin{remark}[Finite large-step threshold]
The limit $\tau\to\infty$ is used only to simplify the proof; the scheme is
not evaluated at an infinite time step. Since the limiting inequalities
obtained above are strict, the order-preserving property of limits implies
that they remain valid for all sufficiently large but finite $\tau$.
Consequently, there exists a finite threshold
$\tau_{*,\mathrm{CS}}>0$ such that
\(u^3(x)<u^2(x)\) for \(x\in\Omega,\)
for every $\tau>\tau_{*,\mathrm{CS}}$. An explicit formula for this
threshold is not needed for the present counterexample.
\end{remark}
\begin{remark}[Extension beyond homogeneous data]
The constant datum in Theorem~\ref{thm:CSmodCN-nonmonotone} is used to
make the strict wrong-signed increment explicit. The corresponding
perturbation proof is given in Proposition~\ref{prop:CSmodCN-nonhomogeneous-persistence}. It shows that, for any
fixed finite step in this strict-decrement regime, the inequality
\(u^3<u^2\) persists under sufficiently small smooth admissible
perturbations \(u^0_\delta=\overline u_0+\delta\psi\) satisfying
\(\Delta\psi\not\equiv0\). Hence the loss of pointwise monotonicity is not
confined to the spatially homogeneous ODE reduction; it also occurs when
diffusion is genuinely active.
\end{remark}

\subsection{Stabilized semi-implicit schemes}
\label{sec:stabilized}

We next consider stabilized semi-implicit methods. The presence of a
stabilization term alone does not guarantee energy stability or
pointwise monotonicity; suitable conditions on the stabilization
parameter, and sometimes on the time step, are still required. In the
strongly stabilized regime, the first-order scheme can preserve the
pointwise direction, but its large-step behavior is affected by
artificial damping and a distorted physical time scale. The
second-order stabilized CN/AB scheme exhibits a different limitation:
even in the modified-energy-stable regime, the Adams--Bashforth
extrapolation may reverse the pointwise increment for sufficiently
large time steps.

\subsubsection{First-order stabilized scheme and energy stability}

The first-order stabilized semi-implicit scheme for the Allen--Cahn equation \eqref{eq:AC} is given by
\begin{equation}
    \frac{u^{n+1}-u^n}{\tau}
    =\Delta u^{n+1}
    -\frac1{\eps^2}f(u^n)
    -S(u^{n+1}-u^n),
    \qquad S\ge0,
    \label{eq:stabilized-scheme}
\end{equation}
where \(f(u)=u^3-u\). For data in the invariant interval \([-1,1]\), the strongly stabilized regime
$S\ge \frac{2}{\eps^2}$
gives the standard unconditional energy-stability condition for this method; see
\cite{tangqiao2020phasefield,tangyang2016imex}. The interval invariance needed
for the positive branch is also established below.
This energy property is recalled only to distinguish global energy dissipation from the pointwise
monotonicity considered below. We therefore work directly with the discrete update rather than
with its standard energy estimate.

\subsubsection{Pointwise monotonicity of the first-order stabilized scheme}
We now show that the monotonicity of the first-order stabilized scheme can be understood through a resolvent structure. Define
\begin{equation}
    B^n:=\frac1{\eps^2}\bigl(u^n-(u^n)^3\bigr),
    \qquad
    R^n:=\Delta u^n+B^n=\Reps(u^n),
    \label{eq:stabilized-B-R}
\end{equation}
and set
$w^{n+1}:=u^{n+1}-u^n,
    \lambda:=\frac1\tau+S .$
Then the stabilized scheme \eqref{eq:stabilized-scheme} is equivalent to
\begin{equation}
    (\lambda-\Delta)w^{n+1}=R^n .
    \label{eq:stabilized-increment-equation}
\end{equation}
Equivalently,
\begin{equation}
    w^{n+1} =
    (\lambda-\Delta)^{-1}R^n
    =
    \tau_S(I-\tau_S\Delta)^{-1}R^n,
    \qquad
    \tau_S:=\frac{\tau}{1+S\tau}.
    \label{eq:stabilized-resolvent-form}
\end{equation}
This form makes the monotonicity mechanism transparent: if the residual is nonnegative, then the elliptic resolvent produces a nonnegative increment.

\begin{theorem}[Pointwise monotonicity of the first-order stabilized scheme]
\label{thm:stab1-monotone}
Consider \eqref{eq:stabilized-scheme} with periodic or homogeneous Neumann boundary conditions and $S\ge 2/\eps^2$.
If $0\le u^n\le 1$ and $\Reps(u^n)\ge 0$ in $\Omega$, then for any $\tau>0$,
\[
0\le u^{n+1}\le 1,\quad u^{n+1}\ge u^n,\quad \Reps(u^{n+1})\ge 0\quad\text{in }\Omega.
\]
Hence, for $u_0\in\Aeps$, the sequence $\{u^n\}$ is pointwise nondecreasing and stays in $[0,1]$.
\end{theorem}


\begin{proof}
Set $w^{n+1}:=u^{n+1}-u^n,
    \lambda:=\frac1\tau+S.$
Suppose that
$0\le u^n\le1,\,\Reps(u^n)\ge0.$
The increment equation \eqref{eq:stabilized-increment-equation} is
$
    (\lambda-\Delta)w^{n+1}=\Reps(u^n).
$
Since \(\lambda>0\), the elliptic maximum principle yields
$w^{n+1}\ge0.$
Hence \(u^{n+1}\ge u^n\ge0\).

Moreover, \eqref{eq:stabilized-scheme} can be rewritten as
$
    (\lambda-\Delta)u^{n+1}=\lambda u^n+\frac1{\eps^2}\bigl(u^n-(u^n)^3\bigr).
$
Thus
\begin{equation}
    (\lambda-\Delta)(1-u^{n+1})
    =\lambda(1-u^n)-\frac1{\eps^2}u^n\bigl(1-(u^n)^2\bigr).
    \label{eq:s1}
\end{equation}
For \(0\le u^n\le1\),
\[
    u^n\bigl(1-(u^n)^2\bigr)
    =u^n(1-u^n)(1+u^n)
    \le2(1-u^n).
\]
Since \(\lambda\ge S\ge2/\eps^2\), the right-hand side of \eqref{eq:s1} is nonnegative.
The elliptic maximum principle then gives \(u^{n+1}\le1\).

Finally, the stabilized scheme yields
$
    \lambda w^{n+1}=\Delta u^{n+1}+\frac1{\eps^2}\bigl(u^n-(u^n)^3\bigr).
$
Hence
\[
    \Reps(u^{n+1})
    =\lambda w^{n+1}
    +\frac1{\eps^2}
    \Bigl[u^{n+1}-(u^{n+1})^3-u^n+(u^n)^3\Bigr].
\]
For each \(x\in\Omega\), the mean-value theorem gives
$
    u^{n+1}-(u^{n+1})^3-u^n+(u^n)^3=(1-3\xi_n^2)w^{n+1},
$
where \(\xi_n(x)\) lies between \(u^n(x)\) and \(u^{n+1}(x)\). Since
\(0\le u^n\le u^{n+1}\le1\), then
$1-3\xi_n^2\ge-2.$
Therefore,
$
    \Reps(u^{n+1})\ge\left(\lambda-\frac2{\eps^2}\right)w^{n+1}\ge0.
$
This proves the one-step implication. If \(u_0\in\Aeps\), the hypotheses
hold at \(n=0\), and induction gives the asserted result for all \(n\ge0\).
\end{proof}

\subsubsection{Artificial damping and large-step behavior of the first-order scheme}

Dividing \eqref{eq:stabilized-resolvent-form} by the nominal time step gives
\begin{equation}
    \frac{u^{n+1}-u^n}{\tau}
    =
    \frac{1}{1+S\tau}(I-\tau_S\Delta)^{-1}\Reps(u^n),
    \qquad
    \tau_S=\frac{\tau}{1+S\tau}.
    \label{eq:stabilized-slope}
\end{equation}
Equivalently, the stabilized update is the usual first-order semi-implicit
scheme with the reduced effective step \(\tau_S=\tau/(1+S\tau)\). Thus it contains the damping factor
\begin{equation}
    \alpha_S(\tau):=\frac{1}{1+S\tau}.
    \label{eq:stabilized-alpha}
\end{equation}
For spatially homogeneous data, we identify the constant function \(u^n(x)\)
with its value \(u^n\). Then \eqref{eq:stabilized-resolvent-form} reduces to
\begin{equation}
    u^{n+1}-u^n
    =\frac{\tau}{\eps^2(1+S\tau)}u^n\bigl(1-(u^n)^2\bigr).
    \label{eq:stab-scalar-update}
\end{equation}
\begin{theorem}[Large-step limit of the first-order stabilized scheme]
\label{thm:stabilized-strong-delay}
Let \(u^0(x)\equiv u_0\) with \(0<u_0<1\), and assume that
\(S\ge 2/\eps^2\). Under periodic or homogeneous Neumann boundary
conditions, the first stabilized iterate of \eqref{eq:stabilized-scheme} remains spatially homogeneous;
let \(u(t)\) denote the exact
homogeneous Allen--Cahn solution of \eqref{eq:AC-ODE} with
\(u(0)=u_0\), and let \(t_1(\tau)=\tau\) be the nominal physical time
after one step. Then
\begin{equation}
    \lim_{\tau\to\infty}u^1
    =u_0+\frac{1}{S\eps^2}u_0(1-u_0^2)
    <1
    =\lim_{\tau\to\infty}u\bigl(t_1(\tau)\bigr).
    \label{eq:stabilized-large-step-limit}
\end{equation}
Consequently, one nominally large stabilized step does not attain the
correct physical-time equilibrium.
\end{theorem}

\begin{proof}
For homogeneous data, \eqref{eq:stab-scalar-update} with \(n=0\) gives
$
    u^1=u_0+\frac{\tau}{\eps^2(1+S\tau)}u_0(1-u_0^2).
$
Hence
$
    \lim_{\tau\to\infty}u^1=u_0+\frac{1}{S\eps^2}u_0(1-u_0^2).
$
Since \(S\eps^2\ge2\) and \(0<u_0<1\),
we have $u_0+\frac{1}{S\eps^2}u_0(1-u_0^2)\le u_0+\frac12u_0(1-u_0^2)<1.$
The exact homogeneous solution satisfies
$
    u'(t)=\frac{1}{\eps^2}u(t)\bigl(1-u(t)^2\bigr),
    \, u(0)=u_0\in(0,1),
$
so \(u(t)\to1\) as \(t\to\infty\). Since
\(t_1(\tau)=\tau\to\infty\), the second limit in
\eqref{eq:stabilized-large-step-limit} follows.
\end{proof}

\begin{remark}[Monotonicity preservation versus large-step delay]
The first-order convex-splitting and first-order stabilized schemes share
a similar large-step behavior. As shown in
Theorem~\ref{thm:CSS1-large-step-limit} and
Theorem~\ref{thm:stabilized-strong-delay}, both schemes preserve the
correct pointwise monotone direction for every \(\tau>0\); however, as
\(\tau\to\infty\), their one-step numerical solutions do not attain the
correct equilibrium on the nominal physical-time scale. Thus, pointwise
monotonicity preservation alone does not guarantee faithful physical-time
relaxation.
\end{remark}

\subsubsection{A second-order stabilized CN/AB scheme}

The second-order stabilized Crank--Nicolson/Adams--Bashforth scheme for the Allen--Cahn equation \eqref{eq:AC} is given by
\begin{equation}
    \frac{u^{n+1}-u^n}{\tau}
    =\Delta\left(\frac{u^{n+1}+u^n}{2}\right)
    -\frac{S\tau}{\eps^2}(u^{n+1}-u^n)
    -\frac1{\eps^2}
    \left(\frac32 f(u^n)-\frac12 f(u^{n-1})\right),
    \qquad n\ge1.
    \label{eq:stab-CNAB}
\end{equation}
This stabilized CN/AB scheme was studied by Feng--Tang--Yang
\cite{feng2013stabilized}. Their modified-energy estimate assumes
\(|f'(s)|\le L\) on the relevant range (or uses the standard globally
Lipschitz truncation of the cubic nonlinearity). If
\(
S\ge \frac{L^2}{4\eps^2},
\)
the scheme is unconditionally stable with respect
to a modified energy. Nevertheless, the counterexample below shows that the
Adams--Bashforth extrapolation may reverse the
pointwise increment after an exact increasing start. Hence modified-energy
stability does not imply pointwise monotonicity.

\begin{theorem}[Large-step loss of pointwise monotonicity]
\label{thm:stab-CNAB-nonmonotone}
Consider the stabilized CN/AB scheme \eqref{eq:stab-CNAB} under periodic
or homogeneous Neumann boundary conditions, with any fixed \(S\ge0\). There exist admissible
monotone data such that, for all sufficiently large time steps \(\tau\),
the numerical solution generated by the scheme fails to preserve
pointwise monotonicity. Consequently, the scheme is not unconditionally
pointwise monotone.
\end{theorem}

\begin{proof}
For simplicity, we demonstrate the loss of monotonicity using a constant
initial datum. Fix \(u_0\in(0,1)\), set \(u^0(x)\equiv u_0\), and choose
the exact increasing starter \(u^1(x)\equiv u(\tau)\), where \(u(t)\)
solves \eqref{eq:AC-ODE} with \(u(0)=u_0\). Since spatially homogeneous
states are preserved by \eqref{eq:stab-CNAB}, the spatial variable may be
omitted, and the scheme reduces to
\begin{equation}
    u^{n+1}-u^n
    =-\frac{\tau}{\eps^2+S\tau^2}
    \left(\frac32 f(u^n)-\frac12 f(u^{n-1})\right),
    \qquad n\ge1.
    \label{eq:stab-CNAB-scalar-threshold}
\end{equation}
The exact starter is
\begin{equation}
    u^1=u(\tau)
    =\frac1{\sqrt{1+(u_0^{-2}-1)e^{-2\tau/\eps^2}}},
    \qquad u_0<u^1<1.
    \label{eq:stab-CNAB-exact-start-formula}
\end{equation}
Set \(B:=u_0^{-2}-1\) and \(z:=e^{-2\tau/\eps^2}\). Then
$
    \frac{-f(u^1)}{-f(u_0)}
    =\frac{z}{u_0^3(1+Bz)^{3/2}}
    \le\frac{z}{u_0^3}.
$
Hence, if
\begin{equation}
    \tau>\tau_{*,\mathrm{stab}}(u_0,\eps)
    :=\frac{\eps^2}{2}\log\left(\frac{3}{u_0^3}\right),
    \label{eq:stab-CNAB-explicit-threshold}
\end{equation}
then \(z<u_0^3/3\), and therefore
$
    \frac32 f(u^1)-\frac12 f(u_0)>0.
$
Taking \(n=1\) in \eqref{eq:stab-CNAB-scalar-threshold} now gives
\[
    u^2-u^1
    =-\frac{\tau}{\eps^2+S\tau^2}
    \left(\frac32 f(u^1)-\frac12 f(u_0)\right)<0.
\]
Thus \(u^2(x)<u^1(x)\) for every \(x\in\Omega\) whenever \(\tau\) is
sufficiently large, which proves the claim.
For any fixed \(S>0\), one also has \(u^2-u^1\to0\) as
\(\tau\to\infty\). Hence, when \(S\) is chosen in the modified-energy-stable
regime, the values used in this counterexample remain in \([-1,1]\) for all
sufficiently large \(\tau\); the same construction therefore applies to the
standard globally Lipschitz truncation of \(f\).
\end{proof}
\begin{remark}[Extension beyond homogeneous data]
The homogeneous construction in Theorem~\ref{thm:stab-CNAB-nonmonotone}
provides a strict negative increment. The corresponding perturbation proof
is given in Proposition~\ref{prop:stab-CNAB-nonhomogeneous-persistence}. By continuous
dependence of the exact PDE starter and the stabilized elliptic update, it
shows that
this sign reversal persists for sufficiently small smooth admissible
perturbations with \(\Delta u^0_\delta\not\equiv0\). Thus the counterexample
remains valid for genuinely spatially nonhomogeneous data.
\end{remark}

The first- and second-order stabilized schemes therefore play different
roles in the present analysis. In the strongly stabilized regime, the
first-order scheme can preserve the pointwise direction, but only at the
cost of artificial damping and a distorted physical time scale. The
second-order stabilized CN/AB scheme exhibits a different limitation:
even when the stabilization parameter is chosen in the
modified-energy-stable regime, the Adams--Bashforth extrapolation of the
nonlinear term may reverse the pointwise increment after an exact
increasing start. Thus stabilization and modified-energy dissipation do
not, by themselves, guarantee preservation of pointwise monotonicity.

\subsection{Invariant energy quadratization scheme}
\label{sec:IEQ}

The first-order IEQ scheme for the Allen--Cahn equation \eqref{eq:AC} is given by
\begin{equation}
\left\{
\begin{aligned}
    \frac{u^{n+1}-u^n}{\tau}
    &=\Delta u^{n+1}-\frac{1}{\eps^2}q^{n+1}H_n,\\
    \frac{q^{n+1}-q^n}{\tau}
    &=\frac12H_n\frac{u^{n+1}-u^n}{\tau},
\end{aligned}
\right.
\qquad n\ge0,
\label{eq:IEQ-u-scheme}
\end{equation}
where
\begin{equation}
    Q_n:=\sqrt{F(u^n)+C_{\rm IEQ}},
    \qquad
    H_n:=\frac{f(u^n)}{Q_n},
    \qquad
    C_{\rm IEQ}>0,
    \qquad
    q^0=Q_0.
    \label{eq:IEQ-QH}
\end{equation}
For later use, define the IEQ scaling factor by
$\vartheta_{n+1}:=\frac{q^{n+1}}{Q_n}.$
Then the phase equation in \eqref{eq:IEQ-u-scheme} becomes
\begin{equation}
    \frac{u^{n+1}-u^n}{\tau}
    =\Delta u^{n+1}
    +\frac{\vartheta_{n+1}}{\eps^2}
    \bigl(u^n-(u^n)^3\bigr).
    \label{eq:IEQ-scaled-form}
\end{equation}
The first-order IEQ scheme \eqref{eq:IEQ-u-scheme} is unconditionally stable
with respect to its standard modified energy; see \cite{yang2020convergence}.
This global property does not by itself control the sign of the pointwise phase
increment.

\subsubsection{Large-step failure of pointwise monotonicity for IEQ}

\begin{theorem}[Large-step loss of pointwise monotonicity]
\label{thm:IEQ-large-step-loss}
Consider the first-order IEQ scheme \eqref{eq:IEQ-u-scheme} under
periodic or homogeneous Neumann boundary conditions. There exist
admissible monotone data such that, for all sufficiently large time steps
\(\tau\), the numerical solution generated by the scheme fails to
preserve pointwise monotonicity. Consequently, the scheme is not
unconditionally pointwise monotone.
\end{theorem}

\begin{proof}
For simplicity, we construct the counterexample from a constant initial
datum. Fix \(u_0\in(0,1)\), set \(u^0(x)\equiv u_0\), and take
\(q^0(x)=Q_0(x)\). The homogeneous data are preserved by
\eqref{eq:IEQ-u-scheme}, so all iterates are independent of \(x\).
Taking \(n=0\), using \(q^0=Q_0\), and eliminating \(q^1\) gives
\begin{equation}
    \vartheta_1
    =\left(1+
    \frac{\tau u_0^2(1-u_0^2)^2}{2\eps^2Q_0^2}\right)^{-1}>0,
    \qquad
    u^1
    =u_0+\frac{\tau u_0(1-u_0^2)/\eps^2}
    {1+\tau u_0^2(1-u_0^2)^2/(2\eps^2Q_0^2)}.
    \label{eq:IEQ-u1-tau}
\end{equation}
Since
$
    Q_0^2=F(u_0)+C_{\mathrm{IEQ}}=\frac14(1-u_0^2)^2+C_{\mathrm{IEQ}},
$
we have \(u^1>1\) whenever
\begin{equation}
    \tau>\tau_{*,\mathrm{IEQ}}(u_0,\eps,C_{\mathrm{IEQ}})
    :=
    \frac{\eps^2\bigl((1-u_0^2)^2+4C_{\mathrm{IEQ}}\bigr)}
    {u_0(1-u_0)(1-u_0^2)^2+4u_0(1+u_0)C_{\mathrm{IEQ}}}.
    \label{eq:IEQ-thm-threshold}
\end{equation}
It follows that
$
    u^1-(u^1)^3<0,\, q^1=\vartheta_1Q_0>0.
$
Combining the two IEQ equations at \(n=1\) yields
\[
    \vartheta_2
    =\frac{2q^1Q_1}
    {2Q_1^2+\frac{\tau}{\eps^2}
    \bigl(u^1-(u^1)^3\bigr)^2}>0.
\]
Therefore, \eqref{eq:IEQ-scaled-form} gives
\[
    u^2-u^1
    =\frac{\tau\vartheta_2}{\eps^2}
    \bigl(u^1-(u^1)^3\bigr)<0.
\]
Thus \(u^2(x)<u^1(x)\) for every \(x\in\Omega\) whenever \(\tau\) is
sufficiently large, which proves the claim.
\end{proof}
\begin{remark}[Extension beyond homogeneous data]
Although Theorem~\ref{thm:IEQ-large-step-loss} is derived on the homogeneous
invariant class, the strict decrement is stable under spatial perturbation.
The corresponding proof in Proposition~\ref{prop:IEQ-nonhomogeneous-persistence}
verifies the continuous dependence of both the phase variable and the IEQ
auxiliary variable and proves persistence of \(u^2<u^1\) for sufficiently small smooth admissible
perturbations with nonzero diffusion.
\end{remark}

\subsection{Scalar auxiliary variable scheme}
\label{sec:SAV}

The first-order SAV scheme for the Allen--Cahn equation \eqref{eq:AC} is given by
\begin{equation}
\left\{
\begin{aligned}
    \frac{u^{n+1}-u^n}{\tau}
    &=\Delta u^{n+1}-\frac{r^{n+1}}{S_n}\frac1{\eps^2}f(u^n),\\
    \frac{r^{n+1}-r^n}{\tau}
    &=\frac{1}{2S_n}\left(\frac1{\eps^2}f(u^n),\frac{u^{n+1}-u^n}{\tau}\right),
\end{aligned}
\right.
\qquad n\ge0,
\label{eq:SAV-u}
\end{equation}
where
\begin{equation}
    S_n:=\left(\int_\Omega\frac1{\eps^2}F(u^n)\,dx+C_0\right)^{1/2},
    \qquad
    C_0>0,
    \qquad
    r^0=S_0.
    \label{eq:Sn}
\end{equation}
For later use, define the SAV scaling factor by
$\theta_{n+1}:=\frac{r^{n+1}}{S_n}.$
Then the phase equation in \eqref{eq:SAV-u} becomes
\begin{equation}
    \frac{u^{n+1}-u^n}{\tau}
    =\Delta u^{n+1}
    +\frac{\theta_{n+1}}{\eps^2}
    \bigl(u^n-(u^n)^3\bigr).
    \label{eq:SAV-monotone-form}
\end{equation}
The first-order SAV scheme \eqref{eq:SAV-u} is unconditionally stable with
respect to its standard modified energy; see \cite{shen2018scalar}. This
modified-energy property does not by itself control the sign of the pointwise
phase increment.

\subsubsection{Large-step failure of pointwise monotonicity for SAV}

\begin{theorem}[Large-step loss of pointwise monotonicity]
\label{thm:SAV-large-step-loss}
Consider the first-order SAV scheme \eqref{eq:SAV-u} under periodic or
homogeneous Neumann boundary conditions. There exist admissible monotone
data such that, for all sufficiently large time steps \(\tau\), the
numerical solution generated by the scheme fails to preserve pointwise
monotonicity. Consequently, the scheme is not unconditionally pointwise
monotone.
\end{theorem}

\begin{proof}
For simplicity, we construct the counterexample from a constant initial
datum. Fix \(u_0\in(0,1)\), set \(u^0(x)\equiv u_0\), and take
\(r^0=S_0\). The homogeneous data are preserved by \eqref{eq:SAV-u}, so
all phase iterates are independent of \(x\). Taking \(n=0\), using
\(r^0=S_0\), and eliminating \(r^1\) gives
\begin{equation}
    \theta_1
    =\left(1+
    \frac{\tau|\Omega|u_0^2(1-u_0^2)^2}{2\eps^4S_0^2}\right)^{-1}>0,
    \qquad
    u^1
    =u_0+\frac{\tau u_0(1-u_0^2)/\eps^2}
    {1+\tau|\Omega|u_0^2(1-u_0^2)^2/(2\eps^4S_0^2)}.
    \label{eq:SAV-u1-tau}
\end{equation}
Since
\[
    S_0^2=\frac{|\Omega|}{\eps^2}F(u_0)+C_0
    =\frac{|\Omega|}{4\eps^2}(1-u_0^2)^2+C_0,
\]
we have \(u^1>1\) whenever
\begin{equation}
    \tau>\tau_{*,\mathrm{SAV}}(u_0,\eps,C_0,|\Omega|)
    :=
    \frac{\eps^2\bigl(|\Omega|(1-u_0^2)^2+4\eps^2C_0\bigr)}
    {|\Omega|u_0(1-u_0)(1-u_0^2)^2
    +4u_0(1+u_0)\eps^2C_0}.
    \label{eq:SAV-thm-threshold}
\end{equation}
It follows that
$u^1-(u^1)^3<0,\, r^1=\theta_1S_0>0.
$
Combining the two SAV equations at \(n=1\) yields
\[
    \theta_2
    =\frac{2r^1S_1}
    {2S_1^2+\frac{\tau|\Omega|}{\eps^4}
    \bigl(u^1-(u^1)^3\bigr)^2}>0.
\]
Therefore, \eqref{eq:SAV-monotone-form} gives
$
    u^2-u^1=\frac{\tau\theta_2}{\eps^2}\bigl(u^1-(u^1)^3\bigr)<0.
$
Thus \(u^2(x)<u^1(x)\) for every \(x\in\Omega\) whenever \(\tau\) is
sufficiently large, which proves the claim.
\end{proof}
\begin{remark}[Extension beyond homogeneous data]
The strict homogeneous sign reversal in Theorem~\ref{thm:SAV-large-step-loss}
is likewise robust. The corresponding proof in Proposition~\ref{prop:SAV-nonhomogeneous-persistence} uses the coercivity of the SAV phase
update together with continuous dependence of the scalar auxiliary variable
to show that \(u^2<u^1\) persists for sufficiently small smooth
admissible perturbations satisfying \(\Delta u^0_\delta\not\equiv0\).
\end{remark}

\section{Persistence under spatially nonhomogeneous admissible perturbations}
\label{sec:nonhomogeneous-persistence}

The four counterexamples in the main text are first constructed on the
spatially homogeneous invariant class because the ODE reduction makes the
strict wrong-signed increments explicit. We now show that each strict
inequality persists under a sufficiently small smooth spatial perturbation.
Thus the counterexamples remain valid when the diffusion term is genuinely
active.

Throughout this section, $C_B^{2,\alpha}(\overline\Omega)$ denotes the
periodic H\"older space in the periodic case and the subspace of
$C^{2,\alpha}(\overline\Omega)$ satisfying $\partial_{\boldsymbol n}v=0$
on $\partial\Omega$ in the homogeneous Neumann case. Let
$\overline u_0\in(0,1)$ and choose a nonconstant
$\psi\in C_B^{2,\alpha}(\overline\Omega)$ such that
$\Delta\psi\not\equiv0$. For
$u^0_\delta=\overline u_0+\delta\psi,$
we have, as $\delta\to0$,
\begin{equation}
\begin{aligned}
    \|u^0_\delta-\overline u_0\|_{C^0(\overline\Omega)}&\longrightarrow0,\\
    \left\|\Reps(u^0_\delta)
    -\eps^{-2}\overline u_0(1-\overline u_0^2)\right\|_{C^0(\overline\Omega)}
    &\longrightarrow0.
\end{aligned}
\label{eq:common-perturbation-margin}
\end{equation}
Because the limiting residual is a strictly positive constant, there is a
number $\delta_{\rm adm}>0$ such that
\begin{equation}
    u^0_\delta\in\Aeps,
    \qquad
    \Delta u^0_\delta=\delta\Delta\psi\not\equiv0,
    \qquad 0<|\delta|<\delta_{\rm adm}.
\label{eq:common-perturbation-admissible}
\end{equation}
Consequently, Theorem~\ref{thm:exact-monotone-growth} implies that the exact
Allen--Cahn solution issued from $u^0_\delta$ is pointwise nondecreasing for
all later times. The remaining task in each case is to verify continuous
dependence of the numerical update on the perturbed data. The strict
homogeneous decrement is then preserved by uniform convergence.

\subsection{Second-order convex splitting}
\label{sec:CSmodCN-nonhomogeneous-persistence}

\begin{proposition}[Persistence of the CSmodCN sign reversal under nonhomogeneous perturbations]
\label{prop:CSmodCN-nonhomogeneous-persistence}
Let $\overline u_0\in(0,1)$, and fix a finite time step $\tau$ for which
the homogeneous construction in Theorem~\ref{thm:CSmodCN-nonmonotone}
produces
$\overline u^{3}-\overline u^{2}=-\eta_{\rm CS}<0.$
Set
\[
    u^0_\delta=\overline u_0+\delta\psi,
    \qquad
    u^1_\delta=\mathcal S_\tau u^0_\delta,
\]
where $\mathcal S_t$ is the exact Allen--Cahn solution operator, and compute
$u^2_\delta,u^3_\delta$ from \eqref{eq:CSmodCN}. Then there exists
$\delta_0>0$ such that, whenever $0<|\delta|<\delta_0$,
\begin{equation}
    u^0_\delta\in\Aeps,
    \qquad
    \Delta u^0_\delta\not\equiv0,
    \qquad
    \max_{x\in\overline\Omega}
    \bigl(u^3_\delta(x)-u^2_\delta(x)\bigr)<0.
\label{eq:CSmodCN-perturbation-conclusion}
\end{equation}
Thus the exact PDE trajectory is pointwise nondecreasing, whereas the
second-order convex splitting scheme produces a strictly wrong-signed
increment for genuinely nonhomogeneous data.
\end{proposition}

\begin{proof}
The first two assertions in \eqref{eq:CSmodCN-perturbation-conclusion}
follow from \eqref{eq:common-perturbation-admissible}. Continuous dependence
of the semilinear parabolic flow gives
\[
    u^1_\delta\longrightarrow\overline u^1
    \quad\text{in }C_B^{2,\alpha}(\overline\Omega)
    \quad\text{as }\delta\to0.
\]
For one convex splitting update, let $v$ denote the new value and let
$a,b$ denote the two preceding values. Define
\[
\mathcal F(v;a,b)
=\frac{v-a}{\tau}-\frac12\Delta(v+a)
-\frac1{\eps^2}
\left(\frac32a-\frac12b-D F_c(v,a)\right).
\]
The Fr\'echet derivative with respect to $v$ is
\begin{equation}
    D_v\mathcal F(v;a,b)w
    =-\frac12\Delta w+c(v,a)w,
    \qquad
    c(v,a):=\frac1\tau
    +\frac{2v^2+(v+a)^2}{4\eps^2}\ge\frac1\tau.
\label{eq:CSmodCN-linearized-operator}
\end{equation}
Hence $D_v\mathcal F$ is an isomorphism from
$C_B^{2,\alpha}(\overline\Omega)$ to
$C^{0,\alpha}(\overline\Omega)$ by periodic or Neumann Schauder theory.
Moreover, $\partial_vD F_c(v,a)\ge0$, so the nonlinear update operator is
strongly monotone because of the term $\tau^{-1}v$; in particular, the
update is unique. The implicit-function theorem applied at each homogeneous
reference update therefore gives continuous dependence on $(a,b)$.
Applying it successively at $n=1$ and $n=2$ yields
\[
    u^2_\delta\longrightarrow\overline u^2,
    \qquad
    u^3_\delta\longrightarrow\overline u^3
    \quad\text{uniformly on }\overline\Omega.
\]
Therefore,
\[
    \left\|(u^3_\delta-u^2_\delta)
    -(\overline u^3-\overline u^2)\right\|_{L^\infty(\Omega)}
    \longrightarrow0.
\]
Choosing $\delta_0\le\delta_{\rm adm}$ so that the last norm is smaller
than $\eta_{\rm CS}/2$ proves
\eqref{eq:CSmodCN-perturbation-conclusion}.
\end{proof}

\subsection{Stabilized CN/AB scheme}
\label{sec:stab-CNAB-nonhomogeneous-persistence}

\begin{proposition}[Persistence of the stabilized CN/AB sign reversal under nonhomogeneous perturbations]
\label{prop:stab-CNAB-nonhomogeneous-persistence}
Let $\overline u_0\in(0,1)$, fix $S\ge0$, and choose a finite
$\tau>\tau_{*,\mathrm{stab}}(\overline u_0,\eps)$. Denote the strict
homogeneous decrement from Theorem~\ref{thm:stab-CNAB-nonmonotone} by
$\overline u^2-\overline u^1=-\eta_{\rm stab}<0.$
Set
\[
    u^0_\delta=\overline u_0+\delta\psi,
    \qquad
    u^1_\delta=\mathcal S_\tau u^0_\delta,
\]
and let $u^2_\delta$ be generated by \eqref{eq:stab-CNAB}. Then there
exists $\delta_0>0$ such that, whenever $0<|\delta|<\delta_0$,
\begin{equation}
    u^0_\delta\in\Aeps,
    \qquad
    \Delta u^0_\delta\not\equiv0,
    \qquad
    \max_{x\in\overline\Omega}
    \bigl(u^2_\delta(x)-u^1_\delta(x)\bigr)<0.
\label{eq:stab-CNAB-perturbation-conclusion}
\end{equation}
Hence the stabilized CN/AB sign reversal persists when diffusion is active.
\end{proposition}

\begin{proof}
By \eqref{eq:common-perturbation-admissible}, the perturbed datum is
admissible and has nonzero diffusion for sufficiently small nonzero
$\delta$. The exact starter satisfies $u^1_\delta\ge u^0_\delta$, and
parabolic continuous dependence gives
$u^1_\delta\to\overline u^1$ in
$C_B^{2,\alpha}(\overline\Omega)$.
Put
\[
    c_{\tau,S}:=\frac1\tau+\frac{S\tau}{\eps^2}>0,
    \qquad
    \mathcal L_{\tau,S}:=c_{\tau,S}I-\frac12\Delta.
\]
The update equation can be written as
\[
\mathcal L_{\tau,S}u^2_\delta
=c_{\tau,S}u^1_\delta+\frac12\Delta u^1_\delta
-\frac1{\eps^2}
\left(\frac32f(u^1_\delta)-\frac12f(u^0_\delta)\right).
\]
The operator $\mathcal L_{\tau,S}$ is an isomorphism from
$C_B^{2,\alpha}(\overline\Omega)$ to
$C^{0,\alpha}(\overline\Omega)$, and the right-hand side depends
continuously on $(u^0_\delta,u^1_\delta)$. Hence
$u^2_\delta\to\overline u^2$ uniformly. It follows that
\[
    \left\|(u^2_\delta-u^1_\delta)
    -(\overline u^2-\overline u^1)\right\|_{L^\infty(\Omega)}
    \longrightarrow0.
\]
Taking $\delta_0\le\delta_{\rm adm}$ so that this norm is below
$\eta_{\rm stab}/2$ proves
\eqref{eq:stab-CNAB-perturbation-conclusion}.
\end{proof}

\paragraph{An explicit periodic perturbation}
On $\Omega=(0,2\pi)$, take $\eps=\frac12$,
$\overline u_0=\frac12$, and $\psi(x)=\cos x$. Then
\[
    u^0_\delta(x)=\frac12+\delta\cos x,
    \qquad
    \Reps(u^0_\delta)
    =\frac32-6\delta^2\cos^2x-4\delta^3\cos^3x.
\]
For $0<\delta\le1/4$,
\[
    \Reps(u^0_\delta)(x)
    \ge \frac32-6\delta^2-4\delta^3
    \ge\frac{17}{16}>0,
\]
while $\Delta u^0_\delta=-\delta\cos x\not\equiv0$. For $S=5$ and
$\tau=0.5$, with the exact PDE starter, the homogeneous decrement is
approximately $-0.0371448$. The preceding continuity argument therefore
guarantees the same negative sign for every sufficiently small positive
$\delta$. This is a local perturbation result near $\delta=0$; without an
additional quantitative Lipschitz estimate, it does not assert that the
endpoint $\delta=1/4$ lies inside the perturbative neighborhood.

\subsection{IEQ scheme}
\label{sec:IEQ-nonhomogeneous-persistence}

\begin{proposition}[Persistence of the IEQ sign reversal under nonhomogeneous perturbations]
\label{prop:IEQ-nonhomogeneous-persistence}
Let $\overline u_0\in(0,1)$, and fix a finite
$\tau>\tau_{*,\mathrm{IEQ}}
(\overline u_0,\eps,C_{\mathrm{IEQ}})$. Let
$\overline u^2-\overline u^1=-\eta_{\rm IEQ}<0$ be the strict homogeneous
decrement from Theorem~\ref{thm:IEQ-large-step-loss}. Initialize
\[
    u^0_\delta=\overline u_0+\delta\psi,
    \qquad
    q^0_\delta=\sqrt{F(u^0_\delta)+C_{\rm IEQ}},
\]
and generate two IEQ steps from \eqref{eq:IEQ-u-scheme}. Then there exists
$\delta_0>0$ such that, whenever $0<|\delta|<\delta_0$,
\begin{equation}
    u^0_\delta\in\Aeps,
    \qquad
    \Delta u^0_\delta\not\equiv0,
    \qquad
    \max_{x\in\overline\Omega}
    \bigl(u^2_\delta(x)-u^1_\delta(x)\bigr)<0.
\label{eq:IEQ-perturbation-conclusion}
\end{equation}
Thus the IEQ loss of pointwise monotonicity also occurs for genuinely
nonhomogeneous admissible data.
\end{proposition}

\begin{proof}
The admissibility and nonzero-diffusion assertions follow from
\eqref{eq:common-perturbation-admissible}; hence the exact Allen--Cahn
solution from the same datum is pointwise nondecreasing.

For the IEQ iterates, define pointwise
\[
    Q_n=\sqrt{F(u^n)+C_{\rm IEQ}},
    \qquad H_n=\frac{f(u^n)}{Q_n},
    \qquad
    a_n(x):=\frac1\tau+\frac{H_n(x)^2}{2\eps^2}.
\]
Since $F\ge0$ and $C_{\rm IEQ}>0$, one has
$Q_n\ge\sqrt{C_{\rm IEQ}}$, so the Nemytskii maps
$u^n\mapsto Q_n$ and $u^n\mapsto H_n$ are smooth near the homogeneous
orbit. Eliminating $q^{n+1}$ from \eqref{eq:IEQ-u-scheme} gives
\begin{equation}
    \bigl(a_n(x)I-\Delta\bigr)u^{n+1}
    =a_n(x)u^n-\frac1{\eps^2}H_nq^n,
    \qquad
    q^{n+1}=q^n+\frac12H_n(u^{n+1}-u^n).
\label{eq:IEQ-eliminated-perturbation}
\end{equation}
The coefficient satisfies $a_n(x)\ge1/\tau$. Thus the elliptic operator
in \eqref{eq:IEQ-eliminated-perturbation} is uniformly coercive and, by
periodic or Neumann Schauder theory, is an isomorphism from
$C_B^{2,\alpha}(\overline\Omega)$ to
$C^{0,\alpha}(\overline\Omega)$. Starting from
$(u^0_\delta,q^0_\delta)\to(\overline u^0,\overline q^0)$ and applying
this continuous solution map successively at $n=0$ and $n=1$ gives
\[
    (u^1_\delta,q^1_\delta)\longrightarrow
    (\overline u^1,\overline q^1),
    \qquad
    (u^2_\delta,q^2_\delta)\longrightarrow
    (\overline u^2,\overline q^2)
\]
in particular uniformly on $\overline\Omega$. The strict margin
$-\eta_{\rm IEQ}$ therefore persists after decreasing
$\delta_0\le\delta_{\rm adm}$, which proves
\eqref{eq:IEQ-perturbation-conclusion}.
\end{proof}

\subsection{SAV scheme}
\label{sec:SAV-nonhomogeneous-persistence}

\begin{proposition}[Persistence of the SAV sign reversal under nonhomogeneous perturbations]
\label{prop:SAV-nonhomogeneous-persistence}
Let $\overline u_0\in(0,1)$, and fix a finite
$\tau>\tau_{*,\mathrm{SAV}}
(\overline u_0,\eps,C_0,|\Omega|)$. Let
$\overline u^2-\overline u^1=-\eta_{\rm SAV}<0$ denote the strict
homogeneous decrement from Theorem~\ref{thm:SAV-large-step-loss}. Initialize
\[
    u^0_\delta=\overline u_0+\delta\psi,
    \qquad
    r^0_\delta=
    \left(\int_\Omega\frac1{\eps^2}F(u^0_\delta)\,dx+C_0\right)^{1/2},
\]
and generate two SAV steps from \eqref{eq:SAV-u}. Then there exists
$\delta_0>0$ such that, whenever $0<|\delta|<\delta_0$,
\begin{equation}
    u^0_\delta\in\Aeps,
    \qquad
    \Delta u^0_\delta\not\equiv0,
    \qquad
    \max_{x\in\overline\Omega}
    \bigl(u^2_\delta(x)-u^1_\delta(x)\bigr)<0.
\label{eq:SAV-perturbation-conclusion}
\end{equation}
Hence the SAV sign reversal is stable under small nonhomogeneous
admissible perturbations.
\end{proposition}

\begin{proof}
For sufficiently small nonzero $\delta$, the exact PDE solution from
$u^0_\delta$ is pointwise nondecreasing by
\eqref{eq:common-perturbation-admissible} and
Theorem~\ref{thm:exact-monotone-growth}.
Define
\[
    b_n:=\frac{1}{\eps^2S_n}f(u^n).
\]
Since $F\ge0$ and $C_0>0$, $S_n\ge\sqrt{C_0}$ and $b_n$ depends
continuously on $u^n$. Eliminating $r^{n+1}$ from \eqref{eq:SAV-u} gives
\begin{equation}
\begin{aligned}
    \left(\frac1\tau I-\Delta\right)u^{n+1}
    +\frac12(b_n,u^{n+1})b_n
    ={}&\frac1\tau u^n-r^n b_n
       +\frac12(b_n,u^n)b_n,\\
    r^{n+1}={}&r^n+\frac12(b_n,u^{n+1}-u^n),
\end{aligned}
\label{eq:SAV-eliminated-perturbation}
\end{equation}
where $(\cdot,\cdot)$ is the $L^2(\Omega)$ inner product. The bilinear
form of the phase equation satisfies
\[
    \frac1\tau\|v\|_{L^2}^2
    +\|\nabla v\|_{L^2}^2
    +\frac12(b_n,v)^2
    \ge\frac1\tau\|v\|_{L^2}^2.
\]
It is therefore uniformly coercive. Lax--Milgram, followed by periodic or
Neumann Schauder regularity, shows that the phase update is uniquely
solvable in $C_B^{2,\alpha}(\overline\Omega)$ and depends continuously on
$(u^n,r^n)$. The scalar update in
\eqref{eq:SAV-eliminated-perturbation} is continuous as well. Starting from
$(u^0_\delta,r^0_\delta)\to(\overline u^0,\overline r^0)$ and applying the
update twice yields
\[
    u^1_\delta\longrightarrow\overline u^1,
    \qquad
    u^2_\delta\longrightarrow\overline u^2
    \quad\text{uniformly on }\overline\Omega.
\]
The strict margin $-\eta_{\rm SAV}$ is therefore preserved after choosing
$\delta_0\le\delta_{\rm adm}$ sufficiently small, which proves
\eqref{eq:SAV-perturbation-conclusion}.
\end{proof}

\section{Numerical experiments}
\label{sec:num}

The numerical experiments are organized in three stages. First, we use
two positive one-dimensional cosine profiles from the introductory example
to visualize the initial velocity
$
    u_t(x,0)=\Reps(u_0)(x)
$
of the continuous Allen--Cahn equation. One profile has a residual that is
positive at every spatial point, whereas the other has a sign-changing
residual. This initial-time diagnostic shows that the range condition
$0<u_0<1$ alone does not determine the pointwise direction of the PDE flow.

Second, using only the residual-positive profile, we compare the different
time discretizations in one spatial dimension. These tests examine whether
the fully implicit Euler method reproduces the admissible pointwise structure,
whether energy-stable mechanisms can distort the physical time scale without
reversing the increment, and whether modified-energy stability or
maximum-bound preservation can coexist with an incorrect pointwise time
direction.

Third, we repeat the scheme comparisons for a genuinely two-dimensional
nonhomogeneous datum on $(0,2\pi)^2$. Two spatial grids, $128^2$ and $256^2$,
are employed for every method in order to verify that the observed increment
patterns persist when both spatial directions are active and are not artifacts
of a particular spatial resolution. Together, the one- and two-dimensional
experiments distinguish preservation of the pointwise time direction from
energy stability, maximum-bound preservation, and modified-energy
dissipation.

\subsection{Initial-velocity diagnostic for the continuous problem}
\label{subsec:continuous-initial-velocity}

We first isolate a property of the continuous Allen--Cahn equation that is
essential for interpreting the later numerical results. On
$\Omega=(0,2\pi)$, with periodic boundary conditions and $\eps=0.5$, consider
the two positive cosine profiles
\begin{equation}
    u_0^+(x)=\frac12+\frac14\cos x,
    \qquad
    u_0^{\rm mix}(x)=\frac12+\frac14\cos(3x),
    \qquad x\in(0,2\pi).
    \label{eq:numerical-initial-data-pair}
\end{equation}
Both profiles satisfy
$
    \frac14\leq u_0^+(x),\,u_0^{\rm mix}(x)\leq\frac34<1.
$
They nevertheless have different initial PDE velocities. For the first
profile,
\[
    \Reps(u_0^+)(x)
    =-\frac14\cos x+4\left(\frac12+\frac14\cos x
      -\left(\frac12+\frac14\cos x\right)^3\right)
    \geq \frac{17}{16}>0.
\]
Hence $u_0^+\in\Aeps$. In contrast,
\begin{align*}
    \Reps(u_0^{\rm mix})(x)
    &=-\frac94\cos(3x)+4\left(\frac12+\frac14\cos(3x)
      -\left(\frac12+\frac14\cos(3x)\right)^3\right) \\
    &=\frac32-2\cos(3x)-\frac38\cos^2(3x)-\frac1{16}\cos^3(3x),
\end{align*}
so that
$
    \Reps(u_0^{\rm mix})(0)=-\frac{15}{16}<0,
    \,
    \Reps(u_0^{\rm mix})\left(\frac{\pi}{3}\right)=\frac{51}{16}>0.
$
Thus $u_0^{\rm mix}\notin\Aeps$: although it remains strictly between $0$
and $1$, its initial velocity is negative near the spatial crests and positive
near the troughs.

Figure~\ref{fig:initial-velocity-sign-contrast-final} plots this initial-time
contrast directly. The left panel shows that both initial profiles stay in
the same positive range. The right panel shows the corresponding residuals,
with the zero line marking the sign of the initial velocity. The datum
$u_0^+$ has a positive residual everywhere, whereas $u_0^{\rm mix}$ has both
positive and negative residual regions. Thus the PDE direction is selected by
the full residual rather than by the scalar range condition alone.

\begin{figure}[!htbp]
    \centering
    \includegraphics[width=0.96\textwidth]{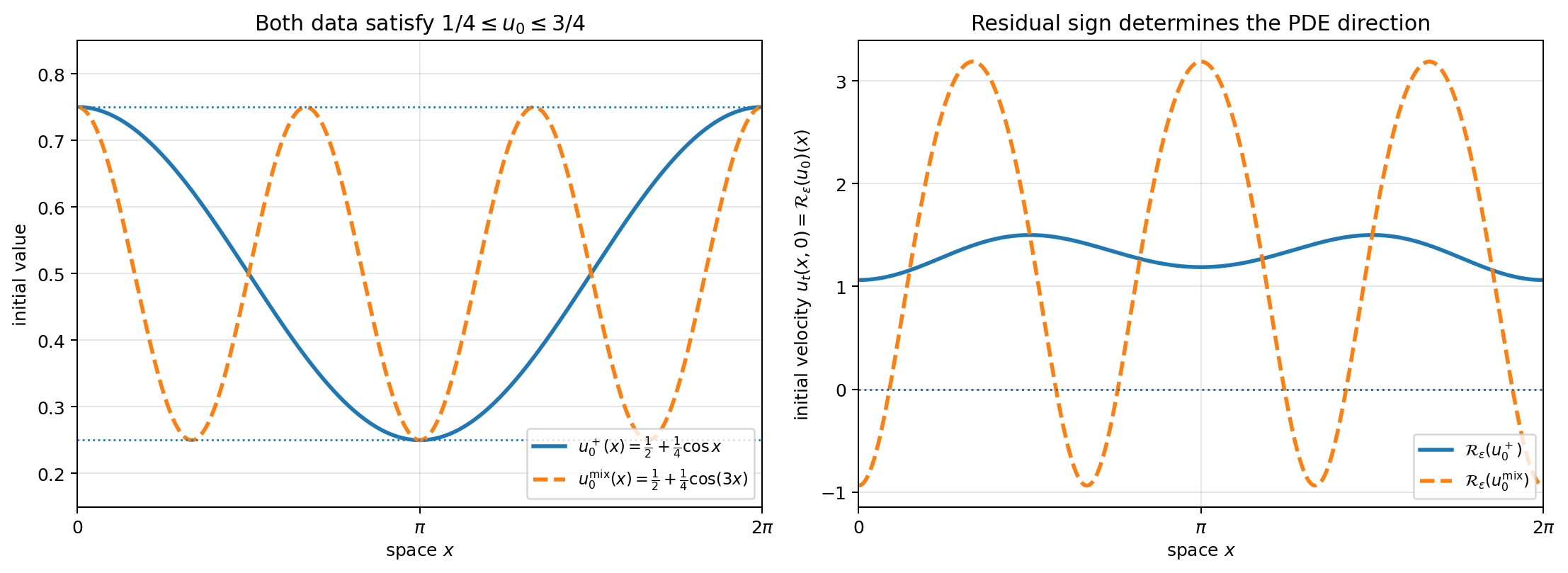}
    \caption{Initial-velocity sign contrast for two positive cosine profiles with $\eps=0.5$. Left: both $u_0^+$ and $u_0^{\rm mix}$ remain in the interval $[1/4,3/4]$. Right: the residual $\Reps(u_0^+)$ is strictly positive, whereas $\Reps(u_0^{\rm mix})$ changes sign; equivalently, the initial PDE velocity $u_t(x,0)$ has both positive and negative regions for the mixed-residual datum.}
    \label{fig:initial-velocity-sign-contrast-final}
\end{figure}

This comparison concerns only the continuous equation at the initial time and
does not involve a temporal discretization. Its role is to identify the
residual-positive datum used in the subsequent numerical comparisons. In
particular, a negative discrete increment observed later cannot be attributed
to an initially mixed pointwise direction of the continuous flow.

\subsection{One-dimensional numerical experiments}
\label{subsec:one-dimensional-experiments}

\subsubsection{Common setting and diagnostics}

All one-dimensional scheme-comparison experiments are performed on
$\Omega=(0,2\pi)$ with periodic boundary conditions. The spatial Laplacian is
approximated by the standard second-order periodic finite-difference operator
on a uniform grid with
$
    N=160, \,
    h=\frac{2\pi}{N},\,
    \eps=0.5.
$
For all these experiments, we use only the residual-positive datum
\begin{equation}
    u_0(x)=u_0^+(x)=\frac12+\frac14\cos x,
    \qquad x\in(0,2\pi),
    \label{eq:numerical-initial-data-final}
\end{equation}
introduced in Subsection~\ref{subsec:continuous-initial-velocity}. This is the
$a=1/2$, $b=1/4$, $k=1$ member of the introductory family, and it satisfies
$
    \Reps(u_0)(x)\geq\frac{17}{16}>0.
$
Consequently, the corresponding continuous solution is pointwise
nondecreasing, and any negative increment observed below represents a
numerical reversal of the admissible time direction rather than a consequence
of a sign-changing initial residual.

For a computed solution \(\{u_j^n\}\), we use the diagnostics
\[
    m_\Delta:=\min_{n,j}\bigl(u_j^{n+1}-u_j^n\bigr),
    \qquad
    U_{\min}:=\min_{n,j}u_j^n,
    \qquad
    U_{\max}:=\max_{n,j}u_j^n.
\]
A positive value of \(m_\Delta\) indicates pointwise monotone growth, whereas
\(m_\Delta<0\) records a local reversal of the time direction. For two-step
schemes, the first value is generated by a highly resolved fully implicit
Euler starter over the first nominal step. This guarantees that the starting
pair is already in the positive monotone direction, so any later negative
increment is produced by the tested two-step update itself. The negative-branch
case is omitted because the Allen--Cahn equation is odd under \(u\mapsto-u\),
and therefore gives the corresponding monotone-decay tests by symmetry.

\subsubsection{Reference computation: fully implicit Euler}

We first give a fully implicit Euler computation as a monotone reference. This method is not used here as another energy-stable comparison scheme; it serves as the baseline because the theory predicts pointwise monotone growth in the unique-solvability regime. We take
$\tau=0.03, T=3 .$
Since \(\tau<\eps^2=0.25\), the computation lies in the convexity regime used in the monotonicity theorem. The discrete initial residual is positive,
$
    \min_x R_h^0(x)\approx 1.062532>0,
$
and the computed increments satisfy
$
    m_\Delta
    =\min_{n,j}\left(u_j^{n+1}-u_j^n\right)
    \approx 1.54\times10^{-10}>0 .
$
Thus the fully implicit Euler surface is pointwise nondecreasing in time. The corresponding reference computation is displayed in Figure~\ref{fig:ie-monotone-surface-final}; the surface has no visible temporal fold or local decrease and will be used as the monotone benchmark below.

\begin{figure}[!htbp]
    \centering
    \includegraphics[width=0.450\textwidth]{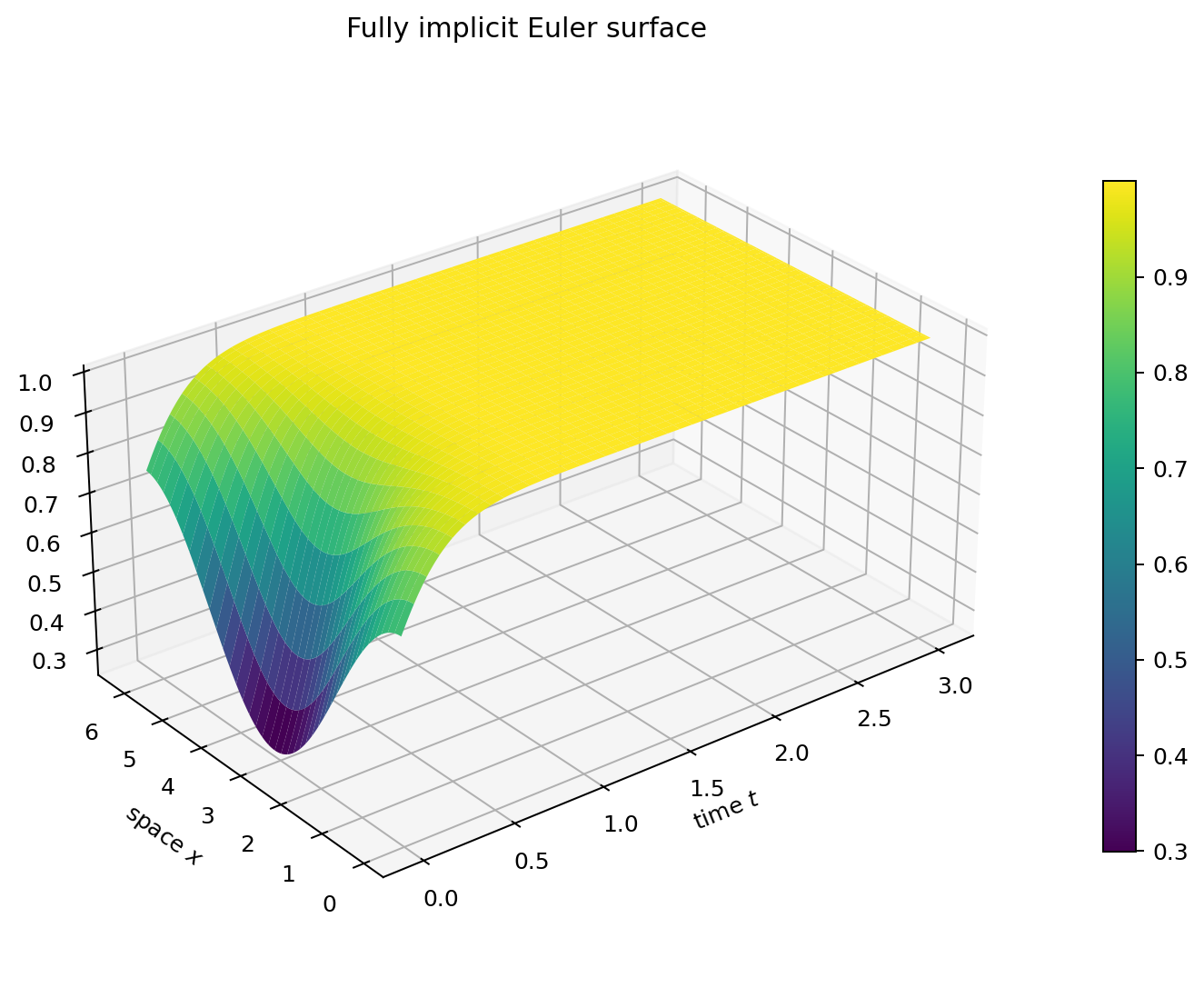}
    \caption{Fully implicit Euler reference surface for the admissible datum \eqref{eq:numerical-initial-data-final}. The computation uses \(\eps=0.5\) and \(\tau=0.03\), and the surface is pointwise nondecreasing in time.}
    \label{fig:ie-monotone-surface-final}
\end{figure}

\subsubsection{Monotone but delayed dynamics: CSS1 and first-order stabilization}

The first group of experiments concerns schemes that retain the pointwise direction in the tested stable regime, but evolve on a reduced effective time scale for large steps. This is the numerical counterpart of the delay mechanisms proved for CSS1 and for the strongly stabilized first-order semi-implicit method.

\paragraph{CSS1}
We compare CSS1 at
$T=3, \tau_{\rm s}=0.05,\tau_{\rm l}=0.5 .$
The effective-time compression factors are
\[
    \alpha_{\rm s}=\frac{\eps^2}{\eps^2+\tau_{\rm s}}=0.833333,
    \qquad
    \alpha_{\rm l}=\frac{\eps^2}{\eps^2+\tau_{\rm l}}=0.333333.
\]
Both runs remain pointwise monotone:
\[
    m_{\Delta,{\rm s}}
    \approx 1.27\times10^{-8}>0,
    \qquad
    m_{\Delta,{\rm l}}
    \approx 3.47\times10^{-3}>0.
\]
The large-step defect is therefore not sign reversal, but visible delay caused by the compressed effective time. This distinction between preservation of the pointwise direction and distortion of the transient time scale is directly visible in the two surfaces in Figure~\ref{fig:css1-step-compare-final}.

\begin{figure}[!htbp]
    \centering
    \begin{minipage}[b]{0.45\textwidth}
        \centering
        \includegraphics[width=0.96\textwidth]{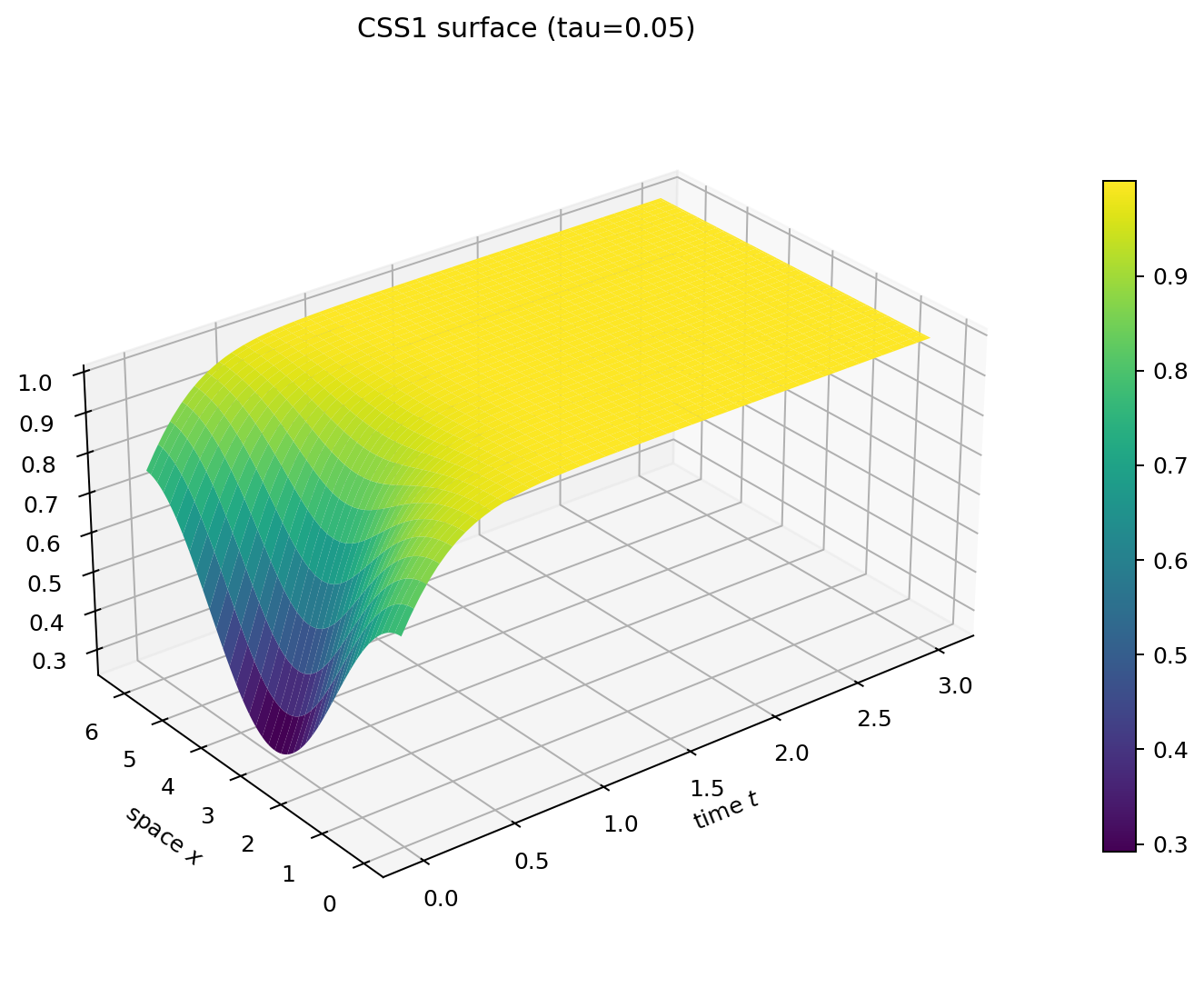}
    \end{minipage}\hfill
    \begin{minipage}[b]{0.45\textwidth}
        \centering
        \includegraphics[width=0.96\textwidth]{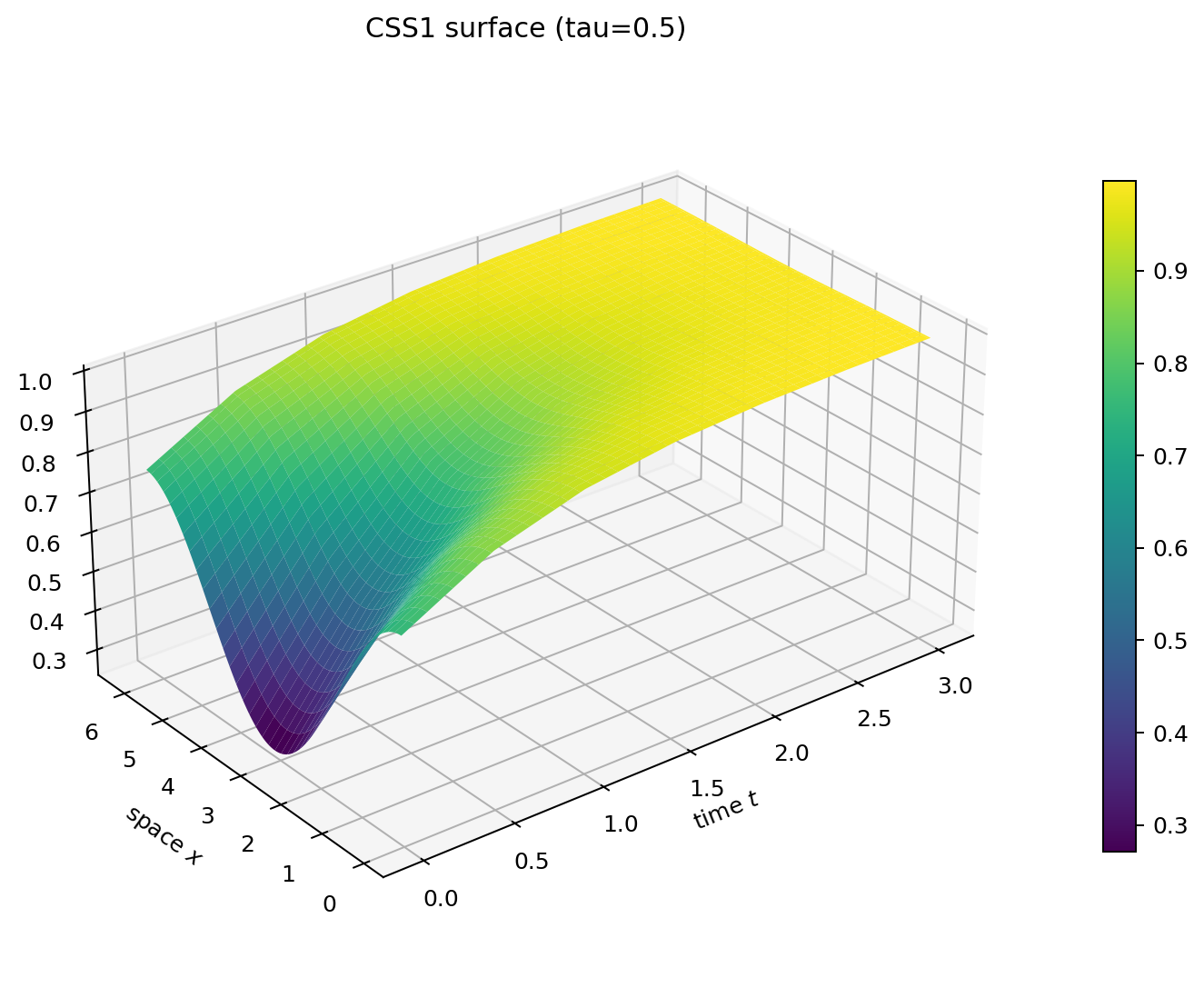}
    \end{minipage}
    \caption{CSS1 comparison with the common admissible datum. Left: small-step run with \(\tau=0.05\), which remains monotone and follows the nominal transient dynamics more closely. Right: large-step run with \(\tau=0.5\), which is still monotone but is substantially delayed.}
    \label{fig:css1-step-compare-final}
\end{figure}

To quantify this delay, we compare the point trace at \(x=0\) for the large-step CSS1 run with a small-step fully implicit Euler reference at the nominal time \(t\) and at the delayed time \(t_{\rm eff}=\alpha_{\rm l}t\).

\begin{figure}[!htbp]
    \centering
    \includegraphics[width=0.45\textwidth]{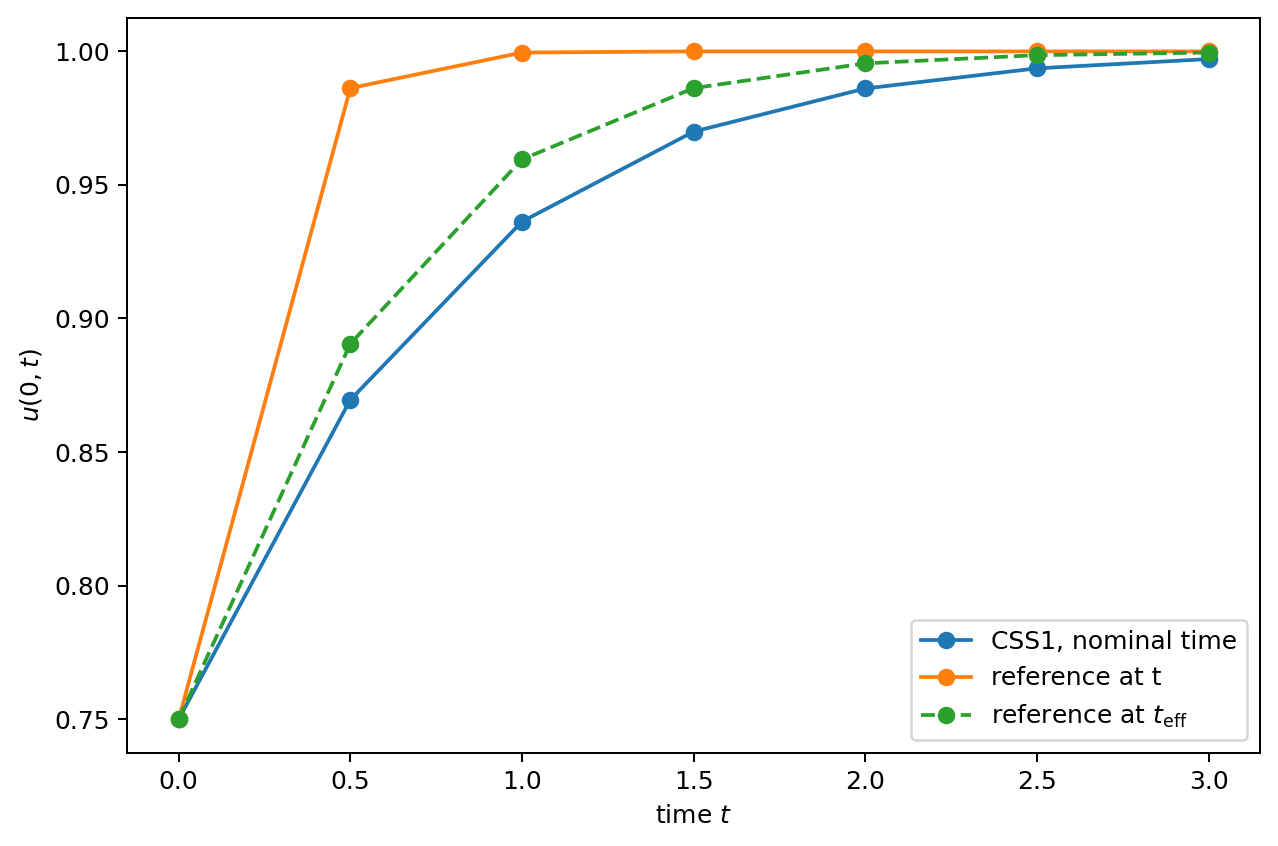}
    \caption{Point trace at \(x=0\) for the large-step CSS1 run. The CSS1 trajectory is closer to the delayed-time reference than to the nominal-time reference for this parameter set.}
    \label{fig:css1-point-delay-final}
\end{figure}

At \(T=3\), the \(L^2\)-error against the nominal-time reference is
$
    \|u_{\rm CSS1}(T)-u_{\rm ref}(T)\|_2
    \approx 1.412\times10^{-2},
$
whereas the error against the delayed-time reference is
$
    \|u_{\rm CSS1}(T)-u_{\rm ref}(T_{\rm eff})\|_2
    \approx 1.207\times10^{-2}.
$
This confirms that CSS1 evolves closer to the delayed effective trajectory than to the nominal physical-time trajectory, as is also seen from the point traces in Figure~\ref{fig:css1-point-delay-final}.

\paragraph{First-order stabilized semi-implicit scheme}
We next test the strongly stabilized first-order semi-implicit method with
$S=32,
    T=2,
    \tau_{\rm s}=0.05,
    \tau_{\rm l}=0.5 .$
The damping factors are
\[
    \beta_{\rm s}=\frac1{1+S\tau_{\rm s}}=0.384615,
    \qquad
    \beta_{\rm l}=\frac1{1+S\tau_{\rm l}}=0.058824.
\]
Both computations preserve the pointwise direction:
$
    m_{\Delta,{\rm s}}
    \approx 2.28\times10^{-4}>0,
    \,
    m_{\Delta,{\rm l}}
    \approx 2.46\times10^{-2}>0 .
$
The large-step surface is nevertheless strongly damped, in agreement with the theoretical large-step limit for this method. The contrast between the small-step and large-step transients is evident in Figure~\ref{fig:stab-step-compare-final}.

\begin{figure}[!htbp]
    \centering
    \begin{minipage}[b]{0.48\textwidth}
        \centering
        \includegraphics[width=0.96\textwidth]{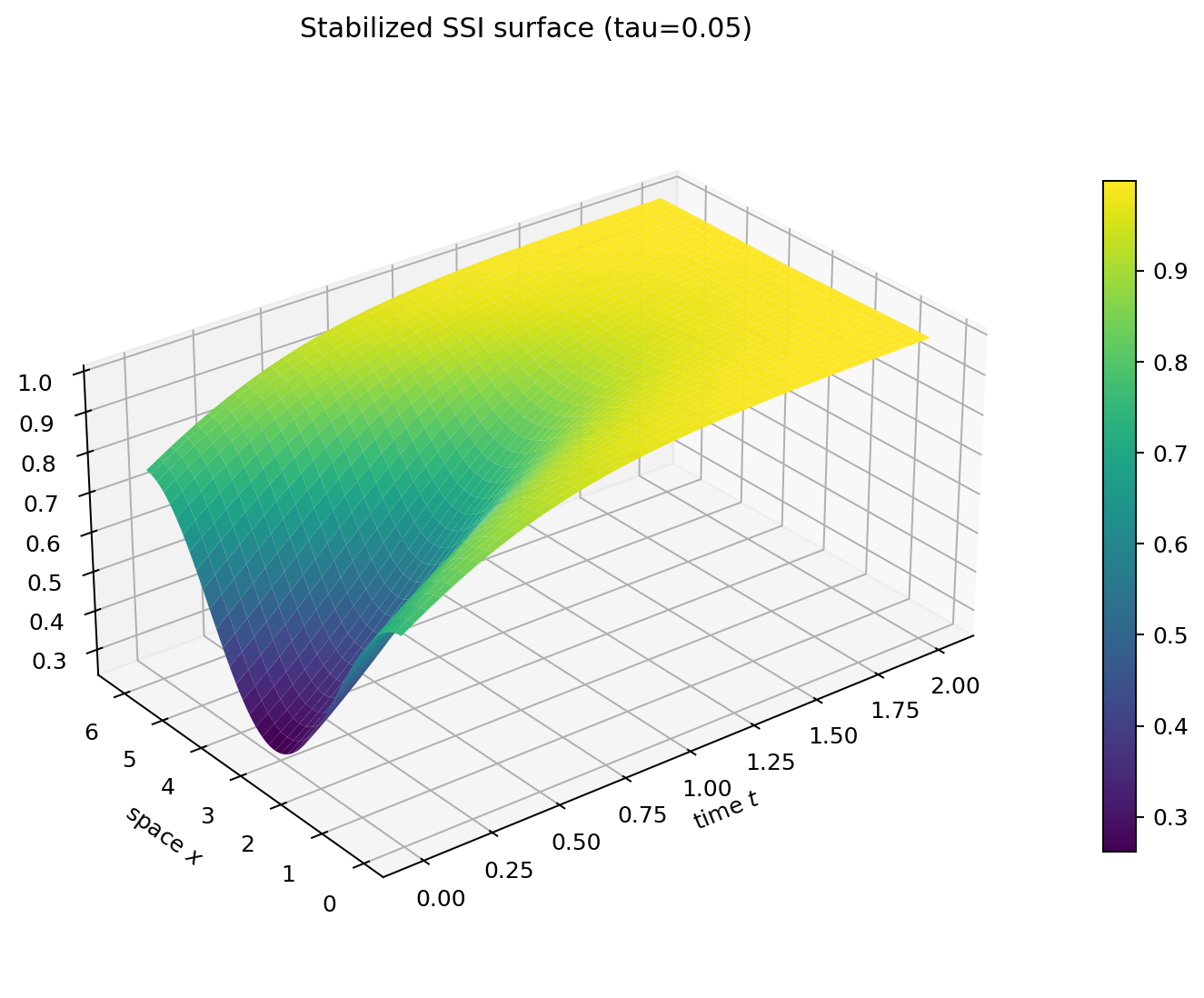}
    \end{minipage}\hfill
    \begin{minipage}[b]{0.48\textwidth}
        \centering
        \includegraphics[width=0.96\textwidth]{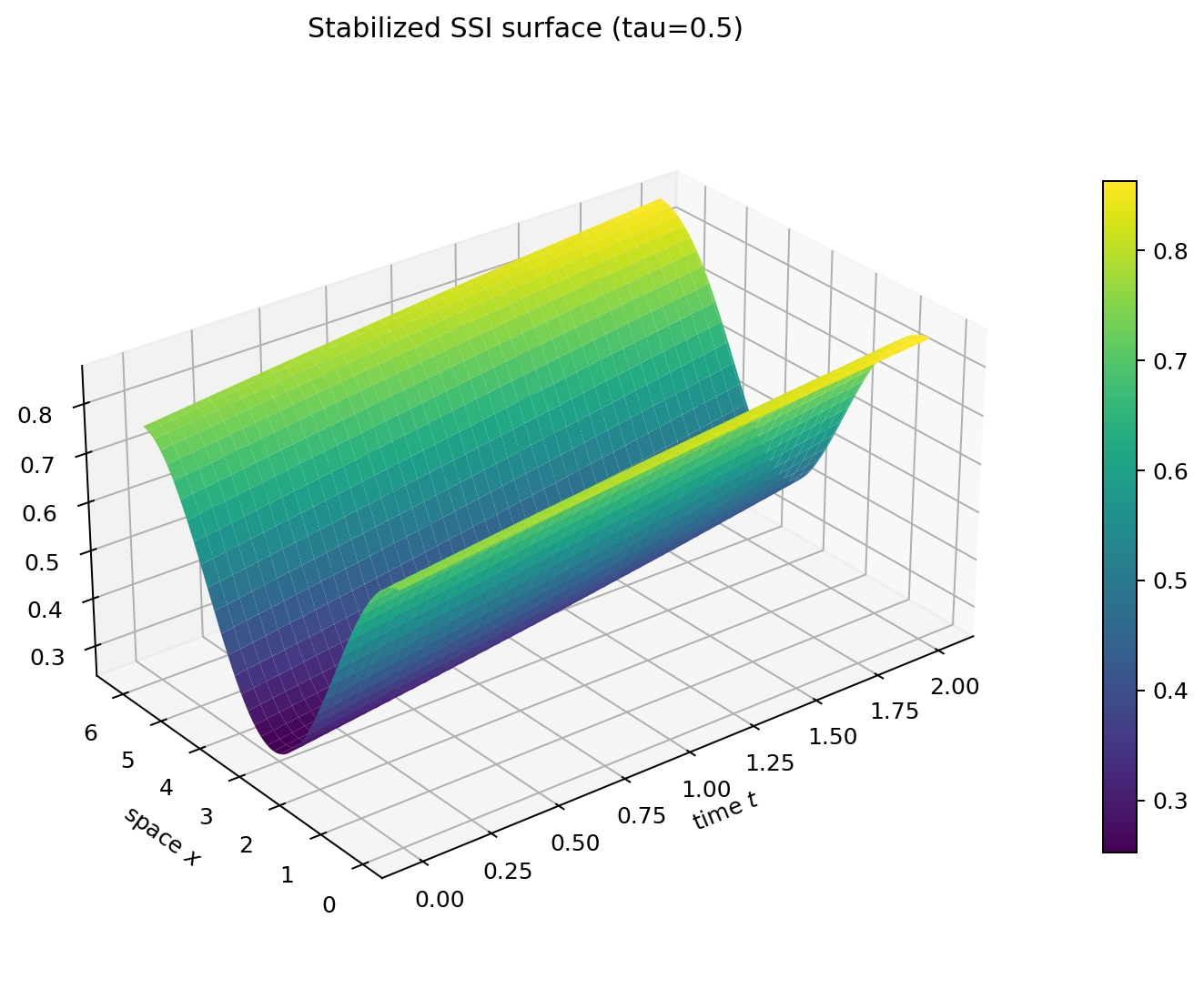}
    \end{minipage}
    \caption{First-order stabilized semi-implicit comparison with \(S=32\). Left: small-step run with \(\tau=0.05\). Right: large-step run with \(\tau=0.5\), which remains monotone but exhibits a much stronger artificial delay.}
    \label{fig:stab-step-compare-final}
\end{figure}

The delay is also visible in the point trace at \(x=0\). We compare the large-step stabilized trace with the reference trajectory at the nominal time and at the damped time \(t_{\rm eff}=\beta_{\rm l}t\).

\begin{figure}[!htbp]
    \centering
    \includegraphics[width=0.45\textwidth]{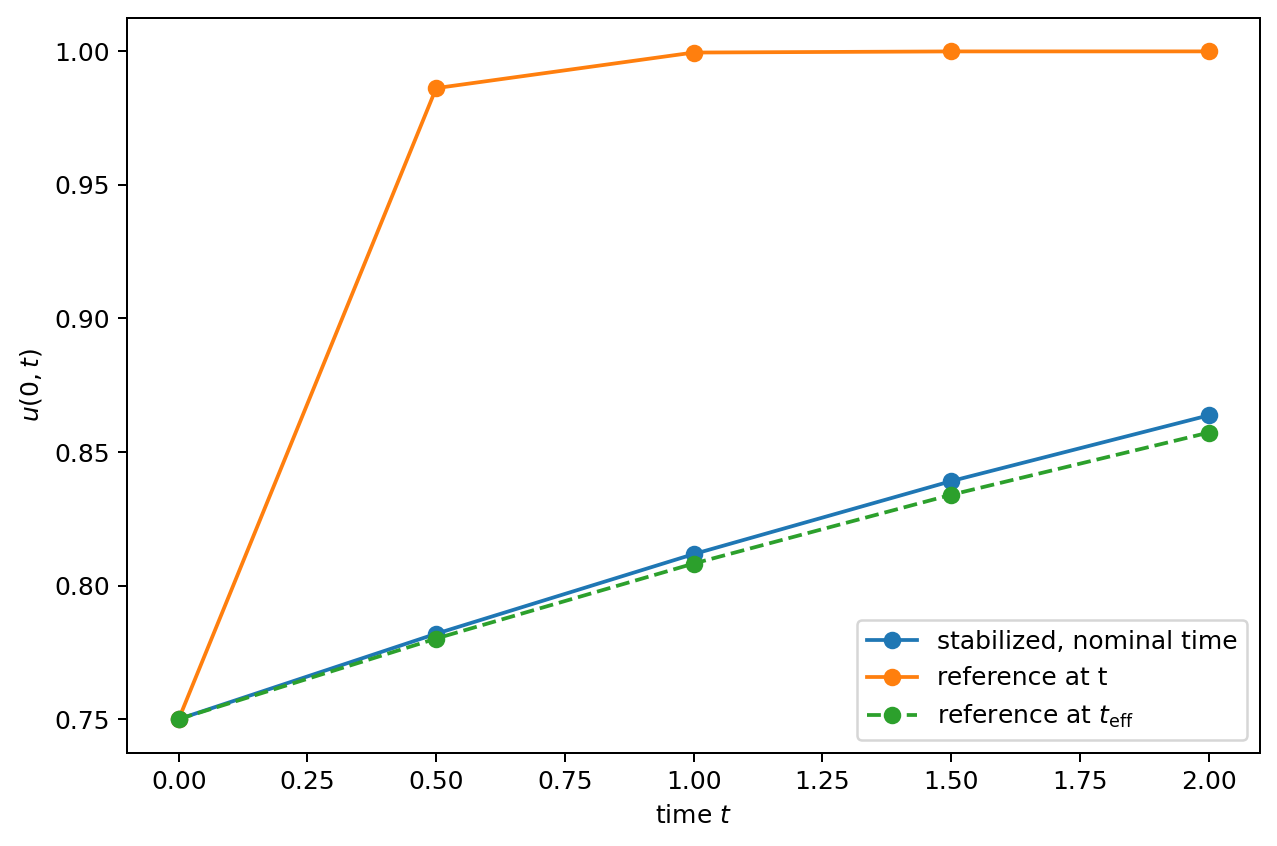}
    \caption{Point trace at \(x=0\) for the first-order stabilized scheme with a large time step. The numerical trajectory is consistent with the damped effective-time dynamics rather than with the solution evaluated at the nominal physical time.}
    \label{fig:stab-point-delay-final}
\end{figure}

At \(T=2\), the \(L^2\)-error against the nominal-time reference is
$
    \|u_{\rm stab}(T)-u_{\rm ref}(T)\|_2
    \approx 9.510\times10^{-1},
$
whereas the error against the damped-time reference is
$
    \|u_{\rm stab}(T)-u_{\rm ref}(T_{\rm eff})\|_2
    \approx 1.373\times10^{-2}.
$
Thus sufficiently strong first-order stabilization yields unconditional energy stability and preserves the tested monotone direction, but the price is a pronounced distortion of the physical time scale; the effective-time interpretation is quantitatively illustrated by Figure~\ref{fig:stab-point-delay-final}.

\subsubsection{Two-step schemes: loss of pointwise monotonicity}

The next two experiments concern two-step schemes. In both cases the first value is generated by a highly resolved fully implicit Euler starter, so the starting pair has the correct monotone direction. The observed negative increment is therefore a property of the two-step update.

\paragraph{Second-order convex splitting}
For the CS-modCN scheme \eqref{eq:CSmodCN}, we take
$\tau_{\rm s}=0.01,
    T_{\rm s}=2,$
for the small-step run and
$\tau_{\rm l}=1,
    T_{\rm l}=3,$
for the large-step run. The small-step computation remains pointwise increasing:
$
    m_{\Delta,{\rm s}}
    \approx 1.55\times10^{-8}>0 .
$
For the large step, the numerical solution overshoots above the stable level and then decreases:
$
    U_{\max}\approx 1.188961,
    \,
    m_{\Delta,{\rm l}}
    \approx -1.87\times10^{-1}<0 .
$
This agrees with the theory: modified-energy stability of the second-order convex splitting scheme does not prevent loss of pointwise monotonicity. The transition from monotone growth to overshoot followed by decrease is displayed explicitly in Figure~\ref{fig:csmodcn-step-compare-final}.

\begin{figure}[!htbp]
    \centering
    \begin{minipage}[b]{0.45\textwidth}
        \centering
        \includegraphics[width=0.96\textwidth]{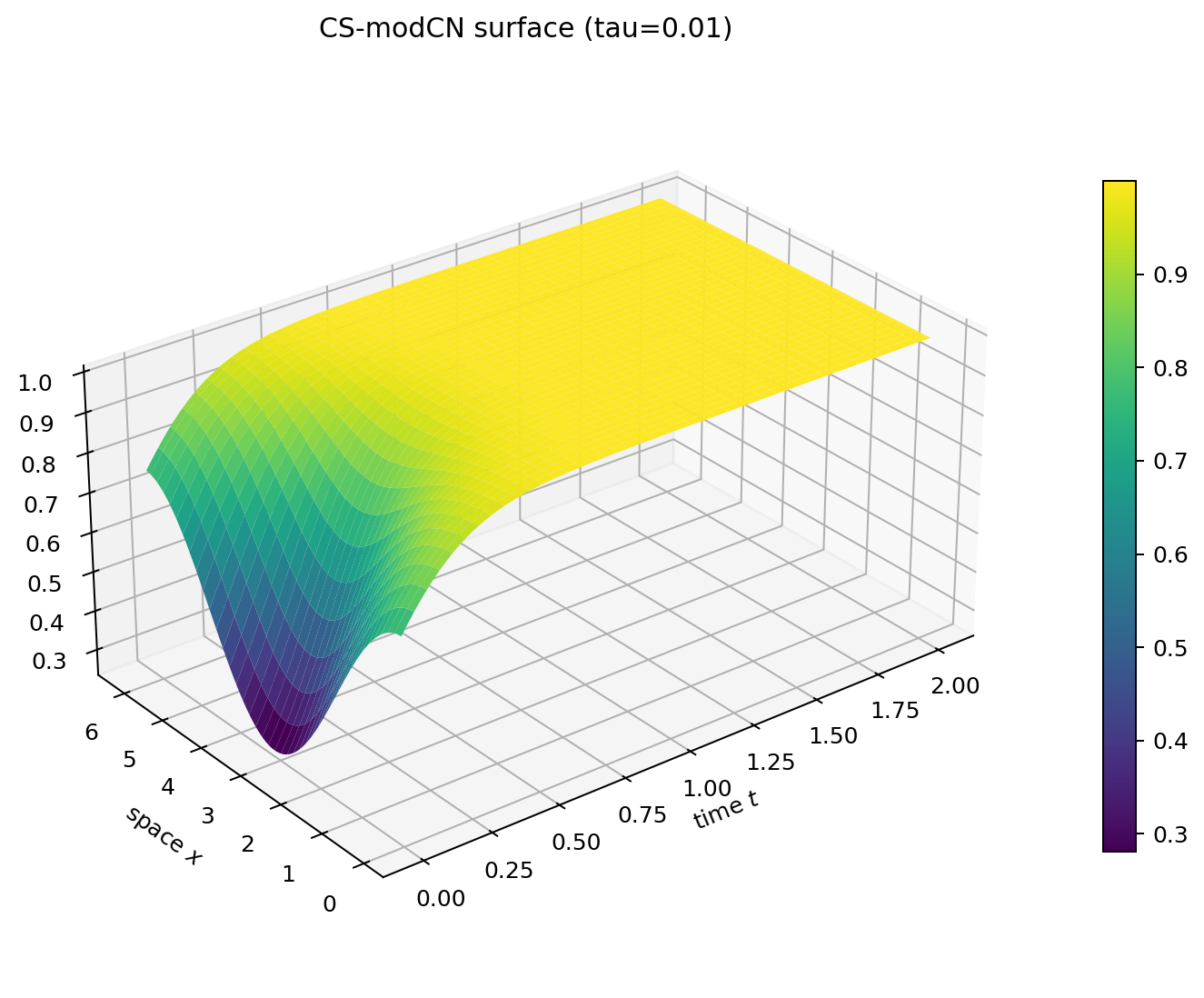}
    \end{minipage}\hfill
    \begin{minipage}[b]{0.45\textwidth}
        \centering
        \includegraphics[width=0.96\textwidth]{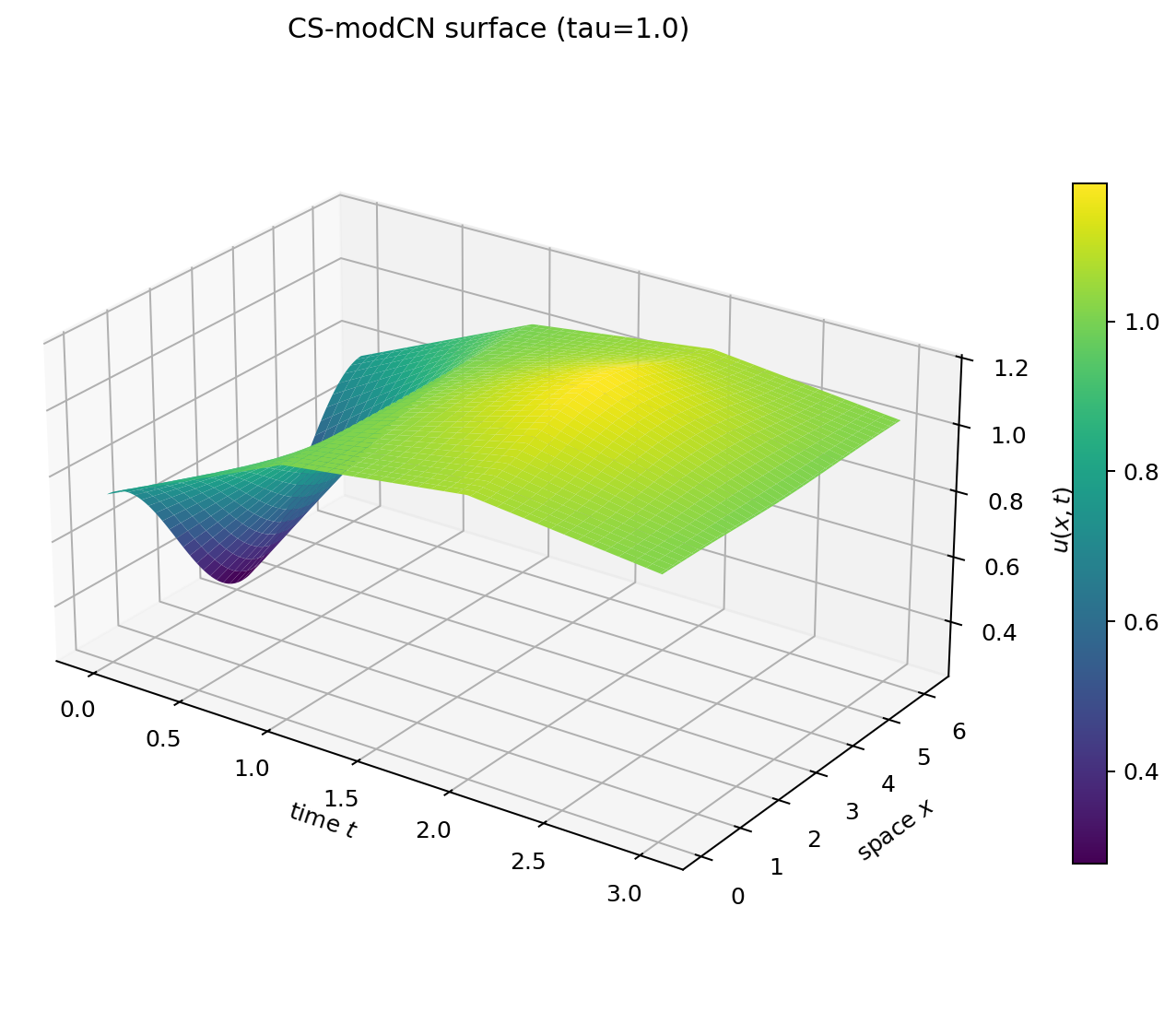}
    \end{minipage}
    \caption{Second-order convex splitting comparison. Left: small-step run with \(\tau=0.01\) on \([0,2]\), for which the solution is pointwise increasing. Right: large-step run with \(\tau=1\) on \([0,3]\), where overshoot and subsequent decrease make the loss of pointwise monotonicity visible.}
    \label{fig:csmodcn-step-compare-final}
\end{figure}

\paragraph{Second-order stabilized CN/AB}
The stabilized CN/AB scheme \eqref{eq:stab-CNAB} is the most representative numerical test for distinguishing maximum-bound preservation from pointwise monotonicity. We use
$S=5,\tau_{\rm s}=0.05,T_{\rm s}=2,$
and $\tau_{\rm l}=0.5,T_{\rm l}=5.$

The smaller-step surface remains monotone: $m_{\Delta,{\rm s}} \approx 3.28\times10^{-7}>0.$
For the larger step, the solution remains inside the maximum-bound interval $[0,1]$, i.e.,
$2.50\times10^{-1}\leq u_{j,{\rm l}}^n\leq 9.995\times10^{-1},$
but the pointwise increment becomes negative:
$m_{\Delta,{\rm l}}
    \approx -4.37\times10^{-2}<0 .$
Hence the failure is not a maximum-bound violation. It is a local reversal of the monotone time direction inside the maximum-bound range. On this interval, \(|f'|\le2\); because \(\eps=0.5\), the standard modified-energy condition is \(S\ge L^2/(4\eps^2)=4\), which is satisfied by the choice \(S=5\).

\begin{figure}[!htbp]
    \centering
    \begin{minipage}[b]{0.45\textwidth}
        \centering
        \includegraphics[width=0.96\textwidth]{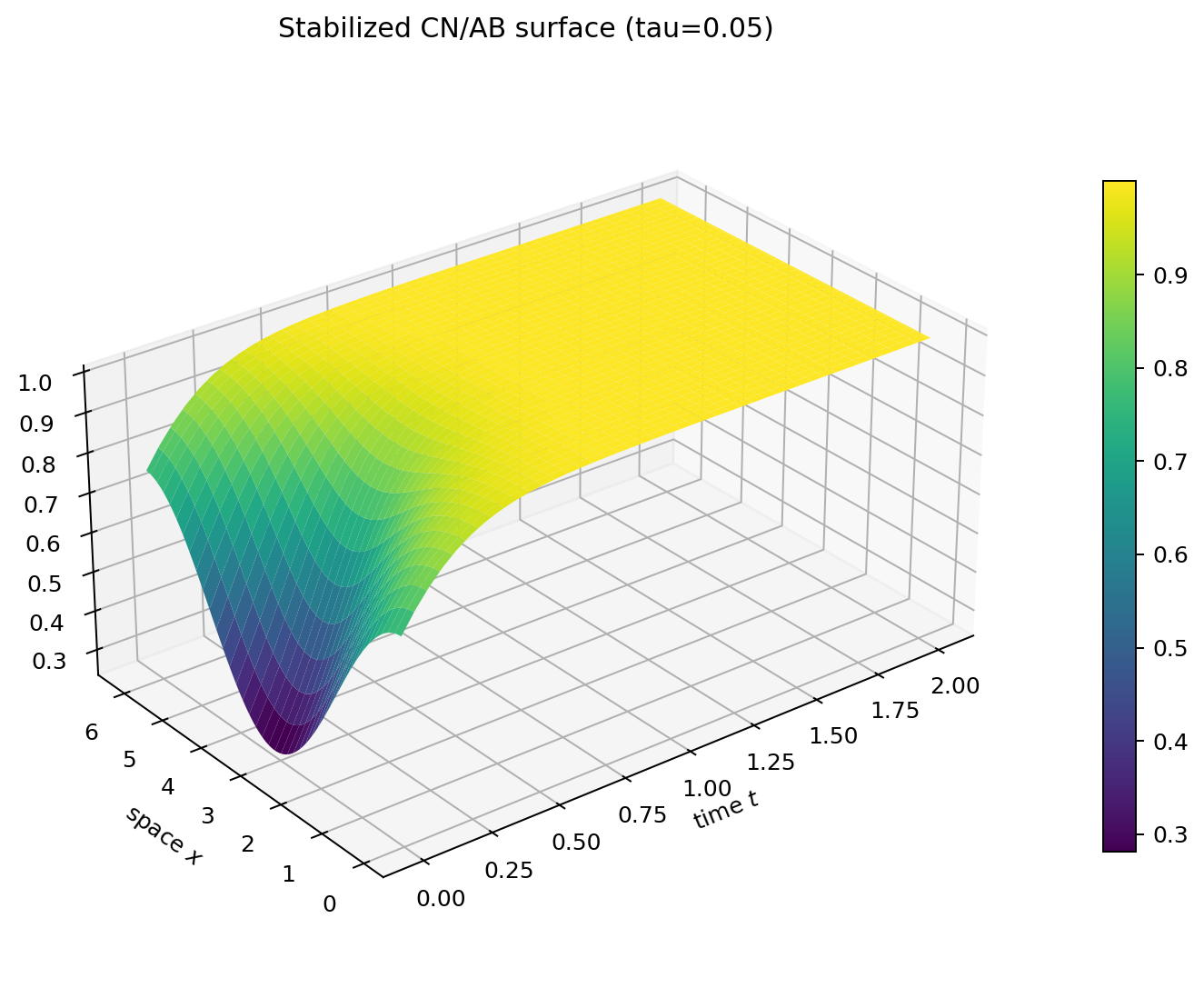}
    \end{minipage}\hfill
    \begin{minipage}[b]{0.45\textwidth}
        \centering
        \includegraphics[width=0.96\textwidth]{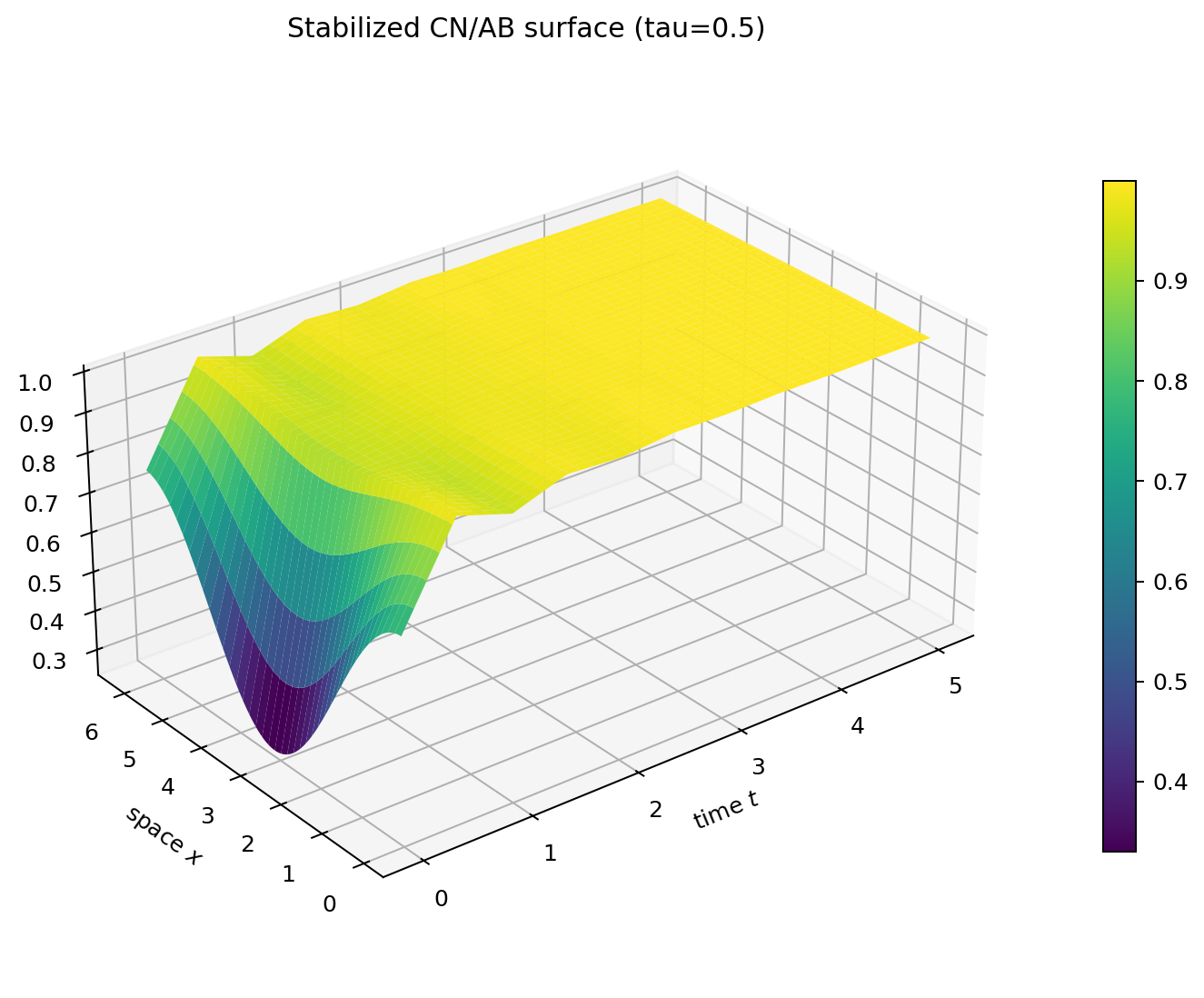}
    \end{minipage}
    \caption{Stabilized CN/AB comparison. Left: small-step run with \(\tau=0.05\), for which the numerical solution remains monotone. Right: large-step run with \(\tau=0.5\), for which the solution remains inside the maximum-bound interval but loses pointwise monotonicity after a monotone initialization.}
    \label{fig:stab-cnab-step-compare-final}
\end{figure}

This experiment is sharper than the convex-splitting test: the stabilized CN/AB solution does not leave the maximum-bound range, but it still loses the local time direction. The small- and large-step regimes in Figure~\ref{fig:stab-cnab-step-compare-final} make this distinction visually transparent. Thus maximum-bound preservation, even when accompanied by modified-energy stability, is not a substitute for pointwise monotonicity.

\subsubsection{Auxiliary-variable schemes: IEQ and SAV}

We finally test two standard auxiliary-variable strategies. These methods are designed to dissipate a modified energy, but the experiments below show that this modified-energy stability does not determine the pointwise direction of the original phase variable for large steps.

\paragraph{IEQ}
For IEQ we take
$ C_{\rm IEQ}=1,
    \tau_{\rm s}=0.1,
    T_{\rm s}=2,$
and
$\tau_{\rm l}=1,
    T_{\rm l}=8 .$
The smaller-step run is pointwise nondecreasing:
$ m_{\Delta,{\rm s}}
    \approx 2.74\times10^{-12}>0 .$
For the larger step, the numerical solution develops a negative pointwise increment:
$m_{\Delta,{\rm l}}
    \approx -1.08\times10^{0}<0 .$  The associated change from monotone growth to a wrong-signed increment is visible in Figure~\ref{fig:ieq-step-compare-final}.

\begin{figure}[!htbp]
    \centering
    \begin{minipage}[b]{0.45\textwidth}
        \centering
        \includegraphics[width=0.96\textwidth]{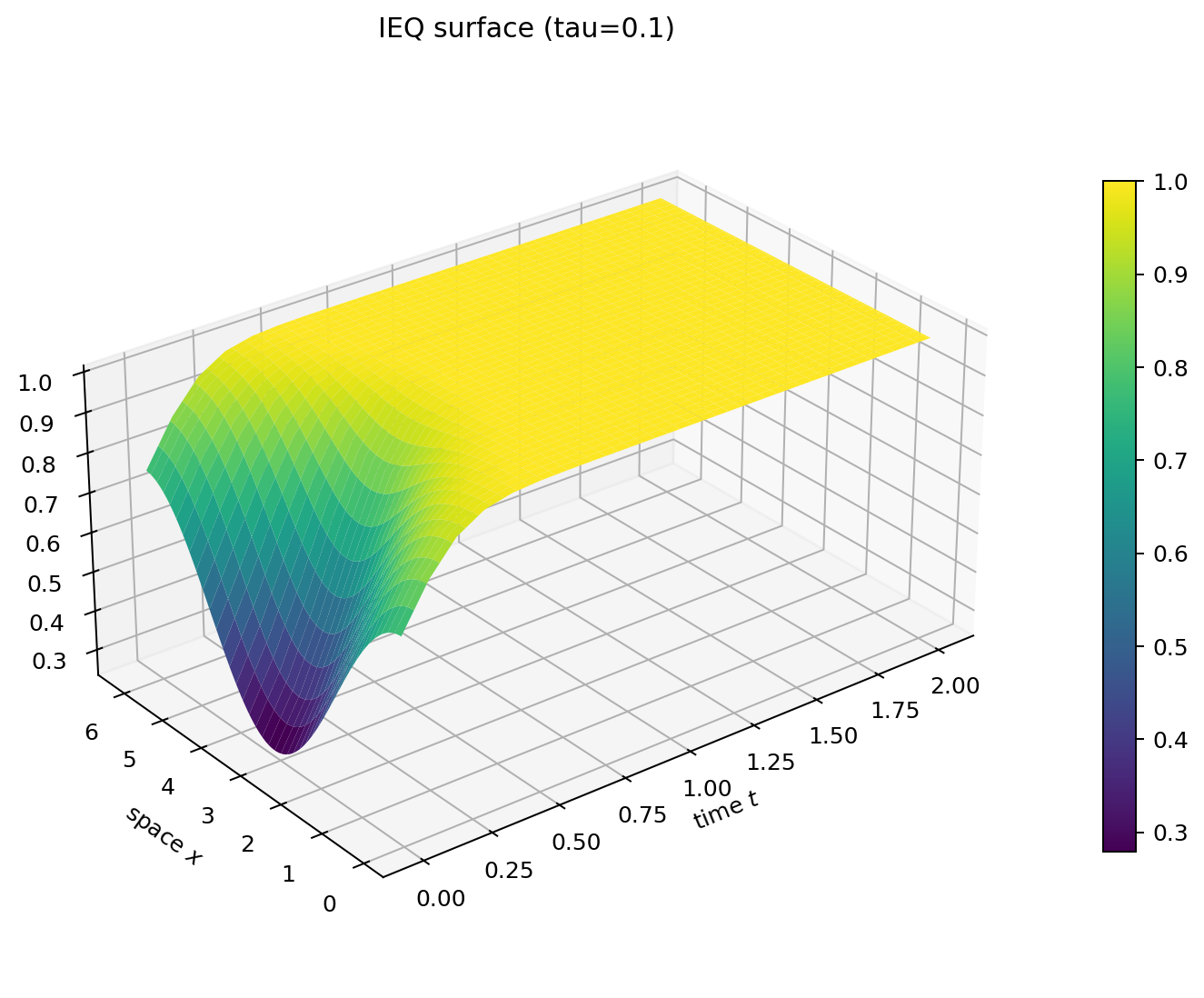}
    \end{minipage}\hfill
    \begin{minipage}[b]{0.45\textwidth}
        \centering
        \includegraphics[width=0.96\textwidth]{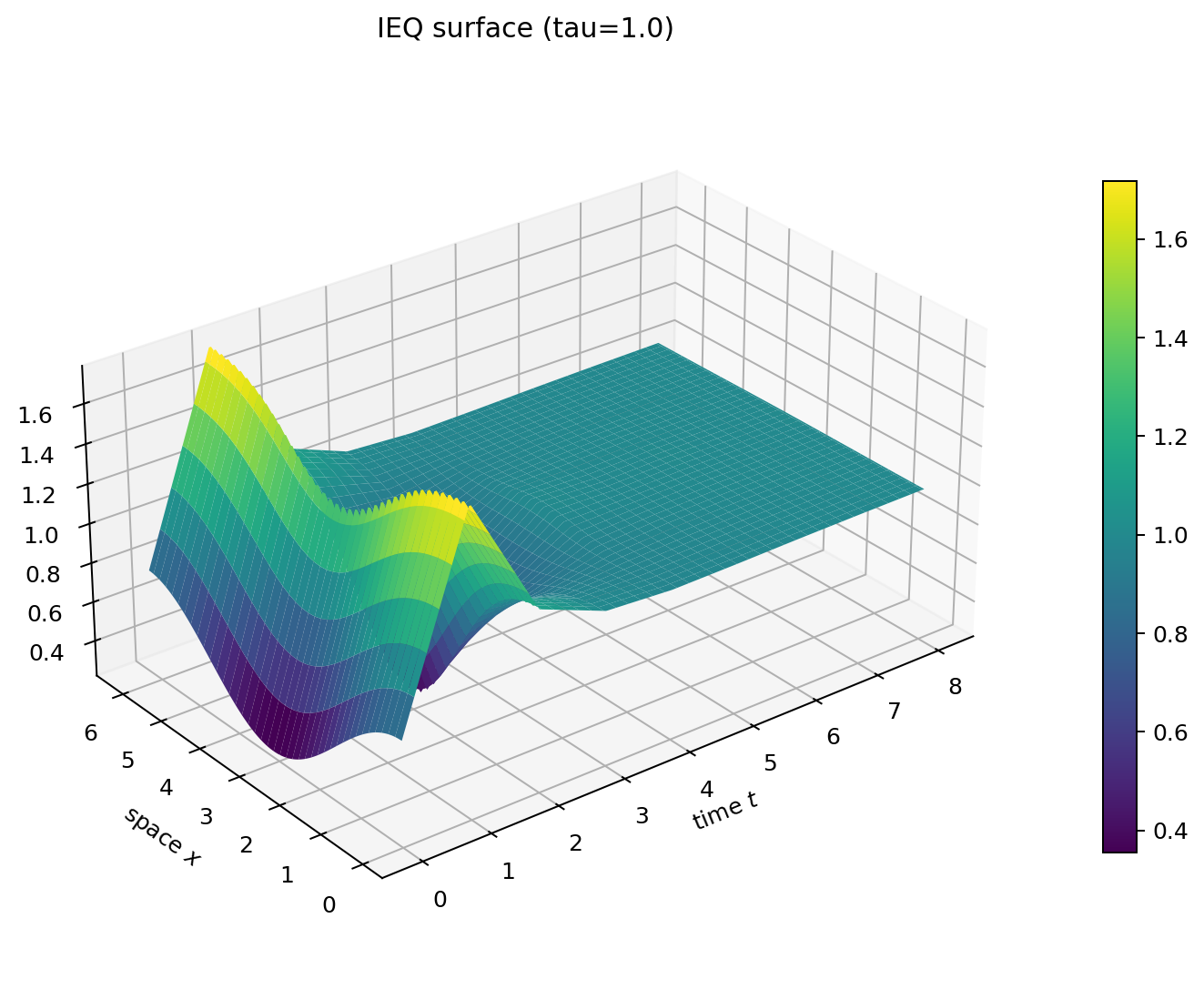}
    \end{minipage}
    \caption{IEQ comparison. Left: small-step run with \(\tau=0.1\) on \([0,2]\), for which the numerical solution remains monotone. Right: large-step run with \(\tau=1\) on \([0,8]\), for which pointwise monotonicity is lost.}
    \label{fig:ieq-step-compare-final}
\end{figure}

\paragraph{SAV}
For SAV we use
$C_0=1,
    \tau_{\rm s}=0.1,
    T_{\rm s}=2,$
and
$\tau_{\rm l}=1,
    T_{\rm l}=8 .$
The smaller-step run remains pointwise nondecreasing:
$m_{\Delta,{\rm s}}
    \approx 1.30\times10^{-11}>0 .$
For the larger step, a negative pointwise increment appears:
$m_{\Delta,{\rm l}}
    \approx -3.28\times10^{-1}<0 .$  The two corresponding surface plots in Figure~\ref{fig:sav-step-compare-final} provide the direct numerical evidence for this large-step loss of pointwise monotonicity.

\begin{figure}[!htbp]
    \centering
    \begin{minipage}[b]{0.45\textwidth}
        \centering
        \includegraphics[width=0.96\textwidth]{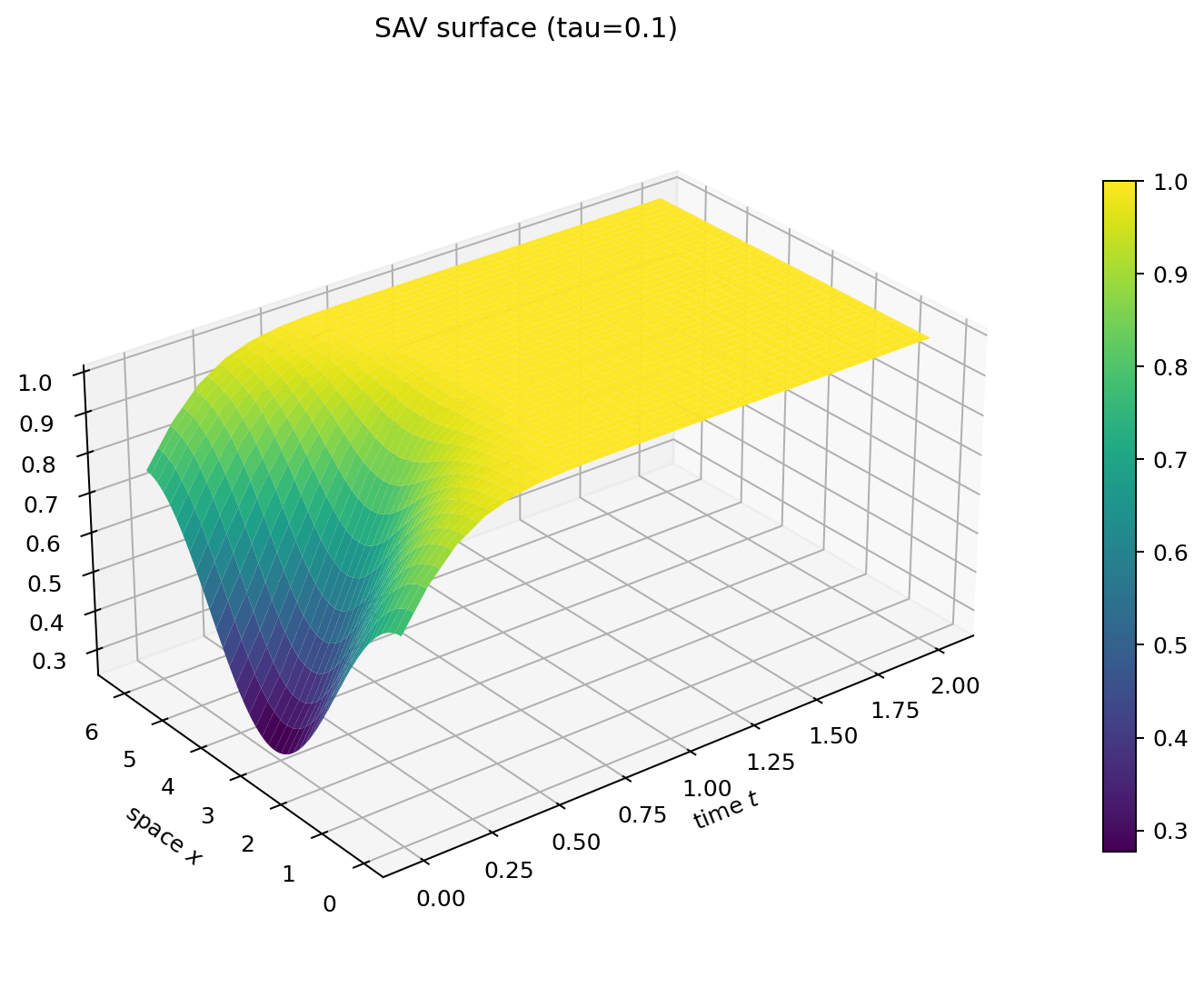}
    \end{minipage}\hfill
    \begin{minipage}[b]{0.45\textwidth}
        \centering
        \includegraphics[width=0.96\textwidth]{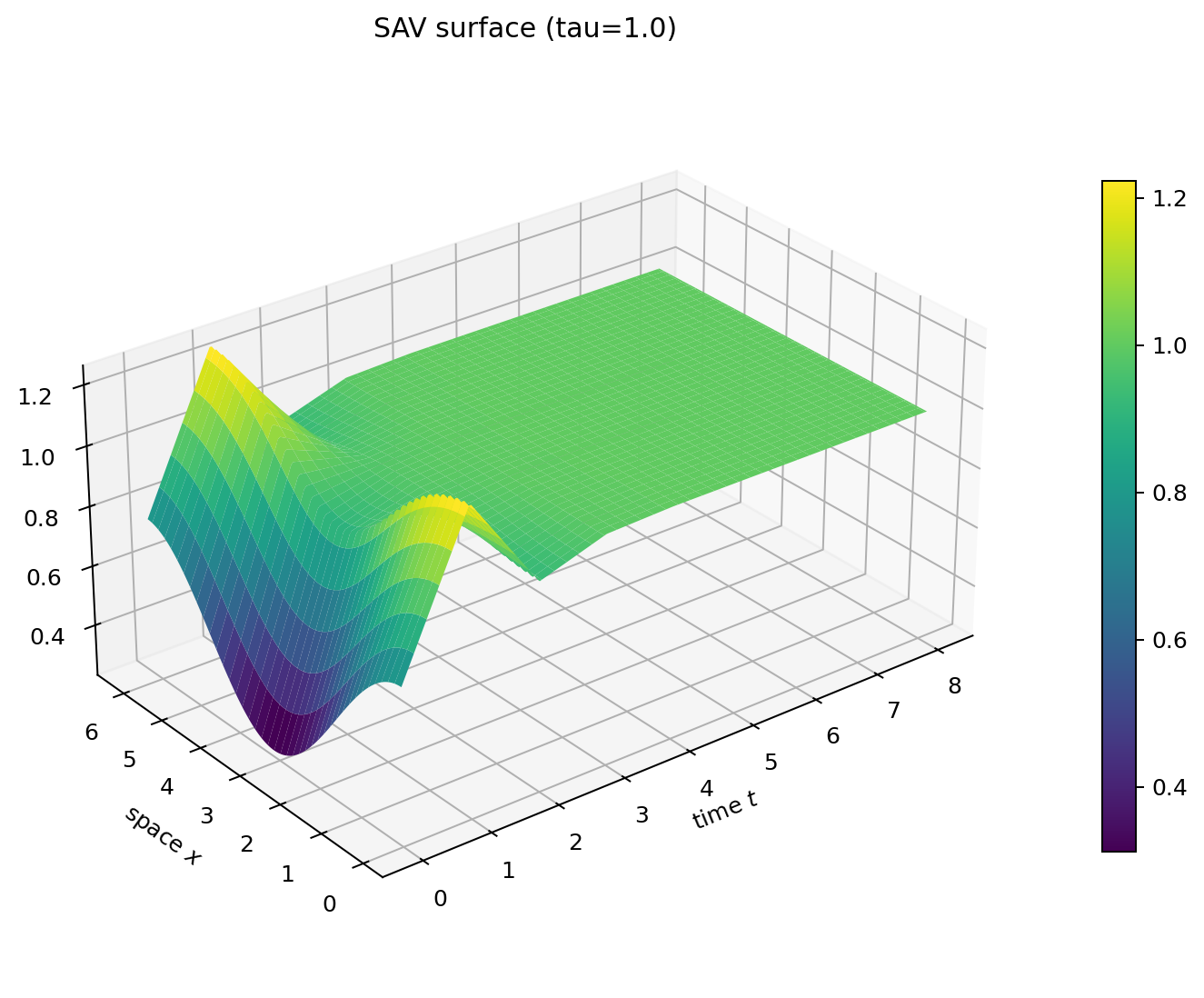}
    \end{minipage}
    \caption{SAV comparison. Left: small-step run with \(\tau=0.1\) on \([0,2]\), for which the numerical solution remains monotone. Right: large-step run with \(\tau=1\) on \([0,8]\), for which pointwise monotonicity is lost.}
    \label{fig:sav-step-compare-final}
\end{figure}

\subsection{Two-dimensional tests and grid refinement}
\label{subsec:two-dimensional-grid-refinement}

We finally examine whether the one-dimensional observations persist for a
genuinely two-dimensional nonhomogeneous solution and whether the signs of the
computed increments are insensitive to spatial grid refinement. Let
$\Omega_2=(0,2\pi)^2$
with periodic boundary conditions and choose the genuinely two-dimensional
initial datum $u_0(x,y)=\frac12+\frac18(\cos x+\cos y).$
Writing $q=\cos x+\cos y\in[-2,2]$, we have
$
    \frac14\le u_0(x,y)\le\frac34,
    \,
    \Delta u_0=-\frac18q.
$
Since $\eps=0.5$, direct expansion gives
\begin{align}
    \Reps(u_0)
    &=-\frac18q
      +4\left(\frac12+\frac18q-
      \left(\frac12+\frac18q\right)^3\right) \notag\\
    &=\frac32-\frac{3}{32}q^2-\frac{1}{128}q^3
      \ge \frac{17}{16}>0.
    \label{eq:two-dimensional-residual-positive}
\end{align}
Thus $u_0\in\Aeps$ and the exact two-dimensional Allen--Cahn flow is
pointwise nondecreasing. In particular, a negative discrete increment cannot
be attributed to a mixed-sign initial residual.

The spatial Laplacian is discretized by the second-order periodic
finite-difference operator on two uniform grids, $128^2$ and $256^2$. For
both meshes the discrete initial residual is strictly positive:
$
    \min R_h^0=1.062550\quad(N=128),
    \,
    \min R_h^0=1.062513\quad(N=256).
$
For every method we compute the spatial diagnostic
\begin{equation}
    D_h(x_i,y_j)
    :=\min_n\bigl(u_{i,j}^{n+1}-u_{i,j}^{n}\bigr).
    \label{eq:two-dimensional-increment-diagnostic}
\end{equation}
Each numerical scheme is displayed in a separate figure, and the two panels
in that figure correspond to the two spatial grids. This organization avoids
compressing all methods into one crowded multi-panel plot and makes the
mesh-to-mesh comparison direct. Constant-coefficient periodic elliptic
systems are solved by diagonalizing the discrete Laplacian with the
two-dimensional discrete Fourier transform. The nonlinear equations are
iterated to a tolerance of $2\times10^{-11}$, and the two-step schemes are
initialized by a monotone fully implicit Euler starter with temporal substep
at most $0.025$.

\paragraph{Fully implicit Euler}
We use $\tau=0.03$ and $T=0.6$. Figure~\ref{fig:2d-ie-grid-refinement}
shows $D_h>0$ on both meshes. The minimum increments are
$1.762821\times10^{-3}$ and $1.762877\times10^{-3}$, respectively.

\begin{figure}[!htbp]
    \centering
    \includegraphics[width=0.7\textwidth]{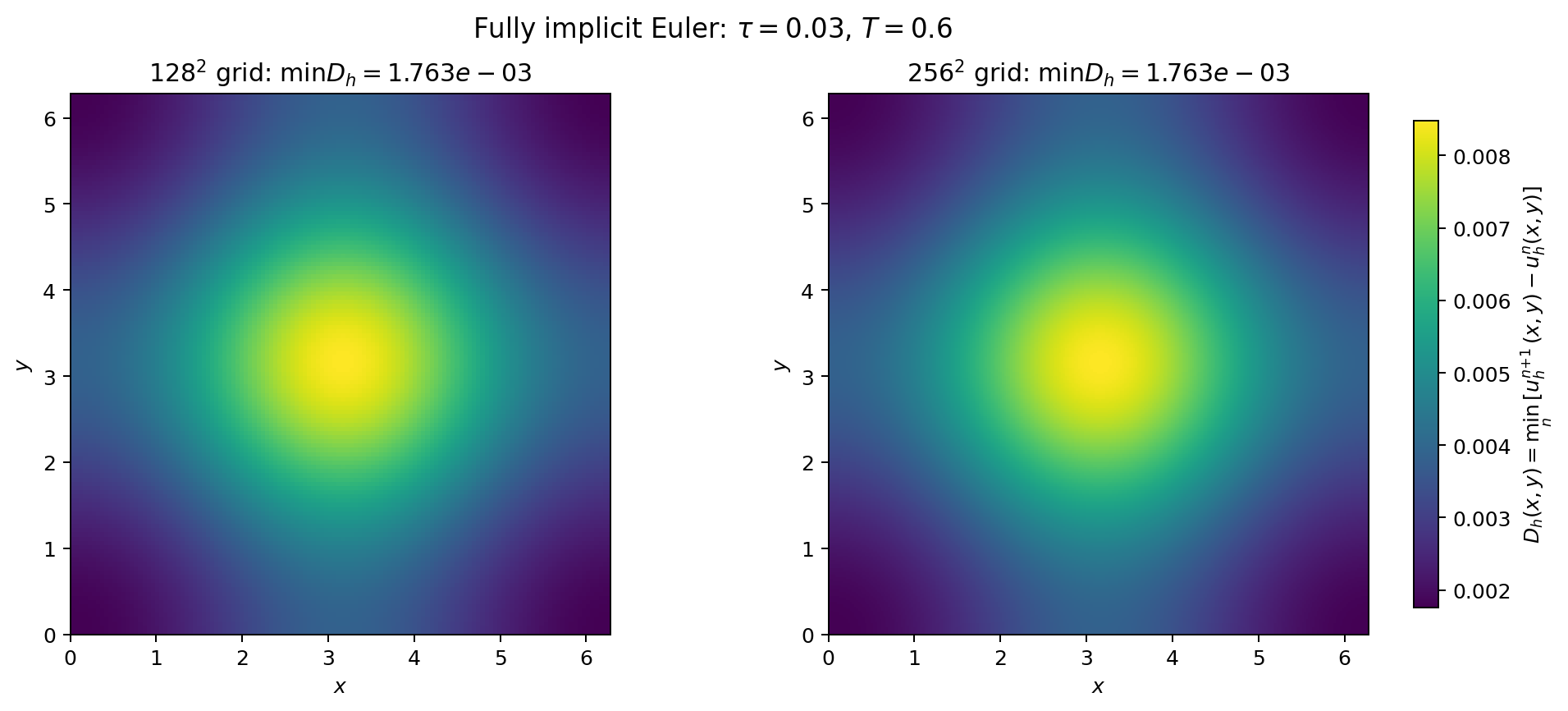}
    \caption{Two-dimensional grid-refinement test for fully implicit Euler.
    The left and right panels use $128^2$ and $256^2$ grids, respectively.
    In both cases $D_h(x,y)$ is strictly positive, confirming the expected
    pointwise monotone growth.}
    \label{fig:2d-ie-grid-refinement}
\end{figure}

\paragraph{First-order convex splitting}
For CSS1 we take $\tau=0.5$ and $T=3$. Figure~\ref{fig:2d-css1-grid-refinement}
shows minimum values of $D_h$ equal to $3.421144\times10^{-3}$ and
$3.421247\times10^{-3}$ on the two grids. Thus the large-step computation
remains pointwise increasing in two space dimensions, although the
one-dimensional tests show that the transient is delayed.

\begin{figure}[!htbp]
    \centering
    \includegraphics[width=0.7\textwidth]{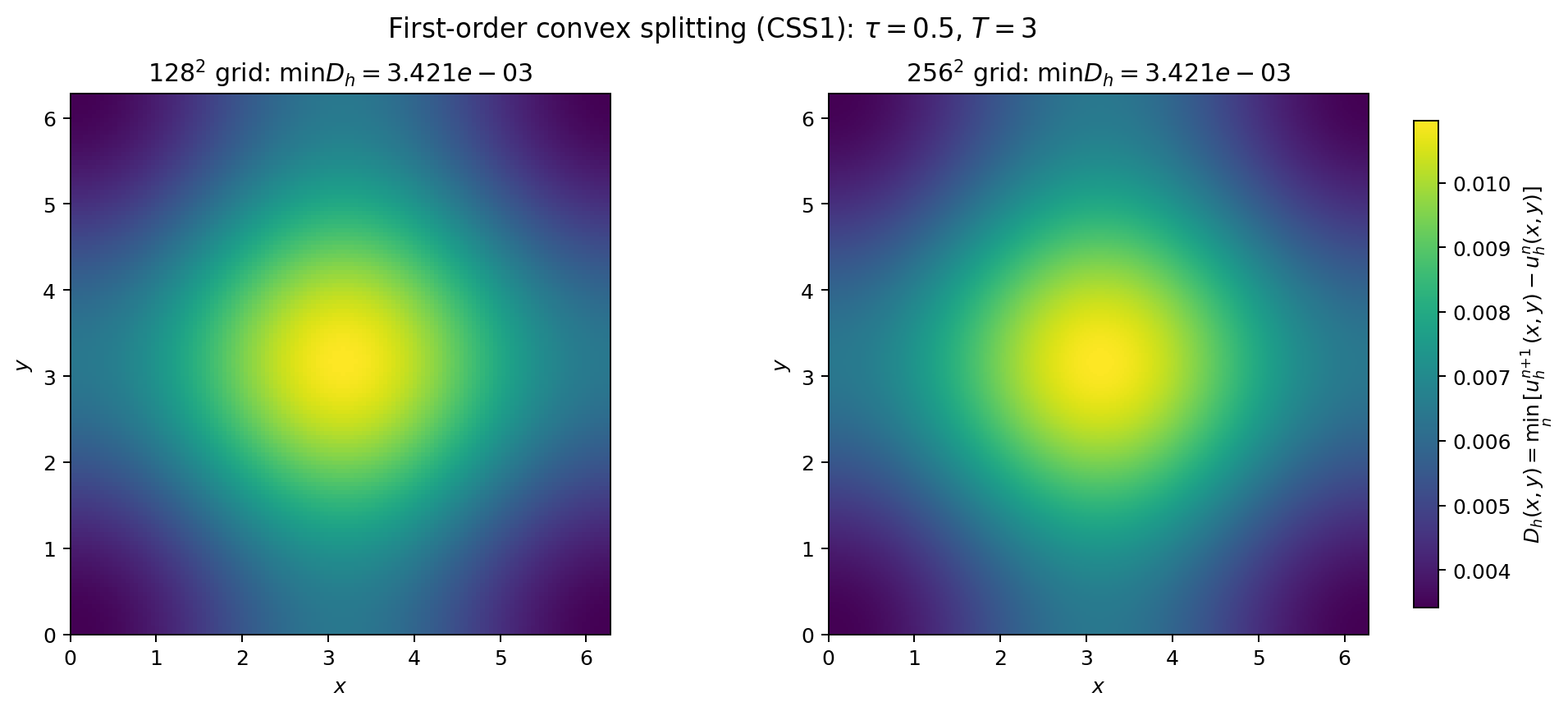}
    \caption{Two-dimensional grid-refinement test for CSS1 with $\tau=0.5$.
    The diagnostic remains positive on both meshes, and the spatial pattern
    is unchanged under refinement.}
    \label{fig:2d-css1-grid-refinement}
\end{figure}

\paragraph{First-order stabilized semi-implicit scheme}
With $S=32$, $\tau=0.5$, and $T=2$,
Figure~\ref{fig:2d-stabilized-first-grid-refinement} gives minimum increments
of $2.473010\times10^{-2}$ and $2.472999\times10^{-2}$ on the two grids.
Hence the strongly stabilized first-order method also preserves the sign of
the pointwise increment on both grids, while retaining the artificial damping
described above.

\begin{figure}[!htbp]
    \centering
    \includegraphics[width=0.7\textwidth]{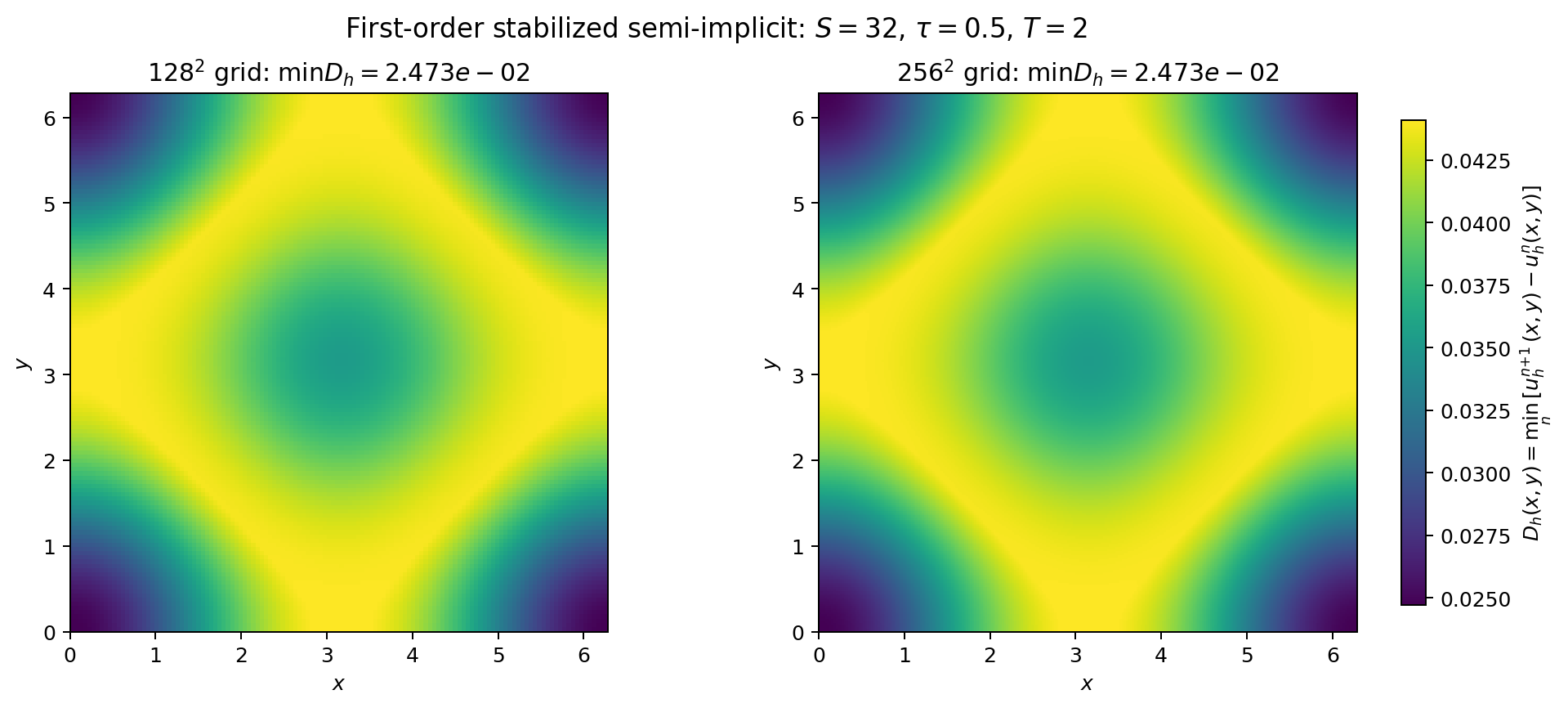}
    \caption{Two-dimensional grid-refinement test for the first-order
    stabilized semi-implicit method with $S=32$ and $\tau=0.5$. Both grids
    give a positive diagnostic with essentially identical magnitude and
    shape.}
    \label{fig:2d-stabilized-first-grid-refinement}
\end{figure}

\paragraph{Second-order convex splitting}
For CS-modCN we use $\tau=1$ and $T=3$.
Figure~\ref{fig:2d-csmodcn-grid-refinement} shows minimum increments of
$-1.866301\times10^{-1}$ on the $128^2$ grid and
$-1.866300\times10^{-1}$ on the $256^2$ grid. The negative increment is
therefore several orders of magnitude larger than the nonlinear tolerance
and remains unchanged under refinement.

\begin{figure}[!htbp]
    \centering
    \includegraphics[width=0.7\textwidth]{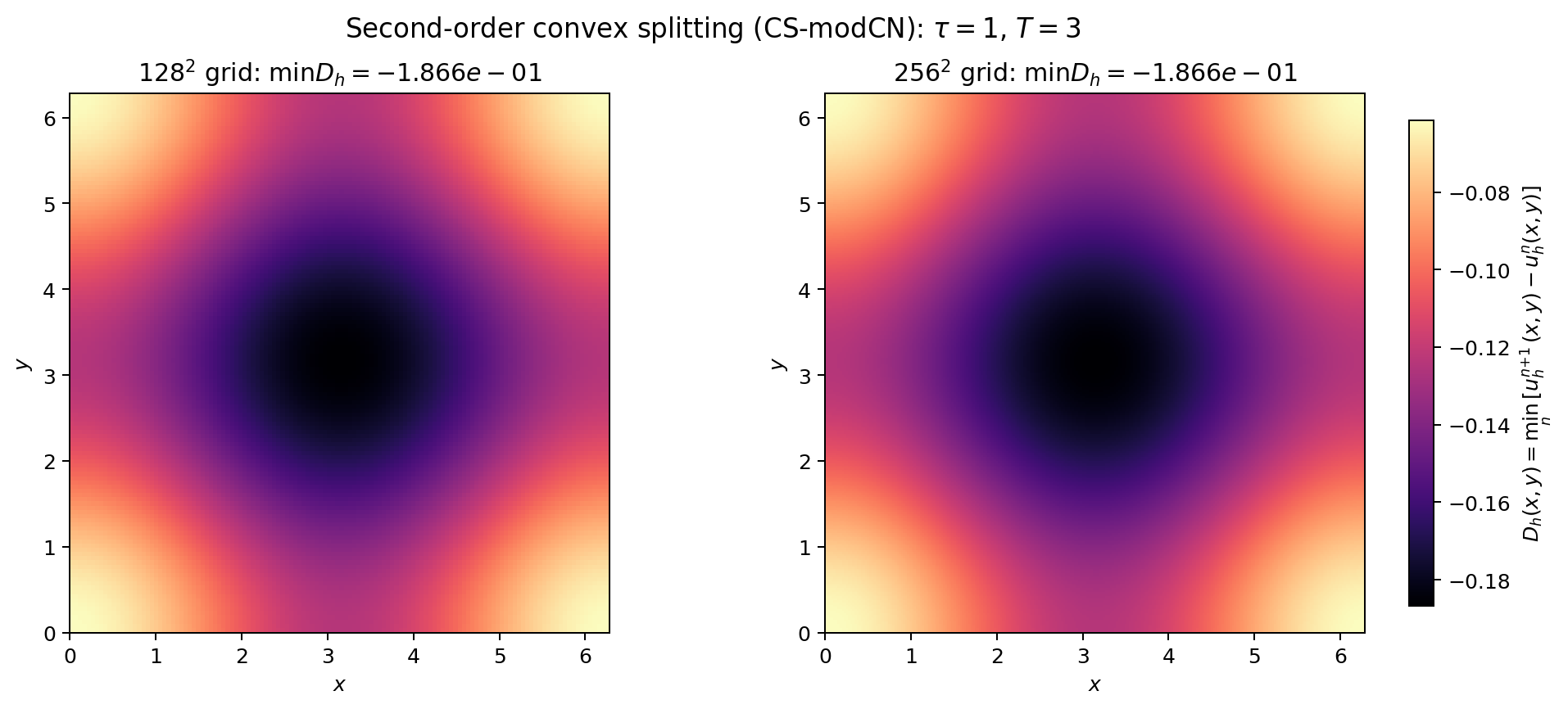}
    \caption{Two-dimensional grid-refinement test for CS-modCN with
    $\tau=1$. The diagnostic is negative on both grids, showing that the
    wrong pointwise direction is not a one-dimensional or coarse-grid
    artifact.}
    \label{fig:2d-csmodcn-grid-refinement}
\end{figure}

\paragraph{Second-order stabilized CN/AB}
For $S=5$, $\tau=0.5$, and $T=5$, the minimum increments are
$-4.283470\times10^{-2}$ and $-4.283526\times10^{-2}$. Moreover, the
negative region occupies approximately $87.4\%$ of the grid points for both
meshes. Figure~\ref{fig:2d-stab-cnab-grid-refinement} shows that the location
and geometry of this region are stable under refinement.

\begin{figure}[!htbp]
    \centering
    \includegraphics[width=0.7\textwidth]{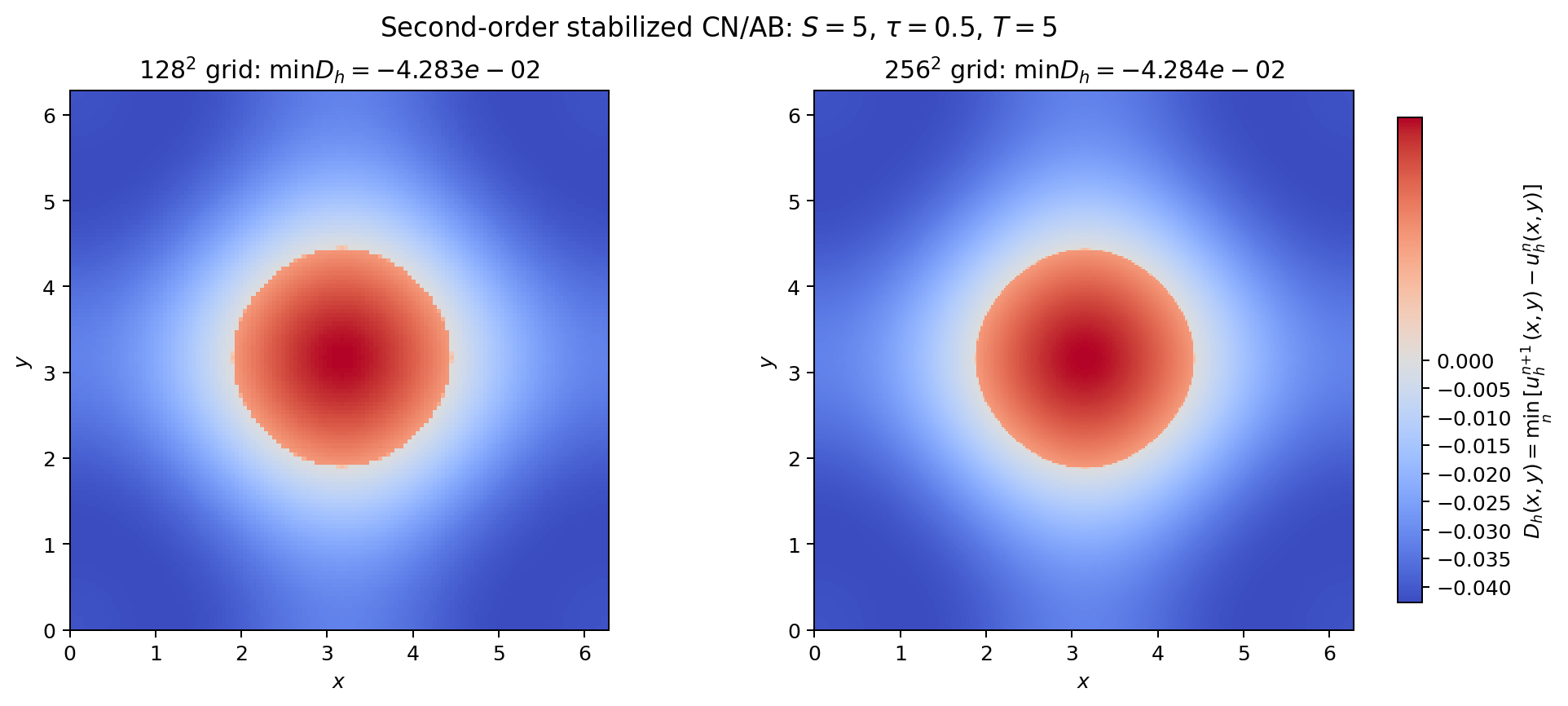}
    \caption{Two-dimensional grid-refinement test for stabilized CN/AB with
    $S=5$ and $\tau=0.5$. The negative region and its minimum value are
    essentially mesh independent.}
    \label{fig:2d-stab-cnab-grid-refinement}
\end{figure}

\paragraph{IEQ}
For $C_{\rm IEQ}=1$, $\tau=1$, and $T=8$,
Figure~\ref{fig:2d-ieq-grid-refinement} shows computed minima of $-1.031358$
and $-1.031322$ on the two grids. The wrong-signed increment is therefore
robust with respect to grid refinement and is much larger than either the
nonlinear solver tolerance or the difference between the two meshes.

\begin{figure}[!htbp]
    \centering
    \includegraphics[width=0.7\textwidth]{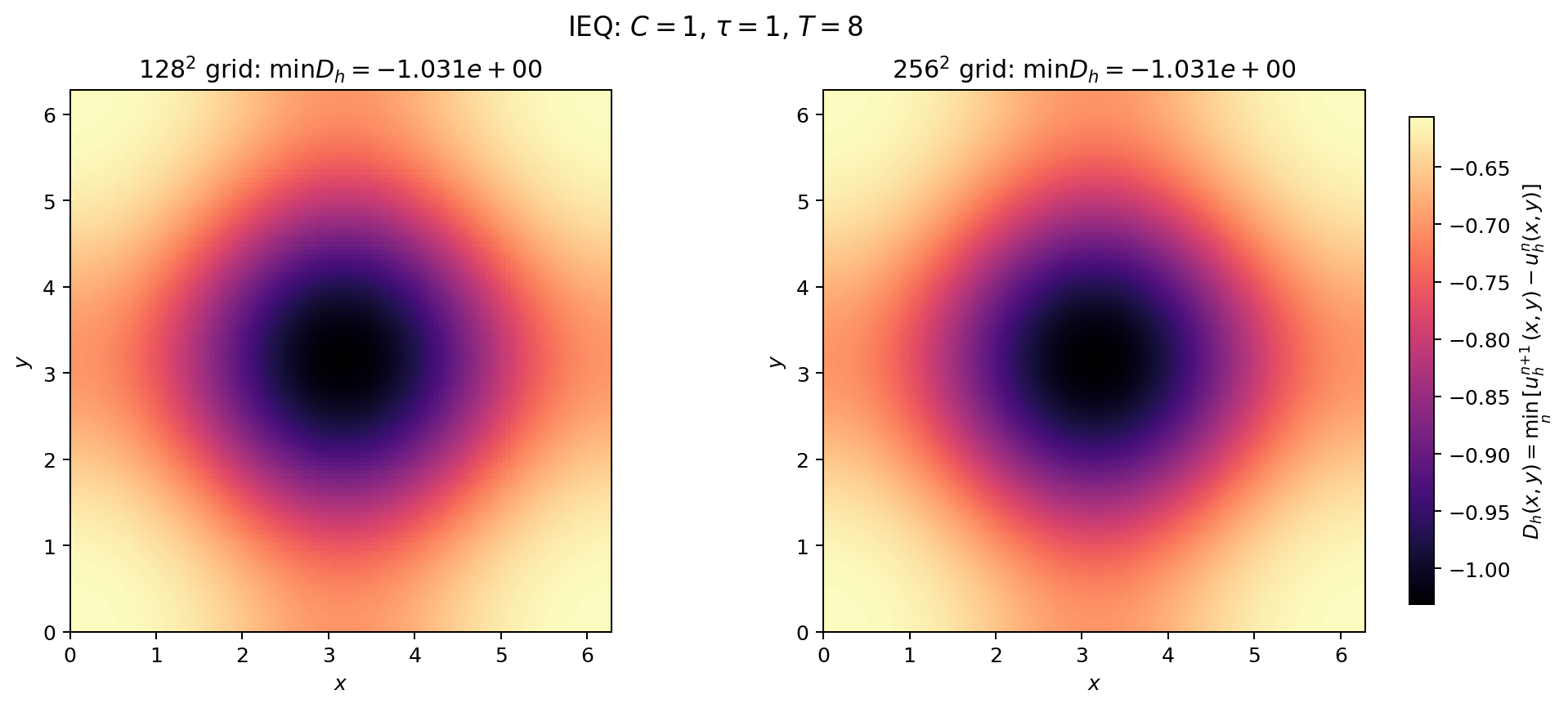}
    \caption{Two-dimensional grid-refinement test for IEQ with $\tau=1$.
    Both grids produce the same negative-increment pattern and nearly the same
    minimum value.}
    \label{fig:2d-ieq-grid-refinement}
\end{figure}

\paragraph{SAV}
Finally, for $C_0=1$, $\tau=1$, and $T=8$,
Figure~\ref{fig:2d-sav-grid-refinement} shows minimum increments of
$-2.962691\times10^{-1}$ and $-2.962770\times10^{-1}$ on the two grids.
The negative region covers approximately $98.2\%$ of the grid points on both
meshes, and its boundary is visually stable under refinement.

\begin{figure}[!htbp]
    \centering
    \includegraphics[width=0.7\textwidth]{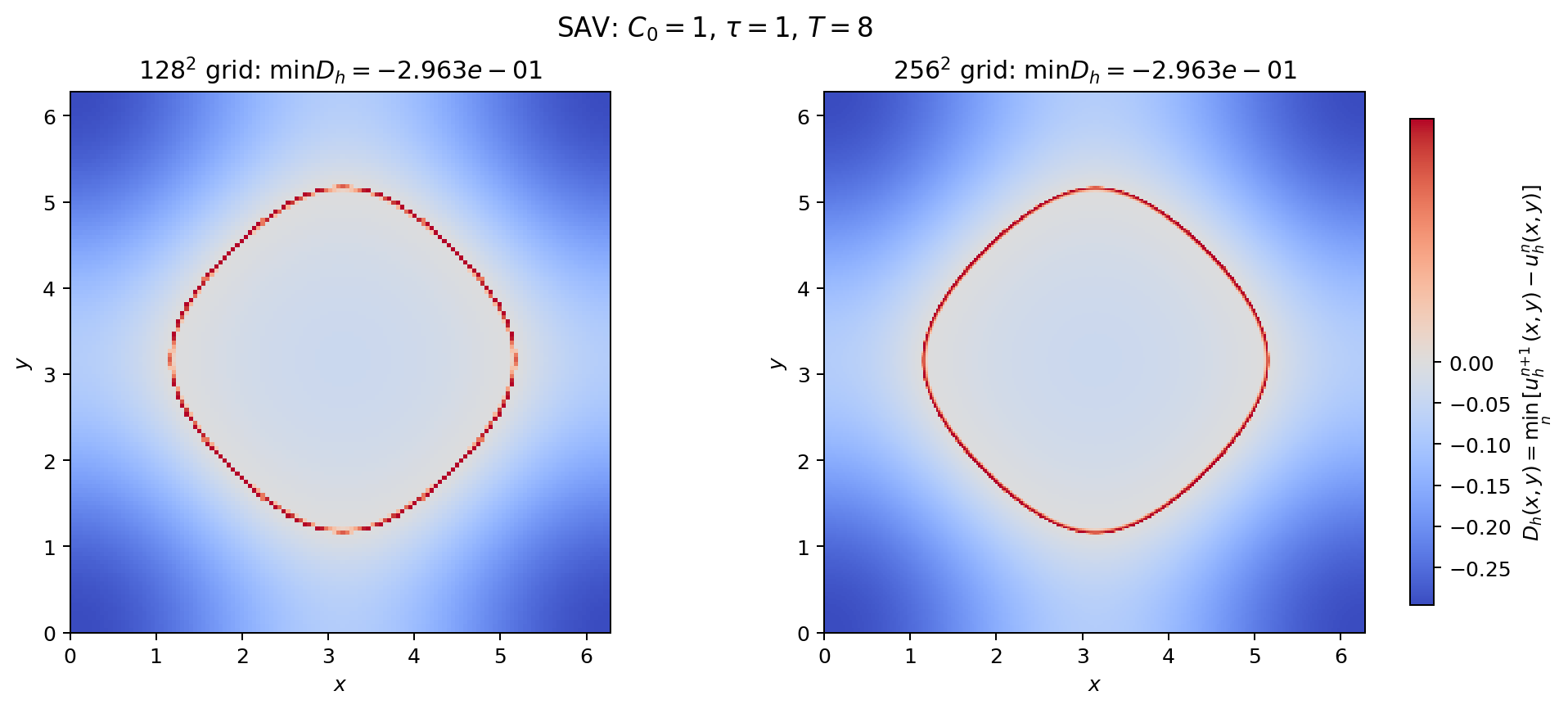}
    \caption{Two-dimensional grid-refinement test for SAV with $\tau=1$.
    The negative region and the magnitude of the minimum increment persist
    from the $128^2$ grid to the $256^2$ grid.}
    \label{fig:2d-sav-grid-refinement}
\end{figure}

The minimum-increment values are summarized in
Table~\ref{tab:two-dimensional-grid-refinement}. For every scheme, the
relative change in $\min D_h$ between the two grids is below
$3.5\times10^{-5}$. Consequently, the positive signs obtained for fully
implicit Euler, CSS1, and first-order stabilization, as well as the negative
signs obtained for CS-modCN, stabilized CN/AB, IEQ, and SAV, are insensitive
to this factor-two spatial refinement.

\begin{table}[!htbp]
\centering
\small
\caption{Two-dimensional minimum-increment diagnostic on two periodic grids.}
\label{tab:two-dimensional-grid-refinement}
\renewcommand{\arraystretch}{1.18}
\begin{tabular}{@{}p{0.38\textwidth}ccp{0.18\textwidth}@{}}
\hline
Scheme and representative parameters & $128^2$ & $256^2$ & Observed sign \\
\hline
Fully implicit Euler, $\tau=0.03$ & $1.762821\times10^{-3}$ & $1.762877\times10^{-3}$ & Positive \\
CSS1, $\tau=0.5$ & $3.421144\times10^{-3}$ & $3.421247\times10^{-3}$ & Positive \\
First-order stabilized, $S=32$, $\tau=0.5$ & $2.473010\times10^{-2}$ & $2.472999\times10^{-2}$ & Positive \\
CS-modCN, $\tau=1$ & $-1.866301\times10^{-1}$ & $-1.866300\times10^{-1}$ & Negative \\
Stabilized CN/AB, $S=5$, $\tau=0.5$ & $-4.283470\times10^{-2}$ & $-4.283526\times10^{-2}$ & Negative \\
IEQ, $\tau=1$ & $-1.031358$ & $-1.031322$ & Negative \\
SAV, $\tau=1$ & $-2.962691\times10^{-1}$ & $-2.962770\times10^{-1}$ & Negative \\
\hline
\end{tabular}
\renewcommand{\arraystretch}{1.0}
\end{table}

\FloatBarrier

\subsection{Synthesis of the numerical evidence}

The experiments support the theoretical classification for genuinely nonhomogeneous data. The one-dimensional fully implicit Euler reference in Figure~\ref{fig:ie-monotone-surface-final} remains pointwise monotone. The CSS1 and first-order stabilized comparisons in Figures~\ref{fig:css1-step-compare-final}, \ref{fig:css1-point-delay-final}, \ref{fig:stab-step-compare-final}, and~\ref{fig:stab-point-delay-final} show that large steps can preserve the pointwise direction while severely distorting the physical time scale. By contrast, the two-step and auxiliary-variable experiments in Figures~\ref{fig:csmodcn-step-compare-final}, \ref{fig:stab-cnab-step-compare-final}, \ref{fig:ieq-step-compare-final}, and~\ref{fig:sav-step-compare-final} reproduce the predicted emergence of negative pointwise increments for large steps. The two-dimensional tests in Figures~\ref{fig:2d-ie-grid-refinement}--\ref{fig:2d-sav-grid-refinement} and Table~\ref{tab:two-dimensional-grid-refinement} show that the same sign classification persists when both spatial directions are active and is stable from the $128^2$ grid to the $256^2$ grid. The stabilized CN/AB experiment remains especially representative: its numerical solution can stay inside the maximum-bound interval while losing pointwise monotonicity. Thus the combined one- and two-dimensional evidence confirms that unconditional energy stability, modified-energy stability, and maximum-bound preservation do not determine the local pointwise time direction.


\section{Concluding remarks}
\label{sec:con}

This paper has developed a pointwise monotonicity perspective for the Allen--Cahn equation and for several widely used time discretizations. At the continuous level, the sign of the Allen--Cahn residual identifies admissible classes of data for which the solution moves pointwise in the direction of a stable phase. Thus the exact flow carries a local-in-space time direction: data in \(\mathcal A_\varepsilon^+\) generate monotone growth in the direction of \(+1\), while data in \(\mathcal A_\varepsilon^-\) generate the symmetric monotone decay in the direction of \(-1\). This property is more informative than range control, because it records not only where the solution lies but also in which direction it evolves at each spatial point.

The fully implicit Euler method serves as the positive reference scheme. Under the usual unique-solvability restriction, and provided the discrete residual has the appropriate sign, the implicit resolvent preserves the pointwise monotone direction and propagates the residual sign. In this sense, the monotone-growth mechanism of the Allen--Cahn flow can be inherited by a genuinely implicit discretization under a computable time-step condition.

The behavior of unconditionally energy-stable and
modified-energy-stable schemes is more subtle. For convex splitting
methods, the first-order scheme preserves the monotone direction but
evolves with a reduced effective time step, whereas the second-order
scheme may produce wrong-signed pointwise increments for sufficiently
large time steps despite its modified-energy stability. Stabilized
schemes exhibit a similar distinction. In the strongly stabilized
regime, the first-order scheme can preserve the pointwise direction,
but at the cost of artificial damping and a distorted physical time
scale. By contrast, the second-order stabilized CN/AB scheme may lose
pointwise monotonicity for large time steps even in the
modified-energy-stable regime. Our numerical experiments further show
that its computed solution may remain within the maximum-bound interval
while the pointwise increment has already acquired the wrong sign,
illustrating that bound preservation and pointwise monotonicity are
different dynamical properties. The IEQ and SAV formulations likewise
admit admissible monotone data for which sufficiently large time steps
reverse the pointwise direction, although their corresponding modified
energies remain nonincreasing. For second-order convex splitting,
stabilized CN/AB, IEQ, and SAV, the strict homogeneous sign reversal is
stable under sufficiently small smooth nonhomogeneous admissible
perturbations. Hence these counterexamples persist when
\(\Delta u^0\not\equiv0\) and diffusion genuinely participates in the
PDE dynamics. Table~\ref{tab:concluding-classification}
summarizes the time-step regimes established in this paper and the
resulting pointwise monotonicity behavior. The positive statements apply
under the indicated residual and range hypotheses, whereas the negative
statements record explicit counterexamples and should not be read as failure
for every initial datum.

\begin{table}[!htbp]
\centering
\small
\caption{Time-step regimes proved in this paper and their implications for pointwise monotonicity.}
\label{tab:concluding-classification}
\renewcommand{\arraystretch}{1.25}
\begin{tabular}{@{}p{0.26\textwidth}p{0.40\textwidth}p{0.26\textwidth}@{}}
\hline
Numerical scheme & Time-step regime & Pointwise monotonicity status \\
\hline
Fully implicit Euler &
\(R^0\ge0, 0<\tau<\eps^2\) &
Preserved \\
First-order convex splitting &
\(R^0\ge0, \tau>0\) &
Preserved, with delay \\
First-order stabilized scheme &
\(u_0\in\Aeps, \tau>0, S\ge2/\eps^2\) &
Preserved, with damping \\
Second-order convex splitting &
\(\tau\) sufficiently large &
Not preserved \\
Second-order stabilized scheme &
\(\tau\) sufficiently large &
Not preserved \\
IEQ &
\(\tau\) sufficiently large &
Not preserved \\
SAV &
\(\tau\) sufficiently large &
Not preserved \\
\hline
\end{tabular}
\renewcommand{\arraystretch}{1.0}
\end{table}

These conclusions sharpen the distinction between three stability concepts. Energy stability is a global Lyapunov-type property; maximum-bound preservation is a range constraint; pointwise monotonicity is a local dynamical property. The first two principles are indispensable in phase-field computation, but neither determines the sign of the local time increment.

The message is therefore not that energy-stable schemes are unimportant. Rather, energy stability should be complemented by local dynamical criteria when transient fidelity is essential. For Allen--Cahn simulations, a reliable large-step discretization should be assessed not only by whether it dissipates an energy or preserves a maximum bound, but also by whether it follows the correct pointwise direction on the correct physical time scale.

\appendix
\renewcommand{\thetheorem}{\Alph{section}.\arabic{theorem}}
\refstepcounter{section}
\section*{\thesection.\quad Duhamel proof of the spatially homogeneous reduction}
\addcontentsline{toc}{section}{Appendix A.\enspace Duhamel proof of the spatially homogeneous reduction}
\label{app:homogeneous-reduction}

\begin{theorem}[Spatially homogeneous invariant class and exact ODE reduction]
\label{thm:homogeneous-reduction}
Let $u$ be the unique global mild solution of the Allen--Cahn problem
\eqref{eq:AC} under either periodic or homogeneous Neumann boundary
conditions. If the initial datum is spatially constant,
$u_0(x)\equiv u_0\in\mathbb R,$
then
$
    u(x,t)=v(t),
$
where $v$ is the unique solution of
\begin{equation}
    v'(t)+\frac1{\eps^2}f(v(t))=0,
    \qquad v(0)=u_0.
    \label{eq:appendix-homogeneous-ode}
\end{equation}
In particular, the spatially homogeneous states form an invariant class
of the Allen--Cahn flow.
\end{theorem}

\begin{proof}
Let $S_B(t)=e^{t\Delta_B}$ denote the heat semigroup associated with the
chosen boundary condition $B$. For both periodic and homogeneous Neumann
boundary conditions, the constant function $\mathbf 1$ belongs to the
kernel of $\Delta_B$. Hence
\begin{equation}
    S_B(t)(c\mathbf 1)=c\mathbf 1,
    \qquad c\in\mathbb R,\quad t\ge0.
    \label{eq:appendix-semigroup-constants}
\end{equation}

A mild solution $w$ of \eqref{eq:AC} satisfies the Duhamel formula
\begin{equation}
    w(t)=S_B(t)w(0)
    -\frac1{\eps^2}\int_0^t S_B(t-s)f(w(s))\,\mathrm ds.
    \label{eq:appendix-duhamel-formula}
\end{equation}
Let $v$ solve \eqref{eq:appendix-homogeneous-ode} and define
$\widetilde u(x,t):=v(t).$
Integrating \eqref{eq:appendix-homogeneous-ode} gives
$ v(t)=u_0-\frac1{\eps^2}\int_0^t f(v(s))\,\mathrm ds.$
Since $\widetilde u(\cdot,s)=v(s)\mathbf 1$, relation
\eqref{eq:appendix-semigroup-constants} yields
\begin{align*}
&S_B(t)(u_0\mathbf 1)
-\frac1{\eps^2}\int_0^t
S_B(t-s)f(\widetilde u(s))\,\mathrm ds \\
&\qquad
=u_0\mathbf 1
-\frac1{\eps^2}\int_0^t f(v(s))\mathbf 1\,\mathrm ds
=v(t)\mathbf 1
=\widetilde u(t).
\end{align*}
Thus $\widetilde u$ satisfies \eqref{eq:appendix-duhamel-formula} with the
same initial datum as $u$. By uniqueness of the global mild solution,
$
    u(x,t)=\widetilde u(x,t)=v(t),
    \, x\in\Omega,\, t\ge0.
$
Therefore, the Allen--Cahn flow preserves spatially homogeneous states,
and its restriction to this invariant class is exactly
\eqref{eq:appendix-homogeneous-ode}.
\end{proof}

\begin{remark}[Why invariance still requires an argument]
Under periodic or homogeneous Neumann boundary conditions on a connected
domain, a compatible harmonic initial datum is necessarily constant. However,
checking only that the diffusion term vanishes initially is not, by itself, an
invariance proof. The argument above constructs the spatially constant
candidate $\widetilde u(x,t)=v(t)$, verifies that it satisfies the same Duhamel
formula as the PDE solution, and then uses uniqueness.
\end{remark}

\begin{remark}[Role of the boundary condition]
The reduction relies on the invariance of constants under the heat
semigroup. This property holds for periodic and homogeneous Neumann
boundary conditions. It fails for the homogeneous Dirichlet semigroup:
a nonzero constant is incompatible with $u|_{\partial\Omega}=0$ and is not
in the kernel of the Dirichlet Laplacian. In particular,
$e^{t\Delta_D}\mathbf 1\ne\mathbf 1$ for $t>0$. If a nonzero constant is
nevertheless admitted as an $L^2$ initial datum, the Dirichlet semigroup
immediately produces a spatially nonconstant state. Thus the PDE does not
reduce to \eqref{eq:appendix-homogeneous-ode}; the zero state is the only
constant exception.
\end{remark}

\section*{Acknowledgements}
During the preparation of this work, we used ChatGPT (OpenAI) to assist with language polishing. The authors reviewed all suggested edits and take full responsibility for the published content.

\bibliographystyle{alpha}
\bibliography{sta}

\end{document}